\documentclass[preprint,12pt]{elsarticle}

\usepackage[utf8]{inputenc}
\usepackage[english]{babel}
\usepackage{amsmath,amssymb,amsthm,latexsym}
\usepackage{bm}
\usepackage{float}
\usepackage{longtable}
\usepackage{multicol}
\usepackage{xcolor}
\usepackage[mathscr]{eucal}
\usepackage[all]{xy}
\xyoption{poly}
\xyoption{2cell}
\usepackage{tikz,pgf}
\usepackage{tikz-cd}
\usepackage{tkz-euclide}
\usetikzlibrary{matrix,shapes.geometric,calc,patterns}
\usepackage{adjustbox}
\usepackage{stackrel}
\usepackage{extarrows}
\usepackage{rotating}
\usepackage{hyperref}

\journal{Journal of Algebra}
\newif\ifarxiv
\ifdefined\versionarxiv\arxivtrue\fi

\newtheorem{theorem}{Theorem}[section]
\newtheorem{lemma}[theorem]{Lemma}
\newtheorem{proposition}[theorem]{Proposition}
\newtheorem{corollary}[theorem]{Corollary}
\theoremstyle{definition}
\newtheorem{definition}[theorem]{Definition}
\newtheorem{example}[theorem]{Example}

\theoremstyle{remark}
\newtheorem{remark}[theorem]{Remark}
\numberwithin{equation}{section}

\DeclareMathOperator{\prin}{prin}

\DeclareMathOperator{\rep}{rep}

\DeclareMathOperator{\dm}{\underline{cdim}}

\DeclareMathOperator{\md}{mod}

\DeclareMathOperator{\Hom}{Hom}
\DeclareMathOperator{\Supp}{supp}
\DeclareMathOperator{\csupp}{csupp}
\def\P{\mathscr{P}}

\def\DD{\mathbb{D}}
\def\AA{\mathbb{A}}
\newcommand{\kO}{\Bbbk}
\newcommand{\kT}{\Bbbk^{2}}
\newcommand{\Ext}{\operatorname{Ext}\nolimits}

\newcommand{\ze}{\epsilon}

\newcommand\colorb[1]{\color{blue}{\bm #1}}

\newcommand\smrud[1]{%
  \begingroup\setlength\arraycolsep{0pt}\renewcommand{\arraystretch}{0.5}\Huge\ensuremath{\begin{matrix}  #1\end{matrix}\arrow[rd]\arrow[ru]\endgroup}}

\newcommand\sm[1]{%
  \begingroup\setlength\arraycolsep{0pt}\renewcommand{\arraystretch}{0.5}\Huge\ensuremath{\begin{matrix}#1\end{matrix}\endgroup}}

\newcommand\smru[1]{%
  \begingroup\setlength\arraycolsep{0pt}\renewcommand{\arraystretch}{0.5}\Huge\ensuremath{\begin{matrix}#1\end{matrix}\arrow[ru]\endgroup}}

\newcommand\smrd[1]{%
  \begingroup\setlength\arraycolsep{0pt}\renewcommand{\arraystretch}{0.5}\Huge\ensuremath{\begin{matrix}#1\end{matrix}\arrow[rd]\endgroup}}

\newcommand\smr[1]{%
  \begingroup\setlength\arraycolsep{0pt}\renewcommand{\arraystretch}{0.4}\Huge\ensuremath{\begin{matrix}#1\end{matrix}\arrow[r]\endgroup}}

\begin{document}

\begin{frontmatter}

\title{A geometric realization of socle-projective categories for posets of type $\mathbb{D}$}

\author[ud]{Gabriel Bravo Rios}
\ead{gbravor@udistrital.edu.co}

\author[uconn]{Ralf Schiffler\corref{cor1}}
\ead{schiffler@math.uconn.edu}
\cortext[cor1]{Corresponding author.}

\author[uptc]{Robinson-Julian Serna}
\ead{robinson.serna@uptc.edu.co}

\affiliation[ud]{organization={Facultad de Ciencias Matem\'aticas y Naturales, Universidad Distrital Francisco Jos\'e de Caldas},
            city={Bogot\'a},
            postcode={110110},
            country={Colombia}}

\affiliation[uconn]{organization={Department of Mathematics, University of Connecticut},
            city={Storrs},
            postcode={06269-3009},
            state={CT},
            country={USA}}

\affiliation[uptc]{organization={Escuela de Matem\'aticas y Estad\'istica, Universidad Pedag\'ogica y Tecnol\'ogica de Colombia},
            city={Tunja},
            postcode={150001},
            state={Boyac\'a},
            country={Colombia}}

\begin{abstract}
We introduce posets of type $\DD$, the posets not of type $\AA$ whose category of socle-projective modules embeds fully faithfully, by restriction, into the category of representations of a Dynkin quiver of type $\DD_n$, and we describe them by an admissible Dynkin quiver of type $\DD_n$ together with a compatible set of extra arrows (alien arrows). Continuing a program initiated for posets of type $\AA$ [R. Schiffler, R-J. Serna, \textit{A geometric realization of socle-projective categories for posets of type $\AA$}, J. Pure Appl. Algebra 224 (2020), no. 12, 106436], we characterize these posets combinatorially and by means of alien arrows; in particular, the category of socle-projective $\Bbbk\P$-modules is of finite representation type. Building on Schiffler's geometric model for cluster categories of type $\DD_n$, our main result establishes a $\Bbbk$-linear equivalence $\Theta$ between the full subcategory $(\mathcal C/T)_F$ of the arc category of a punctured $n$-gon whose objects are direct sums of certain sp-arcs and the category of finitely generated socle-projective $\Bbbk\P$-modules, for $\P$ a poset of type $\DD$. We give two applications. First, when there are no alien arrows, the subalgebra of the cluster algebra of type $\DD_n$ generated by the initial cluster and the cluster variables of the sp-arcs is the whole cluster algebra, whereas this may fail in the presence of alien arrows. Second, every total central charge on the module category of the underlying Dynkin quiver, which exists by a theorem of Chang, Qiu and Zhang, remains total on the category of socle-projective $\Bbbk\P$-modules.
\end{abstract}


\end{frontmatter}

\section{Introduction}

Several authors have developed geometric realizations of representation categories, in which objects and morphisms are described through the combinatorial geometry of surfaces. This includes models for cluster categories of surfaces without punctures \cite{BGMS,brustle,BMR,BMRRT,schiffler1,FST,JML}, for $m$-cluster categories of type $\DD_n$ and beyond \cite{karin1,karin2}, for categories of type $\tilde{\AA}$ and morphisms in the infinite radical \cite{karin3}, and for module categories of gentle algebras \cite{karin4}. In particular, \cite{Schiffler} establishes a geometric model for cluster categories of type $\DD$, based on an $n$-gon with a puncture at its center. The model uses the category of arcs of the polygon, whose objects are direct sums of arcs, that is, of homotopy classes of paths between two vertices, and whose morphisms are given by elementary moves and mesh relations over an algebraically closed field $\Bbbk$.

From a different perspective, the theory of representations of partially ordered sets (posets) was developed by Nazarova, Roiter, and their students as an algebraic tool for addressing the second Brauer–Thrall conjecture. The central aim of this theory is to describe the indecomposable objects and the irreducible morphisms in the category of all matrix representations of a poset over $\Bbbk$ \cite{nazarova,GabrielRoiter}. In terms of vector spaces, a representation of a poset $\mathscr{P}$ is a finite-dimensional $\Bbbk$-vector space $V$ together with a family of subspaces $V(s)$, $s\in\mathscr{P}$, such that $V(s)\subseteq V(t)$ whenever $s\preceq t$ \cite{Gabriel73}. The category of these representations is identified with a full subcategory of the category of modules over the incidence algebra $\Bbbk\mathscr{P}^\star$, where $\mathscr{P}^\star$ is obtained from $\mathscr{P}$ by adding a maximum \cite[Section 2.3]{GabrielRoiter}.  In the early 1990s, Simson studied the category of peak spaces of a poset $\mathscr{P}$ with $r$ maximal elements (an $r$-peak poset), which coincides with the category of finitely generated socle-projective modules over the incidence algebra $\Bbbk \mathscr{P}$ \cite{simson,simson3}. As part of this classification, Kosakowska described the sincere $r$-peak posets of finite representation type and their sincere representations \cite{justina,justina1,justina2}.

As a connection between the  geometric models and representation theory of posets, the second and third authors introduced a family of posets, called posets of type $\mathbb{A}$, and described a geometric realization through a subcategory of the category $\mathcal{C}_{T}$ of diagonals of a regular $(n+3)$-gon associated with the triangulation $T$. This subcategory is categorically equivalent to the category of socle-projective representations of posets of type $\mathbb{A}$. This model was also used to define certain subalgebras of the cluster algebra of type $\mathbb{A}$ and to determine when they coincide with the entire cluster algebra \cite{schifflerserna}. On the other hand, inspired by \cite{BGMS}, Iusenko, Bravo Rios and Serna gave in \cite{IusenkoBravoRiosSerna2025} another geometric realization of the category of socle-projective representations of posets of type $\AA$, and used it to prove that every indecomposable object is stable with respect to a weight defined from the geometric model.

In this paper we introduce posets of type $\DD$ (Definition \ref{defposetypeD}). A finite connected poset $\P$ which is not of type $\AA$ is of type $\DD$ if there is a Dynkin quiver $Q$ of type $\DD_n$ on the points of $\P$, whose arrows are relations of $\P$, such that the restriction functor from socle-projective $\Bbbk\P$-modules to socle-projective representations of $Q$ is well defined and fully faithful. Our first result (Theorem \ref{thm:charD}) gives two equivalent descriptions of these posets. The first one is combinatorial: the maximal points of $\P$ are the sinks of $Q$, and every relation $x\preceq z$ with $z$ maximal is realized by a path of $Q$. The second one is explicit: $Q$ is admissible, and $\P$ is the poset associated with a quiver $Q^F$, obtained from $Q$ by adding a set $F$ of alien arrows that is compatible with the two $\AA$-subquivers of $Q$. Admissibility is a condition on the orientation of $Q$ at its branch vertex, which must be a sink or a source of $Q$, receive exactly one arrow, or receive the arrows from both leaves and point to the remaining neighbour. In particular, the category of socle-projective $\Bbbk\P$-modules is of finite representation type (Corollary \ref{finiterepresentation}), hence so is the category of prinjective $\Bbbk\P$-modules (Corollary \ref{finiteprin}), and the dimension vector of every indecomposable object is a positive root of $Q$ (Corollary \ref{suppoftypeD}).

Our main result (Theorem \ref{main}) is a geometric realization of this category. Let $T$ be the triangulation of the punctured $n$-gon associated with $Q$ in Schiffler's model \cite{Schiffler}. We single out a class of arcs, the sp-arcs (Definition \ref{sparc}), by conditions on their crossings with the arcs of $T$, expressed through the fans of $T$ and the alien arrows, and we show that the full subcategory $(\mathcal C/T)_F$ of $\mathcal C/T$ whose objects are the direct sums of sp-arcs is equivalent to the category of finitely generated socle-projective $\Bbbk\P$-modules. As a consequence, the dimensions of morphism spaces between indecomposable socle-projective $\Bbbk\P$-modules can be read off from crossing numbers (Corollary \ref{cor:hom}).

We give two applications (Section \ref{sec:applications}). First, following \cite[Theorem 5.2]{schifflerserna}, we show that when there are no alien arrows the subalgebra of the cluster algebra of type $\DD_n$ generated by the initial cluster and the cluster variables of the sp-arcs is the whole cluster algebra (Corollary \ref{clusterprop}), whereas this fails in general in the presence of alien arrows (Remark \ref{rem:Fnonempty}). Second, we show that every total central charge on $\md\Bbbk Q$, which exists by a theorem of Chang, Qiu and Zhang \cite{CQZ} answering a conjecture of Reineke \cite{Reineke}, is also total on $\md_{sp}(\Bbbk\P)$ (Theorem \ref{thm:total}). The proof uses that the restriction of $\Theta(\gamma)$ to $Q$ is $\varphi_T(\gamma)$ (Theorem \ref{main}). This complements \cite{IusenkoBravoRiosSerna2025}, where the geometric model of type $\AA$ was used to prove that every indecomposable object is stable with respect to a geometrically defined weight.

This paper is organized as follows. In Section \ref{section2}, we introduce the notation and basic definitions used throughout the paper, such as the arc category with a puncture, the category of socle-projective modules over a poset, and posets of type $\AA$. In Section \ref{posetstypeD}, we define posets of type $\DD$ and show that the restriction to the underlying Dynkin quiver is fully faithful. In Section \ref{equi}, we prove the main result of this paper, the categorical equivalence of Theorem \ref{main}. In Section \ref{sec:applications}, we give two applications, to cluster subalgebras (Section \ref{sec:cluster}) and to total stability (Section \ref{sec:stability}).
\ifarxiv
Finally, in \ref{sincereposetsD} we provide the Hasse diagrams of the sincere posets of type $\DD$, named following the numbering and order of \cite[Table 2.2]{justina2} (no sincere poset of type $\DD$ is associated with a quiver whose branch vertex receives the arrows from both leaves, see Remark \ref{rem:D4notsincere}), and in \ref{sincererepsD} we list their corresponding sincere indecomposable socle-projective representations.
\fi

\section{Preliminaries} \label{section2}

Throughout the paper, $\Bbbk$ denotes an algebraically closed field.

\subsection{Category of arcs of a punctured regular $n$-gon}\label{subsec:arcs}
We recall the geometric model of \cite{Schiffler} for the cluster category of type $\DD_n$.

Let $\Pi_n$ be a regular polygon with $n$ vertices and one puncture in the
center.  Let $\delta_{a,b}$ be an oriented path between two vertices $a \neq b$ on the boundary of $\Pi_n$ in counterclockwise direction, such that
$\delta_{a,b}$ does not run through the same point twice. Also let
$\delta_{a,a}$ be the path that runs from $a$ to $a$, i.e. around
the polygon exactly one time. We define $|\delta_{a,b}|$ to be the
number of vertices on the path $\delta_{a,b}$, including $a$ and $b$.

An \textit{edge} is a triple $(a,\gamma,b)$ where $a$ and $b$ are vertices
on the boundary of the polygon and $\gamma$ is an oriented path from $a$
to $b$ lying in the interior of $\Pi_n$ and that is homotopic
to $\delta_{a,b}$. The path should not cross itself and
$|\delta_{a,b}| \geq 3$.  Two edges are equivalent if they start in the
same vertex, end in the same vertex and are homotopic.

Let $E$ be the set of equivalence classes of edges, and denote by
$M_{a,b}$ the equivalence class of edges in $E$ going from $a$ to
$b$. The set of tagged edges is defined as follows.  
$$\{ M_{a,b}^{\epsilon}|M_{a,b} \in E\text{, } \epsilon \in \{-1,1\} \text{
  with }
\epsilon = 1 \text{ if } a \neq b \}$$
From now on tagged edges will be called \textit{arcs}. Arcs
starting and ending in the same vertex $a$ will be represented as
lines between the puncture and the vertex $a$. Arcs with
$\epsilon = -1$ will be drawn with a tag on it. In some cases we will
draw them as loops.  For simplicity, sometimes we use Greek letters instead of the notation \( M_{a,b}^{\epsilon} \) if the context is clear. The crossing number $e(M_{a,b}^\epsilon,N_{c,d}^{\epsilon'})$ is the minimal number of intersections of representations of
$M_{a,b}^\epsilon$ and $N_{c,d}^{\epsilon'}$ in the interior of the punctured polygon. When $a=b$ and $c=d$, we let the crossing
number be $1$ if $a \neq c$ and $\epsilon \neq \epsilon'$ and $0$
otherwise. If $e(M_{a,b}^\epsilon,N_{c,d}^{\epsilon'}) = 0$, we say
that $M_{a,b}^\epsilon$ and $N_{c,d}^{\epsilon'}$ do not cross.

The second author defined a category $\mathcal{C}$ which is equivalent to the cluster
category of Dynkin type $\mathbb{D}_n$ \cite{Schiffler}. The objects in $\mathcal{C}$ are
direct sums of arcs (tagged edges), and the morphism space from
$\alpha$ to $\beta$ is spanned by sequences of \textit{elementary moves}, which move one endpoint of an arc to its counterclockwise neighbour (with the convention of \cite{Schiffler} for the arcs at the puncture), modulo the so-called mesh relations.

A \textit{triangulation} of \( \Pi_n \) is a maximal set of non-crossing arcs. Any such set has \( n \) elements.
Let $T$ be a triangulation of
\( \Pi_n \), we can assign to $T$ a quiver $Q_T$ in the following way. The vertices are the
midpoints of the arcs. There is an arrow between $i$ and $j$ if
the corresponding arcs bound a common triangle. The orientation
is $i \rightarrow j$ if the arc corresponding to $j$ can be
obtained from the arc corresponding to $i$ by rotating clockwise about their common vertex. In the case when there are two arcs $\alpha$ and $\alpha'$ between the puncture and the same vertex $v$ on the boundary, we think of each of them as a loop at $v$ around the puncture \cite{FST}. Thus each of them is a side of the triangle containing that loop, considered separately for $\alpha$ and for $\alpha'$, and there are no arrows between $\alpha$ and $\alpha'$. For example, if both are adjacent to an arc $\beta$ and a boundary edge $\delta$, we consider the triangle with edges $\alpha$, $\beta$ and $\delta$ separately from the triangle with edges $\alpha'$, $\beta$ and $\delta$. If this produces an oriented cycle of length $2$, we delete both of its arrows. See some examples in Figure \ref{figtriquiv}.

\begin{figure}[H]
\centering

\begin{tikzpicture}
\node (a) at (0,0) {$\begin{array}{cc}
\xy/r4.5pc/: {\xypolygon5"A"{~<{}~>{-}{\scriptstyle\bullet}}},
*+{\scriptstyle\bullet},
\POS"A3"\drop{\begin{array}{c}   \\ \\ \end{array}  }
\POS"A4"\drop{\begin{array}{c}  \end{array}  }
\POS"A1"\drop{\begin{array}{c} \end{array}  }  
\POS"A2"\drop{\begin{array}{c}\\ \\  \end{array}  }  
\POS"A5"\drop{\begin{array}{c}\qquad \end{array}  } 
\POS"A4" \ar|-(0.5){\SelectTips{cm}{}\object@{+}}@{-} @/^1.3ex/ "A0"
\POS"A3" \ar@{-} @/^13ex/ "A4"
\POS"A4" \ar@{-} @/_0.8ex/  "A0"
\POS"A3" \ar@{-} @/^8.5ex/  "A5",
\POS"A3" \ar@{-} @/^3.3ex/ "A1"
\endxy 
&
\xy/r4.5pc/: {\xypolygon5"A"{~<{}~>{-}{\scriptstyle\bullet}}},
*+{\scriptstyle\bullet},
\POS"A3"\drop{\begin{array}{c}  \\ \\ \end{array}  } 
\POS"A4"\drop{\begin{array}{c}  \end{array}  }
\POS"A1"\drop{\begin{array}{c} \end{array}  }  
\POS"A2"\drop{\begin{array}{c}\\ \\  \end{array}  }  
\POS"A5"\drop{\begin{array}{c}\qquad \end{array}  } 
\POS"A3" \ar@{-}  "A0",
\POS"A4" \ar@{-}  "A0"
\POS"A3" \ar@{-} @/^13ex/ "A4",
\POS"A3" \ar@{-} @/^8.5ex/  "A5",
\POS"A5" \ar@{-} @/_3.3ex/ "A2",
\endxy 
\end{array}$};
\node (b) at (-2.9,0.27) [scale=0.6, text=blue] {\Large{$\bullet$}};
\node (c) at (-2.65,-0.77) [scale=0.6, text=blue] {\Large{$\bullet$}};
\node (d) at (-2.22,-0.65) [scale=0.6, text=blue] {\Large{$\bullet$}};
\node (e) at (-2.23,0.53) [scale=0.6, text=blue] {\Large{$\bullet$}};
\node (e) at (-0.9,0.56) [scale=0.6, text=blue] {\Large{$\bullet$}};
\node (f) at (1,0.07) [scale=0.6, text=blue] {\Large{$\bullet$}};
\node (g) at (1.44,-1.26) [scale=0.6, text=blue] {\Large{$\bullet$}};
\node (h) at (2.56,-0.16) [scale=0.6, text=blue] {\Large{$\bullet$}};
\node (i) at (3.09,-0.5) [scale=0.6, text=blue] {\Large{$\bullet$}};
\node (j) at (2.95,0.6) [scale=0.6, text=blue] {\Large{$\bullet$}};
\draw[->, >=angle 45 ,blue](-2.65,-0.77) -- (-2.92,0.18); 
\draw[->, >=angle 45 ,blue](-2.22,-0.65) -- (-2.84,0.22); 
\draw[->, >=angle 45 ,blue](-2.23,0.53) -- (-2.8,0.33); 
\draw[->, >=angle 45 ,blue](-0.9,0.56) -- (-2.12,0.56); 
\draw[->, >=angle 45 ,blue](2.56,-0.16) -- (1.11,0.13); 
\draw[->, >=angle 45 ,blue](1.44,-1.28) -- (2.5,-0.22);
\draw[->, >=angle 45 ,blue](3.09,-0.5) -- (2.65,-0.2); 
\draw[->, >=angle 45 ,blue](3.1,-0.5) -- (2.95,0.5);
\end{tikzpicture}
\caption{Triangulations and corresponding quiver.}\label{figtriquiv}
\end{figure}


\subsection{Category of socle-projective modules of a poset}
In this subsection we recall some concepts on posets and on socle-projective modules over incidence algebras of posets, in the sense of \cite{simson3}. We use the notation from \cite{schifflerserna}.

We denote by $(\P,\preceq)$ a finite  partially ordered set (poset, for short) with respect to the partial order $\preceq$. We write $x\prec y$ if $x\preceq y$ and $x\neq y$. For simplicity we write $\P$ instead of $(\P,\preceq)$. Let $\max\P$ (respectively, $\min\P$) be the set of all maximal  (respectively, minimal) points of $\P$. A poset $\P$ is called an \textit{$r$-peak poset} if $\left| \max\P\right|=r$.  A full subposet $\P'$ of $\P$  is said to be a \textit{peak-subposet} if $\max\P'\subseteq\max\P$.  $\P$ is \textit{connected} if its Hasse quiver is connected.

Let $\P$ be a poset and let $x, y \in \P$. We say that $y$ covers $x$ if $x \prec y$ and there is no element $z \in \P$ such that $x \prec z \prec y$.
The Hasse quiver $Q$ of $\P$ is the quiver whose vertices correspond to the elements of $\P$, and which contains an arrow $\alpha: x \to y$ for each pair $x, y \in \P$ such that $y$ covers $x$.
In the diagrams representing the Hasse quiver of a poset $\P$, the maximal elements (or peaks) of $\P$ are placed at the bottom, and each of them is denoted by the symbol~$\star$.
 
Given a finite poset $\P$,  we denote by $\Bbbk\P$ its \textit{incidence algebra}, which can be described as a bound quiver algebra $\Bbbk Q/I$, where $Q$ is the Hasse quiver of $\P$ and $I$ is the ideal of $\Bbbk Q$ generated by all the commutativity relations $\gamma-\gamma'$  with  $\gamma$ and $\gamma'$ parallel paths in $Q$. In this case, the category $\md(\Bbbk\P)$ of the finitely generated $\Bbbk\P$-modules is identified with the well known category $\rep (Q,I)$ of representations  of the bound quiver $(Q,I)$.

The \textit{socle} $\text{soc }M$ of a module $M$ is the semisimple submodule generated by  all simple submodules of $M$. A module $M$ is called \textit{socle-projective} if $\text{soc }M$ is a projective module. The full subcategory of  $\md (\Bbbk\P)$  whose objects are the socle-projective $\Bbbk\P$-modules  is denoted by $\md_{sp}(\Bbbk\P)$. The following proposition describes the objects of $\md_{sp}(\Bbbk\P)$ explicitly.

\begin{proposition}\cite[Section 3]{simson5} \label{conditions,a,b} Each $\Bbbk\P$-module $M$ in $\md(\Bbbk\mathscr{P})$ is identified with  a collection $M=(M_x,{_y}h_x)_{x,y\in\P}$  of finite-dimensional $\Bbbk$-vector spaces $M_x$, one for each point $x\in\P$, and a collection of $\Bbbk$-linear maps ${_y}h_x: M_x \to M_y$, one for each relation $x\preceq y$ in $\P$, such that

\begin{itemize}
\item[(a)]${_x}h_x$ is the identity of $M_x$  for all $x\in\P$ and ${_w}h_y\cdot {_y}h_x= {_w}h_x$ for all $x\preceq y\preceq w$ in $\P$.
\end{itemize}
Furthermore, $M=(M_x,{_y}h_x)_{x,y\in\P}$ is a socle-projective module if  it also holds that  
\begin{itemize}
\item[(b)] For all $x\in\P\setminus\max\P$, the $\Bbbk$-subspace $$I_x=\stackbin[\begin{smallmatrix}  z\in\max\P \\ z\succ x\end{smallmatrix}]{}{\bigcap}\ker {_z}h_x$$ of $M_x$ is the zero subspace.
\end{itemize}
\end{proposition} 

A \textit{morphism} $f:M\to N$ in $\md(\Bbbk\P)$ is a collection $f=(f_x)_{x\in\P}$ of linear maps $f_x:M_x\to N_x$ such that $f_y\circ{_y}h_x={_y}h'_x\circ f_x$ for each relation $x\preceq y$ in $\P$.

\ifarxiv
Given a poset $\P$ and an object $M=(M_x,{_y}h_x)_{x,y\in\P}$ in $\md(\Bbbk\P)$,  the \textit{Jacobson radical}  of $M$ is the $\Bbbk\P$-module given by  $\text{rad }M=((\text{rad }M)_x, {_y}\overline{h}_x)_{x,y\in\P}$, where  $(\text{rad }M)_x=\sum_{a\prec x}\text{Im}({_x}h_a)$ and ${_y}\overline{h}_x$ is the restriction of ${_y}h_x$ to $(\text{rad }M)_x$ for each $x\preceq y$ in $\P$.  The \textit{top} of $M$ is the semisimple quotient $\text{top }M=M\slash\text{rad }M$, with $(\text{top }M)_x=M_x\slash(\text{rad }M)_x$. Moreover,  the  \textit{coordinate vector} of  $M$ is given by the vector $$d = \dm M = (d_x)_{x\in\P}\in\mathbb{N}^{\P}$$ such that $d_x= \dim_{\Bbbk} M_x$ if $x\in\max\P$ and   $d_x=\dim_{\Bbbk}(\text{top }M)_x$ otherwise. If  $M$ is an indecomposable socle-projective $\Bbbk\P$-module, the \textit{coordinate support} $$\csupp M =\lbrace x\in\P\hspace{0.1cm}\vert\hspace{0.1cm} (\dm M)_x\neq 0\rbrace$$ of $M$ is a  peak-subposet  of $\P$. If $\csupp M=\P$, $M$ is called a \textit{sincere socle-projective} $\Bbbk\P$-module. If there exists a sincere socle-projective $\Bbbk\P$-module, we say that $\P$ is  a \textit{sincere poset}. We also write $\underline{\dim}\,M=(\dim_{\Bbbk}M_x)_{x\in\P}$ for the dimension vector of $M$, which in general differs from $\dm M$.
\else
We write $\underline{\dim}\,M=(\dim_{\Bbbk}M_x)_{x\in\P}$ for the dimension vector of a $\Bbbk\P$-module $M$.
\fi

\ifarxiv
\begin{example}
The poset $\P = \{\, 1 < 2,\; 2 < 3,\; 2 < 4\}$ has incidence algebra given by the quiver $Q$ on the left. The module $M$ on the right
\[
\begin{tikzpicture}[baseline=(current bounding box.center),->,>=stealth, node distance=1.8cm]
\node (K0) at (-0.6,0){$Q=$};
\node (1) at (0,0) {1};
\node (2) at (1.5,0) {2};
\node (3) at (3,0.8) {3};
\node (4) at (3,-0.8) {4};
\draw (1) -- (2);
\draw (2) -- (3);
\draw (2) -- (4);
\end{tikzpicture}
\qquad\qquad
\begin{tikzpicture}[baseline=(current bounding box.center),->,>=stealth,node distance=2cm]
\node (K0) at (-0.6,0){$M =$};
\node (K1) at (0,0) {$\Bbbk$};
\node (K2) at (1.5,0) {$\Bbbk^2$};
\node (K3) at (3,0.8) {$\Bbbk$};
\node (K4) at (3,-0.8) {$\Bbbk$};
\draw (K1) -- node[above] {{\scriptsize$\left[\begin{smallmatrix} 1\\0 \end{smallmatrix}\right]$}} (K2);
\draw (K2) -- node[pos=0.35,above, yshift=2pt] {{\scriptsize$\left[\begin{smallmatrix} 1 &\hspace{-7pt}0 \end{smallmatrix}\right]$}} (K3);
\draw (K2) -- node[pos=0.35,below, yshift=-2pt] {{\scriptsize$\left[\begin{smallmatrix} 1 &\hspace{-7pt}1 \end{smallmatrix}\right]$}} (K4);
\end{tikzpicture}
\]
is indecomposable, since the image of $M_1$ and the kernels of the maps $M_2\to M_3$ and $M_2\to M_4$ are three distinct lines in $M_2=\Bbbk^2$, and it is socle-projective, since $\operatorname{soc}M=S(3)\oplus S(4)$ is projective. Since $\operatorname{top}M=S(1)\oplus S(2)$, we have $\dm M=(1,1,1,1)$, so $M$ is sincere. Therefore, $\P$ is a sincere poset.
\end{example}
\fi

\ifarxiv
Kleiner \cite{kleiner75} and Kosakowska \cite{justina,justina1,justina2} classified the sincere $r$-peak posets of finite representation type and their sincere socle-projective representations; for posets of type $\DD$ these lists are reproduced in \ref{sincereposetsD} and \ref{sincererepsD}.
\fi

\subsection{Posets of type $\mathbb{A}$} \label{subsectionposetsof typeA} 
We recall some results and notation on posets of type $\mathbb{A}$ from \cite{schifflerserna}.

\begin{definition} \cite[Definition 3.1]{schifflerserna} \label{defposetypeA}
A finite connected poset $\P$ is said to be a \textit{poset of type $\mathbb{A}$} if $\P$ does not contain as a peak-subposet any of the following posets, where the peaks 
are represented by stars:\vspace{0.2cm}
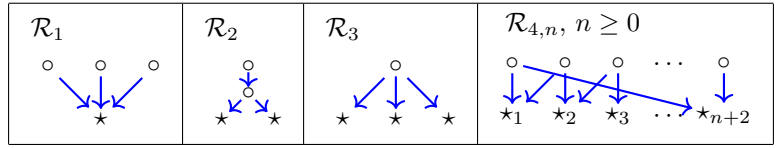
\begin{center}

\begin{tabular}{|l|l|l|l|}
\hline 
\begin{tikzpicture}
\node (top) at (-0.7,0.5) {$\mathcal{R}_1$};

\node (top) at (0,-0.7) {$\star$};
\node [above of=top, node distance=0.7cm](center) {$\circ$};
\node [left  of=center, node distance=0.7cm] (left)  {$\circ$};
\node [right of=center, node distance=0.7cm] (right) {$\circ$};

    \draw [thick, <-, shorten <=-2pt, shorten >=-2pt] (top) -> (left);
    \draw [thick, <-, shorten <=-2pt, shorten >=-2pt] (top) -> (right);
    \draw [thick, <-, shorten <=-2pt, shorten >=-2pt] (top) -> (center);\end{tikzpicture} &
    
    \begin{tikzpicture} 
    \node (top) at (0,0.5) {$\mathcal{R}_2$};
    \node (top1) at (0,-0.7) {$\star$};
    \node [right of=top1, node distance=0.7cm](top2)  {$\star$};
    \node (m) at (0.35,-0.35) {$\circ$};
    \node (bellow) at (0.35,0) {$\circ$};
     \draw [thick,->, shorten <=-2pt, shorten >=-2pt] (top1) -- (m);
    \draw [thick,->, shorten <=-2pt, shorten >=-2pt] (top2) -- (m);
    \draw [thick,->, shorten <=-2pt, shorten >=-2pt] (m) -- (bellow);
    
   \end{tikzpicture}&
   
    \begin{tikzpicture} 
\node (top) at (0,0.5) {$\mathcal{R}_3$};   
   \node (top3) at (0,-0.7){$\star$};    
    \node [right of=top3, node distance=0.7cm](top4) {$\star$};
        \node [right of=top4, node distance=0.7cm] (top5) {$\star$};
        \node [above of=top4, node distance=0.7cm] (be) {$\circ$};
 \draw [thick,->, shorten <=-2pt, shorten >=-2pt] (be) -- (top3);
 \draw [thick, ->, shorten <=-2pt, shorten >=-2pt] (be) -- (top4);
 \draw [thick, ->,shorten <=-2pt, shorten >=-2pt] (be) -- (top5);\end{tikzpicture} & 
 
 \begin{tikzpicture}
 \node (1) at (0,-0.7) {$\star_1$};
\node (A) at (0.8,0.5) {$\mathcal{R}_{4,n}$, $n\geq 0$};
\node [right of=1, node distance=0.7cm](2) {$\star_2$};

\node [right of=2, node distance=0.7cm](3) {$\star_3$};
\node [right of=3, node distance=0.7cm] (4)  {$\cdots$};
\node [right of=4, node distance=0.7cm] (5) {$\star_{n+2}$};

\node [above of=1, node distance=0.7cm] (1')  {$\circ$};
\node [above  of=2, node distance=0.7cm] (2')  {$\circ$};
\node [above  of=3, node distance=0.7cm] (3')  {$\circ$};
\node [above  of=4, node distance=0.7cm] (4')  {$\dots$};
\node [above  of=5, node distance=0.7cm] (5')  {$\circ$};

   \draw [thick, <-, shorten <=-2pt, shorten >=-2pt] (1) -- (1');
    \draw [thick, ->, shorten <=-2pt, shorten >=-2pt] (2') -- (2);
    
\draw [thick,<-, shorten <=-2pt, shorten >=-2pt] (3) -- (3');
    \draw [thick,<-, shorten <=-2pt, shorten >=-2pt] (5) -- (5');    
    
    \draw [thick, <-, shorten <=-2pt, shorten >=-2pt] (1) -- (2');
    
 \draw [thick,<-, shorten <=-2pt, shorten >=-2pt] (2) -- (3');
     \draw [thick, ->,shorten <=-2pt, shorten >=-2pt] (1') -- (5);
\end{tikzpicture} \\ 
\hline 
\end{tabular} 

\end{center}
\vspace{0.2cm}
\end{definition}

Posets of type $\mathbb{A}$ can also be viewed as posets associated to certain quivers which are obtained from Dynkin quivers of type $\mathbb{A}$  by adding some new arrows. To explain this we need some definitions.

All quivers are finite, connected and acyclic. The \textit{underlying graph} of a quiver $Q$ is denoted by $\overline{Q}$, and $Q$ is a \textit{tree quiver} if $\overline{Q}$ is a tree; Dynkin quivers of types $\AA$ and $\DD$ are tree quivers. The \textit{valence} of a vertex is its number of neighbours in $\overline{Q}$. A vertex of valence one is an \textit{end vertex}, and a vertex of valence at least three is a \textit{branch vertex}. A vertex $x$ of $Q$ is a \textit{sink} (respectively, a \textit{source}) if no arrow $\alpha$ has $s(\alpha)=x$ (respectively, $t(\alpha)=x$). In a tree quiver there is at most one path between two vertices, so there are no parallel paths. For a vertex $x$ we write
\[
U(x)=\{z \text{ sink of } Q:\ \text{there is a path from } x \text{ to } z \text{ in } Q\},
\]
so that $x\in\operatorname{Supp}I(z)$ if and only if $z\in U(x)$, where $I(z)$ is the indecomposable injective representation at the sink $z$. We say that $y$ is \textit{reachable} from $x$, or that $x$ reaches $y$, if there is a path from $x$ to $y$ in $Q$. Thus $U(x)$ is the set of sinks reachable from $x$.

Let $Q$ be an acyclic quiver and let $\P_Q=Q_0$ be its set of vertices. We define an order on $\P_Q$ by $x\preceq y$ if and only if there exists a path from $x$ to $y$ in $Q$. We say that $\P_Q$ is the \textit{poset associated to the quiver} $Q$.

Given a  Dynkin quiver $Q$  of type $\mathbb{A}$, its  underlying graph $\overline{Q}$ has the form  
\begin{center}
\begin{tikzcd}[arrows=-,row sep= tiny, column sep = normal]
 \underset{_1}{\circ} \arrow[r] & \underset{_2}{\circ}\arrow[r,dotted, thick] & \underset{_{n-1}}{\circ} \arrow[r] &
  \underset{_{n}}{\circ},  
 \end{tikzcd}
\end{center}
and its end vertices $1$ and $n$ are called the \textit{extreme vertices} of $Q$.

\begin{definition}\cite[Definition 3.5]{schifflerserna}\label{aliensetDef}
A set $F=\{\alpha_1,\dots,\alpha_t\}$ of new arrows for $Q$ is called an \textit{alien set} for $Q$ if the following conditions  hold.

\begin{itemize}
\item[(a)] For all $\alpha\in F$, there exists a sink vertex $z$ in $Q$ such that $s(\alpha),t(\alpha)\in \Supp I(z)$.

\item[(b)] For all $\alpha\in F$, $t(\alpha)$  is not a source vertex in $Q$ unless it is an extreme vertex in $Q$. 
\item [(c)]  For all $\alpha\in F$,  the  arrow $\alpha$ is the unique path from $s(\alpha)$ to  $t(\alpha)$ in $Q^F$, where $Q^F$  is  the quiver such that $Q^F_0=Q_0$ and $Q^F_1=Q_1\cup F$. 
 
\item[(d)] The quiver $Q^F$ is acyclic.
\end{itemize}
The arrows in an alien set for $Q$ will be called \textit{alien arrows}.
\end{definition}

 \begin{proposition} \cite[Proposition 3.8]{schifflerserna} \label{0} A poset $\P$ is of type $\mathbb{A}$ if and only if there exists a Dynkin quiver $Q$ of type $\mathbb{A}$ and an  alien set $F$ for $Q$ such that $\P=\P_{Q^F}$ is the  poset associated to the quiver $Q^F$.
\end{proposition}

\ifarxiv
\begin{lemma} \cite[Lemma 3.9]{schifflerserna}\label{1}
Let $\P$ be a poset of type $\mathbb{A}$.  Then
\begin{itemize}
\item[(a)] $\md_{sp}\Bbbk\P$ is of finite representation type.
\item[(b)] $\P$  is a sincere poset if and only if $\P$ is isomorphic to one of the 
following posets:\vspace{0.2cm}

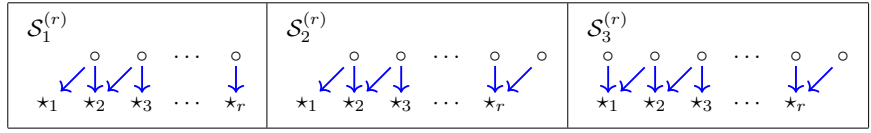
\begin{adjustbox}{max totalsize={0.9\textwidth}{0.9\textheight},center}
\begin{tabular}{|l|l|l|}
\hline 
\begin{tikzpicture}
 \node (1) at (0,0) {$\star_1$};
 
\node (A) at (0,1.15) {$\mathcal{S}^{(r)}_{1}$};
\node [right of=1, node distance=0.7cm](2) {$\star_2$};

\node [right of=2, node distance=0.7cm](3) {$\star_3$};
\node [right of=3, node distance=0.7cm] (4)  {$\cdots$};
\node [right of=4, node distance=0.7cm] (5) {$\star_{r}$};

\node [above  of=2, node distance=0.7cm] (2')  {$\circ$};
\node [above   of=3, node distance=0.7cm] (3')  {$\circ$};
\node [above   of=4, node distance=0.7cm] (4')  {$\dots$};
\node [above   of=5, node distance=0.7cm] (5')  {$\circ$};

   
    \draw [thick, ->, shorten <=-2pt, shorten >=-2pt] (2') -- (2);
    
\draw [thick, <-,shorten <=-2pt, shorten >=-2pt] (3) -- (3');
    \draw [thick,<-, shorten <=-2pt, shorten >=-2pt] (5) -- (5');    
    
    \draw [thick, <-, shorten <=-2pt, shorten >=-2pt] (1) -- (2');
    
 \draw [thick, <-, shorten <=-2pt, shorten >=-2pt] (2) -- (3');
    
\end{tikzpicture} & 

\begin{tikzpicture}
 \node (1) at (0,0) {$\star_1$};
 
\node (A) at (0,1.15) {$\mathcal{S}^{(r)}_{2}$};
\node [right of=1, node distance=0.7cm](2) {$\star_2$};

\node [right of=2, node distance=0.7cm](3) {$\star_3$};
\node [right of=3, node distance=0.7cm] (4)  {$\cdots$};
\node [right of=4, node distance=0.7cm] (5) {$\star_{r}$};

\node [above of=2, node distance=0.7cm] (2')  {$\circ$};
\node [above  of=3, node distance=0.7cm] (3')  {$\circ$};
\node [above  of=4, node distance=0.7cm] (4')  {$\dots$};
\node [above  of=5, node distance=0.7cm] (5')  {$\circ$};
\node [right of=5', node distance=0.7cm] (6')  {$\circ$};

   
    \draw [thick, ->, shorten <=-2pt, shorten >=-2pt] (2') -- (2);
    
\draw [thick, <-,shorten <=-2pt, shorten >=-2pt] (3) -- (3');
    \draw [thick, <-, shorten <=-2pt, shorten >=-2pt] (5) -- (5');    
    
    \draw [thick, <-, shorten <=-2pt, shorten >=-2pt] (1) -- (2');
    
 \draw [thick,<-, shorten <=-2pt, shorten >=-2pt] (2) -- (3');
    
\draw [thick, ->,shorten <=-2pt, shorten >=-2pt] (6') -- (5);

\end{tikzpicture}

 & 
 
\begin{tikzpicture}
 \node (1) at (0,0) {$\star_1$};
 
\node (A) at (0,1.15) {$\mathcal{S}^{(r)}_{3}$};
\node [right of=1, node distance=0.7cm](2) {$\star_2$};

\node [right of=2, node distance=0.7cm](3) {$\star_3$};
\node [right of=3, node distance=0.7cm] (4)  {$\cdots$};
\node [right of=4, node distance=0.7cm] (5) {$\star_{r}$};

\node [above of=1, node distance=0.7cm] (1')  {$\circ$};
\node [above  of=2, node distance=0.7cm] (2')  {$\circ$};
\node [above  of=3, node distance=0.7cm] (3')  {$\circ$};
\node [above  of=4, node distance=0.7cm] (4')  {$\dots$};
\node [above  of=5, node distance=0.7cm] (5')  {$\circ$};
\node [right of=5', node distance=0.7cm] (6')  {$\circ$};

   \draw [thick, <-, shorten <=-2pt, shorten >=-2pt] (1) -- (1');
    \draw [thick, ->,shorten <=-2pt, shorten >=-2pt] (2') -- (2);
    
\draw [thick, <-, shorten <=-2pt, shorten >=-2pt] (3) -- (3');
    \draw [thick, <-, shorten <=-2pt, shorten >=-2pt] (5) -- (5');    
    
    \draw [thick,<-, shorten <=-2pt, shorten >=-2pt] (1) -- (2');
    
 \draw [thick, <-,shorten <=-2pt, shorten >=-2pt] (2) -- (3');
    
\draw [thick, ->, shorten <=-2pt, shorten >=-2pt] (6') -- (5);

\end{tikzpicture} 
 
  \\ 
\hline 
\end{tabular} 
\end{adjustbox}\vspace{0.2cm}
for some $r\geq 1$. Furthermore, the module $M=(M_x,{_y}h_x)_{x,y\in\P}$ in $\md_{sp}(\Bbbk\P)$ such that $M_x=\Bbbk$ for all $x\in\P$  and ${_y}h_x=\mbox{id}_\Bbbk$ for each $x\preceq y$ in $\P$  is the unique sincere  object in $\md_{sp}(\Bbbk\P).$
\end{itemize}
\end{lemma}
\fi

Let $\P$ be a poset of type $\mathbb{A}$ associated with the quiver $Q^{F}$, where $Q$ is a quiver of Dynkin type $\mathbb{A}$ and $F$  an alien set for $Q$. According to \cite{schiffler10}, the quiver $Q$ defines a triangulation $T$ of an $(n+3)$-gon and the category of representations of the quiver $Q$ is equivalent to a category of diagonals denoted by $\mathcal{C}_T$, whose indecomposable objects are the diagonals that do not belong to $T$ and the irreducible morphisms are certain pivoting elementary moves. Following \cite{schifflerserna}, $\mathcal{C}_{(T,F)}$ denotes  the full subcategory of $\mathcal{C}_T$ generated by the so-called \textit{sp-diagonals}, which are defined based on the alien set $F$ and certain criteria of intersection among diagonals.


\section{Posets of type $\mathbb{D}$}\label{posetstypeD}
In this section, we introduce a family of posets which we call posets of type \(\mathbb{D}\).

 \begin{figure}[ht] 
\begin{center}
\begin{tikzcd}[arrows=-,row sep= tiny, column sep = normal]
 & && & \underset{_{n-1}}{\circ} \\
 \underset{_1}{\circ} \arrow[r, "\alpha_1"] & \underset{_2}{\circ}\arrow[r,dotted, thick] & \underset{_{n-3}}{\circ} \arrow[r, "\alpha_{n-3}"] &
  \underset{_{n-2}}{\circ} \arrow[ur, "\alpha_{n-2}"] \arrow[dr, "\alpha_{n-1}"] &  \\
 & && & \underset{_{n}}{\circ}
\end{tikzcd}
\end{center}
\caption{Numbering of vertices and edges in $\overline{Q}$ ($\mathbb{D}_n$  case) }
\label{numberingD}
\end{figure}

The vertex $n-2$ is the unique branch vertex of $Q$, that is, the unique vertex of valence three in $\overline{Q}$, and the end vertices $n-1$ and $n$, which are adjacent to it, are called the \textit{leaves} of $Q$. For $n=4$ all three end vertices are adjacent to the branch vertex, and the labeling of Figure \ref{numberingD} fixes which two of them are the leaves.

\subsection{Definition}\label{subsec:defD}

Let $Q$ be a Dynkin quiver of type $\DD_n$. We say that a poset $\P$ is a \textit{poset over $Q$} if the points of $\P$ are the vertices of $Q$ and $x\prec y$ in $\P$ for every arrow $x\to y$ of $Q$. Let $\P_Q$ be the poset associated with the quiver $Q$ (Section \ref{subsectionposetsof typeA}), that is, $x\preceq y$ in $\P_Q$ if and only if there is a path from $x$ to $y$ in $Q$. Thus a poset over $Q$ is a poset on the vertices of $Q$ which contains all the relations of $\P_Q$. Since $Q$ is a tree, there is at most one path between two vertices, so the Hasse quiver of $\P_Q$ is $Q$ and there are no commutativity relations, and hence $\Bbbk\P_Q=\Bbbk Q$. We define the restriction functor
\[
\operatorname{res}\colon\md(\Bbbk\P)\to\md(\Bbbk Q),\qquad (M_x,{_y}h_x)_{x\preceq y\text{ in }\P}\longmapsto(M_x,{_y}h_x)_{(x\to y)\in Q_1},
\]
which keeps the spaces $M_x$ and the maps of the arrows of $Q$, and is the identity on morphisms. It is well defined because every arrow $x\to y$ of $Q$ gives a relation $x\preceq y$ of $\P$. Equivalently, $\operatorname{res}$ is the restriction of scalars along the algebra inclusion $\Bbbk Q\hookrightarrow\Bbbk\P$ sending each path of $Q$ from $x$ to $y$ to the element of $\Bbbk\P$ given by the relation $x\preceq y$. This map is an injective algebra homomorphism because $Q$ has no parallel paths. In particular, $\operatorname{res}$ is exact, preserves dimension vectors and sends submodules to submodules.

\begin{definition}\label{defposetypeD}
A finite connected poset $\P$ which is not of type $\mathbb{A}$ is said to be a \textit{poset of type $\mathbb{D}$} if there is a Dynkin quiver $Q$ of type $\DD_n$ such that $\P$ is a poset over $Q$, the functor $\operatorname{res}$ sends $\md_{sp}(\Bbbk\P)$ to $\md_{sp}(\Bbbk Q)$, and the induced functor $\operatorname{res}\colon\md_{sp}(\Bbbk\P)\to\md_{sp}(\Bbbk Q)$ is fully faithful. In that case we say that $Q$ is a \textit{Dynkin quiver of} $\P$.
\end{definition}

Thus, for a poset $\P$ of type $\DD$, the category $\md_{sp}(\Bbbk\P)$ is equivalent to a full subcategory of $\md_{sp}(\Bbbk Q)$, and the geometric model of Section \ref{equi} describes this subcategory inside Schiffler's model for $\md\Bbbk Q$. We exclude the posets of type $\AA$, which are treated in \cite{schifflerserna}, because some of them also admit a Dynkin quiver of type $\DD_n$ in the above sense. For instance, the poset $\P_Q$ of the quiver $Q$ of type $\DD_5$ on the left of Example \ref{ex:D4} is of type $\AA$, and the functor $\operatorname{res}$ is the identity for it, since $\Bbbk\P_Q=\Bbbk Q$.

Definition \ref{defposetypeD} is not easy to check directly. In Theorem \ref{thm:charD} we show that a Dynkin quiver $Q$ of type $\DD_n$, such that $\P$ is a poset over $Q$, is a Dynkin quiver of $\P$ if and only if the following two combinatorial conditions hold:
\begin{itemize}
\item[(A)] $\max\P$ is the set of sinks of $Q$;
\item[(B)] for every $x\in\P$ and every $z\in\max\P$ with $x\preceq z$, there is a path from $x$ to $z$ in $Q$.
\end{itemize}
We also show that in that case $Q$ is admissible and $\P$ is obtained from $Q$ by adding a set of alien arrows as in Definition \ref{def:QF} below. This explicit description is the one used in the geometric model, since the sp-arcs of Section \ref{equi} depend on the alien arrows.

\subsection{Admissible quivers and alien arrows}\label{subsec:QF}

A Dynkin quiver $Q$ of type $\mathbb{D}_n$ is called an \textit{admissible quiver} if the branch vertex $n-2$ is a sink of $Q$, receives exactly one arrow of $Q$, is a source of $Q$, or receives the arrows from both leaves $n-1$ and $n$ and points to $n-3$. We call these cases configurations $D_1$, $D_2$, $D_3$ and $D_4$, respectively. Thus $Q$ is not admissible exactly when $n-2$ receives the arrows from $n-3$ and from one leaf, and points to the other leaf. For $n=4$ this never happens, since we may label the two end vertices that point to the branch vertex as the leaves. Note that $n-2$ has exactly one outgoing arrow in configuration $D_4$, no outgoing arrow in configuration $D_1$, and at least two outgoing arrows in configurations $D_2$ and $D_3$.

In configurations $D_1$, $D_2$ and $D_3$ the poset $\P_Q$ is not of type $\mathbb{A}$, since it contains $\mathcal{R}_1$, $\mathcal{R}_2$ or $\mathcal{R}_3$ as a peak-subposet according as $n-2$ receives three arrows, one or none. Indeed, in configuration $D_1$ the three neighbours of the sink $n-2$ are pairwise incomparable and lie below it; in configuration $D_2$ the vertex $n-2$ lies below two distinct sinks, reached through its two outgoing arrows, and above the vertex from which it receives its arrow; and in configuration $D_3$ the vertex $n-2$ lies below three distinct sinks, one reached through each of its outgoing arrows. In configuration $D_4$ the poset $\P_Q$ may or may not be of type $\mathbb{A}$, as the following example shows.

\begin{example} Consider the Dynkin quivers of type $\mathbb{D}_5$
\begin{center}
\begin{tikzcd}[row sep=small, column sep=small]
&&4\arrow[d]&\\
1\arrow[r]&2\arrow[r]&3&5\arrow[l]
\end{tikzcd}
\qquad\qquad
\begin{tikzcd}[row sep=small, column sep=small]
&&5\arrow[d]&\\
1\arrow[r]&2\arrow[r]&3\arrow[r]&4
\end{tikzcd}
\end{center}
In the quiver on the left, the branch vertex $3$ is a sink (configuration $D_1$), so it is admissible. In the one on the right, the branch vertex $3$ receives the arrows from $2$ and from the leaf $5$, and points to the leaf $4$, so it is not admissible, and in fact its associated poset is a poset of type $\mathbb{A}$.
\end{example}

\begin{example}\label{ex:D4}
Consider the Dynkin quivers of type $\mathbb{D}_5$ and $\mathbb{D}_6$
\begin{center}
\begin{tikzcd}[row sep=small, column sep=small]
&&4\arrow[d]&\\
1&2\arrow[l]&3\arrow[l]&5\arrow[l]
\end{tikzcd}
\qquad\qquad
\begin{tikzcd}[row sep=small, column sep=small]
&&&5\arrow[d]&\\
1&2\arrow[l]\arrow[r]&3&4\arrow[l]&6\arrow[l]
\end{tikzcd}
\end{center}
Both are in configuration $D_4$, since the branch vertex receives the arrows from both leaves and points to $n-3$. In the quiver on the left, $\P_Q$ is the poset $4,5\prec3\prec2\prec1$, whose only maximal point is $1$ and whose antichains have at most two elements. Hence it contains neither $\mathcal{R}_1$, which has three pairwise incomparable points, nor $\mathcal{R}_2$, $\mathcal{R}_3$ or $\mathcal{R}_{4,n}$, which have at least two peaks, so $\P_Q$ is of type $\mathbb{A}$. In the quiver on the right, the sinks are $1$ and $3$, and the points $2$, $5$ and $6$ are pairwise incomparable and lie below $3$. Hence $\{3,2,5,6\}$ is a peak-subposet isomorphic to $\mathcal{R}_1$, and $\P_Q$ is not of type $\mathbb{A}$. We show in Example \ref{ex:newD6} that this poset cannot be obtained from a quiver in configurations $D_1$, $D_2$ or $D_3$.
\end{example}

Given a Dynkin  quiver $Q$ of type $\mathbb{D}_n$ whose underlying graph $\overline{Q}$ is given in Figure \ref{numberingD}, we can identify two Dynkin subquivers of type $\mathbb{A}_{n-1}$ whose underlying graphs are  
\begin{center}
\begin{tabular}{ccc}

\begin{tikzcd}[arrows=-,row sep= tiny, column sep = normal]
 \underset{_1}{\circ} \arrow[r, "\alpha_1"] & \underset{_2}{\circ}\arrow[r,dotted, thick] & \underset{_{n-2}}{\circ} \arrow[r, "\alpha_{n-2}"] &
  \underset{_{n-1}}{\circ}  
 \end{tikzcd}
& and &
\begin{tikzcd}[arrows=-,row sep= tiny, column sep = normal]
 \underset{_1}{\circ} \arrow[r, "\alpha_1"] & \underset{_2}{\circ}\arrow[r,dotted, thick] & \underset{_{n-2}}{\circ} \arrow[r, "\alpha_{n-1}"] &
  \underset{_{n}}{\circ}  
 \end{tikzcd}
\end{tabular}
\end{center}
 
These quivers are called the \textit{$\mathbb{A}$-subquivers of $Q$}.

\begin{definition}\label{def:QF}
Let $Q$ be an admissible quiver of type $\DD_n$, with vertices labeled as in Figure \ref{numberingD}, and let $F$ be a set of new arrows such that:
\begin{itemize}
\item[(i)] no arrow of $F$ joins the two leaves $n-1$ and $n$, and the restriction of $F$ to each $\mathbb{A}$-subquiver $S$ of $Q$ is an alien set for $S$ (possibly empty);
\item[(ii)] if the branch vertex $n-2$ receives exactly one arrow of $Q$ and this arrow starts at a leaf $\ell$, then no arrow of $F$ ends at $\ell$;
\item[(iii)] the quiver $Q^F$ obtained from $Q$ by adding the set $F$ of new arrows is acyclic.
\end{itemize}
We say that a poset $\P$ is \textit{associated with} $Q^F$ if $\P=\P_{Q^F}$, which makes sense by (iii).
\end{definition}

\begin{remark}\label{rem:acyclic}
Condition (iii) only concerns oriented cycles through both leaves. Indeed, an oriented cycle of $Q^F$ that avoids a leaf is an oriented cycle of $S^{F|_S}$ for the $\mathbb{A}$-subquiver $S$ obtained by deleting that leaf, which is excluded by condition (d) of Definition \ref{aliensetDef} for $S$. In particular, (iii) is automatic when no arrow of $F$ is incident with a leaf.
\end{remark}

By Theorem \ref{thm:charD}, a finite connected poset which is not of type $\AA$ is of type $\DD$ if and only if it is associated with $Q^F$ for some $Q$ and $F$ as in Definition \ref{def:QF}, and in that case $Q$ is a Dynkin quiver of $\P$.

\begin{remark}\label{rem:eprime}
Condition (ii) only concerns configuration $D_2$ when the arrow received by $n-2$ comes from a leaf $\ell$. Then $\ell$ is a source of $Q$, and $n-2$ points to the other leaf and to $n-3$ (case (c) of Figure \ref{fig-generalquivers}). In that case the only arrows of $Q^F$ that can end at $\ell$ are alien arrows, so (ii) says that $\ell$ is a minimal element of $\P$. In configuration $D_4$ both leaves are sources of $Q$, but condition (ii) does not apply to them, and alien arrows may end at a leaf (Example \ref{ex:D7leaf}). In Lemma \ref{lem:C} we show that, for every admissible quiver, (ii) is equivalent to the fact that alien arrows do not create new relations below the maximal points of $\P$, and this is what makes the restriction functor to $Q$ fully faithful (Proposition \ref{prop:ff}).
\end{remark}

\begin{remark}\label{rem:oldD}
In configurations $D_1$, $D_2$ and $D_3$, condition (ii) is equivalent to the following one: if $n-2$ is not a sink of $Q$, then no arrow of $F$ ends at a leaf which is a source of $Q$. Indeed, a leaf $\ell$ is a source of $Q$ exactly when $\ell\to n-2$ is an arrow of $Q$, and if $n-2$ is not a sink, in these configurations this only happens in configuration $D_2$, where $n-2$ receives exactly one arrow, namely the one from $\ell$. In configuration $D_4$ the two conditions differ, since $n-2$ is not a sink and both leaves are sources of $Q$, while condition (ii) imposes nothing and alien arrows may end at the leaves (Example \ref{ex:D7leaf}).
\end{remark}

\begin{example}\label{ex:newD6}
The poset $\P=\P_Q$ of the quiver $Q$ of type $\DD_6$ on the right of Example \ref{ex:D4} is associated with $Q^F$ for $F=\varnothing$, and hence it is a poset of type $\DD$ by Theorem \ref{thm:charD}, but it is not of the form $\P_{Q'^{F'}}$ for a quiver $Q'$ of type $\DD_6$ in configurations $D_1$, $D_2$ or $D_3$ and any set $F'$ of new arrows. Indeed, suppose that $\P=\P_{Q'^{F'}}$. Every arrow $x\to y$ of $Q'$ gives a relation $x\prec y$ of $\P$, so the underlying graph of $Q'$ is a spanning tree of the comparability graph of $\P$, whose edges are
\[
\{1,2\},\ \{2,3\},\ \{3,4\},\ \{3,5\},\ \{3,6\},\ \{4,5\},\ \{4,6\}.
\]
Since $1$ is only comparable with $2$, and $2$ is only comparable with $1$ and $3$, the tree contains the edges $\{1,2\}$ and $\{2,3\}$, and $1$ is an end vertex at distance two from $3$. The branch vertex of the tree has valence three, so it is $3$ or $4$, since $1$, $2$, $5$ and $6$ have at most two neighbours in the comparability graph. In a tree of type $\DD_6$ the three branches at the branch vertex have $1$, $1$ and $3$ vertices. If the branch vertex were $3$, the branch through $2$ would have exactly the two vertices $2$ and $1$, which is impossible. Hence the branch vertex is $4$, its neighbours are $3$, $5$ and $6$, and the tree has the edges $\{1,2\}$, $\{2,3\}$, $\{3,4\}$, $\{4,5\}$ and $\{4,6\}$, that is, it is the underlying graph of $Q$. The orientation of each edge is determined by the order of $\P$, so $Q'=Q$, which is in configuration $D_4$.
\end{example}

\begin{example}\label{ex:D7leaf}
Let $Q$ be the quiver of type $\DD_7$ given by $1\to2\to3\to4$, $5\to4$, $6\to5$ and $7\to5$, which is in configuration $D_4$, and let $F=\{\alpha\colon1\to6,\ \beta\colon7\to3\}$. The restriction of $F$ to the $\mathbb{A}$-subquiver $S=Q\setminus\{7\}$, given by $1\to2\to3\to4\leftarrow5\leftarrow6$, is $\{\alpha\}$. It is an alien set for $S$. Indeed, $1$ and $6$ lie in the support of $I(4)$, so condition (a) holds; $6$ is a source of $S$, but it is an extreme vertex, so condition (b) holds; and conditions (c) and (d) hold because no path of $S$ joins $1$ and $6$. Similarly, the restriction of $F$ to $Q\setminus\{6\}$ is $\{\beta\}$, which is an alien set, since $7$ and $3$ lie in the support of $I(4)$, $3$ is not a source, and no path of $Q\setminus\{6\}$ joins $7$ and $3$. No arrow of $F$ joins the leaves $6$ and $7$, and condition (ii) does not apply. Finally, $\P_{Q^F}$ is not of type $\mathbb{A}$, since the points $2$, $6$ and $7$ are pairwise incomparable and lie below the unique maximal point $4$, so they form, together with $4$, a peak-subposet isomorphic to $\mathcal{R}_1$. Hence $\P_{Q^F}$ is a poset of type $\DD$ by Theorem \ref{thm:charD}, and the alien arrow $\alpha$ ends at the leaf $6$, which is a source of $Q$.
\end{example}

\begin{example}\label{onepeak-D}The  one-peak poset $\P$  whose  Hasse diagram is 
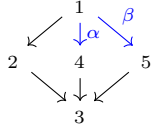
\begin{center}
\begin{tikzcd}[row sep= small, column sep = small]
&_1\arrow[d,"\alpha", blue]\arrow[rd,"\beta", blue]\arrow[dl]&&& \\
_2\arrow[dr]&_4\arrow[d]&_5\arrow[dl]\\
&_3&\\
\end{tikzcd}
\end{center}
is not a poset of type $\mathbb{A}$, but it is a poset of type $\mathbb{D}$ by Theorem \ref{thm:charD}, because it is associated with a quiver $Q^F$, where $F=\{\alpha: 1\to 4, \beta: 1\to 5\}$ and $Q$ is the admissible quiver obtained by deleting the arrows $\alpha$ and $\beta$ (blue color); $\{\alpha\}$ and $\{\beta\}$ are alien sets for the $\AA$-subquivers $1\to 2\to 3\leftarrow 4$ and $1\to 2\to 3\leftarrow 5$ respectively.
\end{example}

\begin{figure}[H]

 
\tikzset{
pattern size/.store in=\mcSize, 
pattern size = 5pt,
pattern thickness/.store in=\mcThickness, 
pattern thickness = 0.3pt,
pattern radius/.store in=\mcRadius, 
pattern radius = 1pt}
\makeatletter
\pgfutil@ifundefined{pgf@pattern@name@_t099xwfrj}{
\pgfdeclarepatternformonly[\mcThickness,\mcSize]{_t099xwfrj}
{\pgfqpoint{0pt}{0pt}}
{\pgfpoint{\mcSize}{\mcSize}}
{\pgfpoint{\mcSize}{\mcSize}}
{
\pgfsetcolor{\tikz@pattern@color}
\pgfsetlinewidth{\mcThickness}
\pgfpathmoveto{\pgfqpoint{0pt}{\mcSize}}
\pgfpathlineto{\pgfpoint{\mcSize+\mcThickness}{-\mcThickness}}
\pgfpathmoveto{\pgfqpoint{0pt}{0pt}}
\pgfpathlineto{\pgfpoint{\mcSize+\mcThickness}{\mcSize+\mcThickness}}
\pgfusepath{stroke}
}}
\makeatother
\tikzset{every picture/.style={line width=0.75pt}} 

\begin{tikzpicture}[x=0.75pt,y=0.75pt,yscale=-1,xscale=1]

\draw    (111.62,140.84) -- (131.74,161.08) ;
\draw [shift={(133.86,163.21)}, rotate = 225.17000000000002] [fill={rgb, 255:red, 0; green, 0; blue, 0 }  ][line width=0.08]  [draw opacity=0] (3.57,-1.72) -- (0,0) -- (3.57,1.72) -- cycle    ;
\draw  [dash pattern={on 0.84pt off 2.51pt}]  (136.27,167.41) -- (145.06,177.12) ;
\draw    (151.45,182.63) -- (171.56,202.87) ;
\draw [shift={(173.68,205)}, rotate = 225.17000000000002] [fill={rgb, 255:red, 0; green, 0; blue, 0 }  ][line width=0.08]  [draw opacity=0] (3.57,-1.72) -- (0,0) -- (3.57,1.72) -- cycle    ;
\draw    (345.73,22) -- (319.51,42.63) ;
\draw [shift={(317.15,44.49)}, rotate = 321.81] [fill={rgb, 255:red, 0; green, 0; blue, 0 }  ][line width=0.08]  [draw opacity=0] (3.57,-1.72) -- (0,0) -- (3.57,1.72) -- cycle    ;
\draw  [dash pattern={on 0.84pt off 2.51pt}]  (311.13,48.24) -- (270.52,77.47) ;
\draw    (264.5,80.47) -- (233.89,101.27) ;
\draw [shift={(231.41,102.96)}, rotate = 325.8] [fill={rgb, 255:red, 0; green, 0; blue, 0 }  ][line width=0.08]  [draw opacity=0] (3.57,-1.72) -- (0,0) -- (3.57,1.72) -- cycle    ;
\draw  [dash pattern={on 0.84pt off 2.51pt}]  (90,60.23) -- (111.06,60.23) ;
\draw    (345.73,22) -- (376.22,39.26) ;
\draw [shift={(378.83,40.74)}, rotate = 209.52] [fill={rgb, 255:red, 0; green, 0; blue, 0 }  ][line width=0.08]  [draw opacity=0] (3.57,-1.72) -- (0,0) -- (3.57,1.72) -- cycle    ;
\draw  [dash pattern={on 0.84pt off 2.51pt}]  (383.34,44.49) -- (396.88,52.73) ;
\draw    (405.91,57.23) -- (436.39,74.49) ;
\draw [shift={(439,75.97)}, rotate = 209.52] [fill={rgb, 255:red, 0; green, 0; blue, 0 }  ][line width=0.08]  [draw opacity=0] (3.57,-1.72) -- (0,0) -- (3.57,1.72) -- cycle    ;
\draw    (405.91,57.23) -- (379.68,77.86) ;
\draw [shift={(377.32,79.72)}, rotate = 321.81] [fill={rgb, 255:red, 0; green, 0; blue, 0 }  ][line width=0.08]  [draw opacity=0] (3.57,-1.72) -- (0,0) -- (3.57,1.72) -- cycle    ;
\draw    (239.84,137.97) -- (220.26,160.07) ;
\draw [shift={(218.28,162.32)}, rotate = 311.52] [fill={rgb, 255:red, 0; green, 0; blue, 0 }  ][line width=0.08]  [draw opacity=0] (3.57,-1.72) -- (0,0) -- (3.57,1.72) -- cycle    ;
\draw  [dash pattern={on 0.84pt off 2.51pt}]  (213.44,166.47) -- (199.84,181.89) ;
\draw    (199.84,181.89) -- (178.48,202.78) ;
\draw [shift={(176.34,204.87)}, rotate = 315.64] [fill={rgb, 255:red, 0; green, 0; blue, 0 }  ][line width=0.08]  [draw opacity=0] (3.57,-1.72) -- (0,0) -- (3.57,1.72) -- cycle    ;
\draw  [dash pattern={on 0.84pt off 2.51pt}]  (91.5,180.42) -- (112.56,180.42) ;
\draw    (240.43,137.69) -- (270.92,154.95) ;
\draw [shift={(273.53,156.43)}, rotate = 209.52] [fill={rgb, 255:red, 0; green, 0; blue, 0 }  ][line width=0.08]  [draw opacity=0] (3.57,-1.72) -- (0,0) -- (3.57,1.72) -- cycle    ;
\draw  [dash pattern={on 0.84pt off 2.51pt}]  (278.04,160.18) -- (291.58,168.42) ;
\draw    (300.6,172.92) -- (331.09,190.18) ;
\draw [shift={(333.7,191.66)}, rotate = 209.52] [fill={rgb, 255:red, 0; green, 0; blue, 0 }  ][line width=0.08]  [draw opacity=0] (3.57,-1.72) -- (0,0) -- (3.57,1.72) -- cycle    ;
\draw    (135.13,48.99) -- (165.61,66.25) ;
\draw [shift={(168.22,67.72)}, rotate = 209.52] [fill={rgb, 255:red, 0; green, 0; blue, 0 }  ][line width=0.08]  [draw opacity=0] (3.57,-1.72) -- (0,0) -- (3.57,1.72) -- cycle    ;
\draw  [dash pattern={on 0.84pt off 2.51pt}]  (172.74,71.47) -- (186.28,79.72) ;
\draw    (195.3,84.22) -- (225.79,101.48) ;
\draw [shift={(228.4,102.96)}, rotate = 209.52] [fill={rgb, 255:red, 0; green, 0; blue, 0 }  ][line width=0.08]  [draw opacity=0] (3.57,-1.72) -- (0,0) -- (3.57,1.72) -- cycle    ;
\draw    (112.62,236.84) -- (132.74,257.08) ;
\draw [shift={(134.86,259.21)}, rotate = 225.17000000000002] [fill={rgb, 255:red, 0; green, 0; blue, 0 }  ][line width=0.08]  [draw opacity=0] (3.57,-1.72) -- (0,0) -- (3.57,1.72) -- cycle    ;
\draw  [dash pattern={on 0.84pt off 2.51pt}]  (137.27,263.41) -- (146.06,273.12) ;
\draw    (152.45,278.63) -- (172.56,298.87) ;
\draw [shift={(174.68,301)}, rotate = 225.17000000000002] [fill={rgb, 255:red, 0; green, 0; blue, 0 }  ][line width=0.08]  [draw opacity=0] (3.57,-1.72) -- (0,0) -- (3.57,1.72) -- cycle    ;
\draw    (240.84,233.97) -- (221.26,256.07) ;
\draw [shift={(219.28,258.32)}, rotate = 311.52] [fill={rgb, 255:red, 0; green, 0; blue, 0 }  ][line width=0.08]  [draw opacity=0] (3.57,-1.72) -- (0,0) -- (3.57,1.72) -- cycle    ;
\draw  [dash pattern={on 0.84pt off 2.51pt}]  (214.44,262.47) -- (200.84,277.89) ;
\draw    (200.84,277.89) -- (179.48,298.78) ;
\draw [shift={(177.34,300.87)}, rotate = 315.64] [fill={rgb, 255:red, 0; green, 0; blue, 0 }  ][line width=0.08]  [draw opacity=0] (3.57,-1.72) -- (0,0) -- (3.57,1.72) -- cycle    ;
\draw  [dash pattern={on 0.84pt off 2.51pt}]  (92.5,276.42) -- (113.56,276.42) ;
\draw    (241.43,233.69) -- (271.92,250.95) ;
\draw [shift={(274.53,252.43)}, rotate = 209.52] [fill={rgb, 255:red, 0; green, 0; blue, 0 }  ][line width=0.08]  [draw opacity=0] (3.57,-1.72) -- (0,0) -- (3.57,1.72) -- cycle    ;
\draw  [dash pattern={on 0.84pt off 2.51pt}]  (279.04,256.18) -- (292.58,264.42) ;
\draw    (301.6,268.92) -- (332.09,286.18) ;
\draw [shift={(334.7,287.66)}, rotate = 209.52] [fill={rgb, 255:red, 0; green, 0; blue, 0 }  ][line width=0.08]  [draw opacity=0] (3.57,-1.72) -- (0,0) -- (3.57,1.72) -- cycle    ;
\draw    (415,232) -- (394.76,247.6) ;
\draw [shift={(392.38,249.43)}, rotate = 322.37] [fill={rgb, 255:red, 0; green, 0; blue, 0 }  ][line width=0.08]  [draw opacity=0] (3.57,-1.72) -- (0,0) -- (3.57,1.72) -- cycle    ;
\draw  [dash pattern={on 0.84pt off 2.51pt}]  (383,257) -- (374,263) ;
\draw    (370.8,265.17) -- (340.19,285.97) ;
\draw [shift={(337.71,287.66)}, rotate = 325.8] [fill={rgb, 255:red, 0; green, 0; blue, 0 }  ][line width=0.08]  [draw opacity=0] (3.57,-1.72) -- (0,0) -- (3.57,1.72) -- cycle    ;
\draw    (394.78,250.69) -- (412.72,266.05) ;
\draw [shift={(415,268)}, rotate = 220.57] [fill={rgb, 255:red, 0; green, 0; blue, 0 }  ][line width=0.08]  [draw opacity=0] (3.57,-1.72) -- (0,0) -- (3.57,1.72) -- cycle    ;
\draw  [color={rgb, 255:red, 255; green, 255; blue, 255 }  ,draw opacity=1 ][fill={rgb, 255:red, 0; green, 0; blue, 255 }  ,fill opacity=0.08 ][line width=0.75]  (230.48,97.5) -- (321,36) -- (120,36) -- cycle ;
\draw  [color={rgb, 255:red, 255; green, 255; blue, 255 }  ,draw opacity=1 ][fill={rgb, 255:red, 0; green, 0; blue, 255 }  ,fill opacity=0.08 ][line width=0.75]  (174,204.87) -- (234.34,140.5) -- (115,140.5) -- cycle ;
\draw  [color={rgb, 255:red, 255; green, 255; blue, 255 }  ,draw opacity=1 ][fill={rgb, 255:red, 0; green, 0; blue, 255 }  ,fill opacity=0.08 ][line width=0.75]  (176,296.87) -- (235,232.5) -- (112,232.5) -- cycle ;
\draw  [color={rgb, 255:red, 255; green, 255; blue, 255 }  ,draw opacity=1 ][fill={rgb, 255:red, 0; green, 0; blue, 255 }  ,fill opacity=0.08 ][line width=0.75]  (336,285.39) -- (412,233) -- (247,233) -- cycle ;
\draw  [color={rgb, 255:red, 255; green, 255; blue, 255 }  ,draw opacity=1 ][pattern=_t099xwfrj,pattern size=7.75pt,pattern thickness=0.2pt,pattern radius=0pt, pattern color={rgb, 155:red, 0; green, 0; blue, 120}] (372,122) -- (366,150) -- (336,185) -- (245,134) -- (292,129) -- cycle ;
\draw  [color={rgb, 255:red, 255; green, 255; blue, 255 }  ,draw opacity=1 ][fill={rgb, 255:red, 0; green, 0; blue, 255 }  ,fill opacity=0.09 ] (410,142) -- (384,162) -- (332.43,182.69) -- (241.43,131.69) -- (353.01,131.29) -- cycle ;

\draw    (368.00,136.00) -- (338.90,186.15) ;
\draw [shift={(337.60,188.40)}, rotate = 300.12] [fill={rgb, 255:red, 0; green, 0; blue, 0 }  ][line width=0.08]  [draw opacity=0] (3.57,-1.72) -- (0,0) -- (3.57,1.72) -- cycle    ;
\draw    (402.00,150.00) -- (341.79,188.79) ;
\draw [shift={(339.60,190.20)}, rotate = 327.21] [fill={rgb, 255:red, 0; green, 0; blue, 0 }  ][line width=0.08]  [draw opacity=0] (3.57,-1.72) -- (0,0) -- (3.57,1.72) -- cycle    ;
\draw (332.43,190.69) node [anchor=north west][inner sep=0.75pt]  [font=\scriptsize]  {$z$};
\draw (51,32) node [anchor=north west][inner sep=0.75pt]   [align=left] {(a)};
\draw (51,137) node [anchor=north west][inner sep=0.75pt]   [align=left] {(b)};
\draw (50,228) node [anchor=north west][inner sep=0.75pt]   [align=left] {(c)};
\draw (370.43,76.69) node [anchor=north west][inner sep=0.75pt]  [font=\scriptsize]  {$x$};
\draw (435.43,74.69) node [anchor=north west][inner sep=0.75pt]  [font=\scriptsize]  {$y$};
\draw (366,122) node [anchor=north west][inner sep=0.75pt]  [font=\scriptsize]  {$x$};
\draw (384,249) node [anchor=north west][inner sep=0.75pt]  [font=\scriptsize]  {$w$};

\draw (404,137) node [anchor=north west][inner sep=0.75pt]  [font=\scriptsize]  {$y$};
\draw (415,265) node [anchor=north west][inner sep=0.75pt]  [font=\scriptsize]  {$y$};
\draw (416.43,220.69) node [anchor=north west][inner sep=0.75pt]  [font=\scriptsize]  {$x$};
\draw (332.43,290.69) node [anchor=north west][inner sep=0.75pt]  [font=\scriptsize]  {$z$};
\draw (169.43,305.69) node [anchor=north west][inner sep=0.75pt]  [font=\scriptsize]  {$z'$};
\draw (170.43,206.69) node [anchor=north west][inner sep=0.75pt]  [font=\scriptsize]  {$z'$};
\draw (225.43,103.69) node [anchor=north west][inner sep=0.75pt]  [font=\scriptsize]  {$z$};

\begin{scope}[shift={(0,100)}]
\draw  [color={rgb, 255:red, 255; green, 255; blue, 255 }  ,draw opacity=1 ][fill={rgb, 255:red, 0; green, 0; blue, 255 }  ,fill opacity=0.08 ][line width=0.75]  (176,296.87) -- (235,232.5) -- (112,232.5) -- cycle ;
\draw  [color={rgb, 255:red, 255; green, 255; blue, 255 }  ,draw opacity=1 ][fill={rgb, 255:red, 0; green, 0; blue, 255 }  ,fill opacity=0.08 ][line width=0.75]  (336,285.39) -- (422,224) -- (247,233) -- cycle ;
\draw  [color={rgb, 255:red, 255; green, 255; blue, 255 }  ,draw opacity=1 ][pattern=_t099xwfrj,pattern size=7.75pt,pattern thickness=0.2pt,pattern radius=0pt, pattern color={rgb, 155:red, 0; green, 0; blue, 120}] (338,287) -- (427,274) -- (400,244) -- (250,236) -- cycle ;
\draw    (112.62,236.84) -- (132.74,257.08) ;
\draw [shift={(134.86,259.21)}, rotate = 225.17] [fill={rgb, 255:red, 0; green, 0; blue, 0 }  ][line width=0.08]  [draw opacity=0] (3.57,-1.72) -- (0,0) -- (3.57,1.72) -- cycle    ;
\draw  [dash pattern={on 0.84pt off 2.51pt}]  (137.27,263.41) -- (146.06,273.12) ;
\draw    (152.45,278.63) -- (172.56,298.87) ;
\draw [shift={(174.68,301)}, rotate = 225.17] [fill={rgb, 255:red, 0; green, 0; blue, 0 }  ][line width=0.08]  [draw opacity=0] (3.57,-1.72) -- (0,0) -- (3.57,1.72) -- cycle    ;
\draw    (240.84,233.97) -- (221.26,256.07) ;
\draw [shift={(219.28,258.32)}, rotate = 311.52] [fill={rgb, 255:red, 0; green, 0; blue, 0 }  ][line width=0.08]  [draw opacity=0] (3.57,-1.72) -- (0,0) -- (3.57,1.72) -- cycle    ;
\draw  [dash pattern={on 0.84pt off 2.51pt}]  (214.44,262.47) -- (200.84,277.89) ;
\draw    (200.84,277.89) -- (179.48,298.78) ;
\draw [shift={(177.34,300.87)}, rotate = 315.64] [fill={rgb, 255:red, 0; green, 0; blue, 0 }  ][line width=0.08]  [draw opacity=0] (3.57,-1.72) -- (0,0) -- (3.57,1.72) -- cycle    ;
\draw  [dash pattern={on 0.84pt off 2.51pt}]  (92.5,276.42) -- (113.56,276.42) ;
\draw    (241.43,233.69) -- (271.92,250.95) ;
\draw [shift={(274.53,252.43)}, rotate = 209.52] [fill={rgb, 255:red, 0; green, 0; blue, 0 }  ][line width=0.08]  [draw opacity=0] (3.57,-1.72) -- (0,0) -- (3.57,1.72) -- cycle    ;
\draw  [dash pattern={on 0.84pt off 2.51pt}]  (279.04,256.18) -- (292.58,264.42) ;
\draw    (301.6,268.92) -- (332.09,286.18) ;
\draw [shift={(334.7,287.66)}, rotate = 209.52] [fill={rgb, 255:red, 0; green, 0; blue, 0 }  ][line width=0.08]  [draw opacity=0] (3.57,-1.72) -- (0,0) -- (3.57,1.72) -- cycle    ;
\draw    (415,232) -- (394.76,247.6) ;
\draw [shift={(392.38,249.43)}, rotate = 322.37] [fill={rgb, 255:red, 0; green, 0; blue, 0 }  ][line width=0.08]  [draw opacity=0] (3.57,-1.72) -- (0,0) -- (3.57,1.72) -- cycle    ;
\draw    (415,268) -- (397.06,252.64) ;
\draw [shift={(394.78,250.69)}, rotate = 40.57] [fill={rgb, 255:red, 0; green, 0; blue, 0 }  ][line width=0.08]  [draw opacity=0] (3.57,-1.72) -- (0,0) -- (3.57,1.72) -- cycle    ;
\draw  [dash pattern={on 0.84pt off 2.51pt}]  (383,257) -- (374,263) ;
\draw    (370.8,265.17) -- (340.19,285.97) ;
\draw [shift={(337.71,287.66)}, rotate = 325.8] [fill={rgb, 255:red, 0; green, 0; blue, 0 }  ][line width=0.08]  [draw opacity=0] (3.57,-1.72) -- (0,0) -- (3.57,1.72) -- cycle    ;
\draw (50,228) node [anchor=north west][inner sep=0.75pt]   [align=left] {(d)};
\draw (416.43,220.69) node [anchor=north west][inner sep=0.75pt]  [font=\scriptsize]  {$x$};
\draw (378,249) node [anchor=north west][inner sep=0.75pt]  [font=\scriptsize]  {$w$};
\draw (417,267) node [anchor=north west][inner sep=0.75pt]  [font=\scriptsize]  {$y$};
\draw (332.43,290.69) node [anchor=north west][inner sep=0.75pt]  [font=\scriptsize]  {$z$};
\draw (169.43,305.69) node [anchor=north west][inner sep=0.75pt]  [font=\scriptsize]  {$z'$};
\end{scope}

\end{tikzpicture}
\caption{General admissible quivers and  alien arrows}
\label{fig-generalquivers}
\end{figure}

\begin{remark}\label{rem-generalquivers} Figure \ref{fig-generalquivers} shows the general form of an admissible quiver according to the configuration at its branch vertex. In case (a), the branch vertex receives an arrow from the chain and points to the leaves $x$ and $y$ (configuration $D_2$), or it coincides with the source from which the paths to $z$ and to the leaves start (configuration $D_3$). In case (b), the branch vertex is the sink $z$, which receives arrows from the leaves $x$ and $y$ and from the chain (configuration $D_1$). In case (c), the leaf $x$ is a source with $x\to w$, and $w$ points to the leaf $y$ and to the chain (configuration $D_2$). In case (d), the leaves $x$ and $y$ are sources with $x\to w$ and $y\to w$, and $w$ points to the chain (configuration $D_4$).

The shaded regions are the supports of the injective representations $I(z)$ and $I(z')$, and by condition (a) of Definition \ref{aliensetDef} each alien arrow has both endpoints in one of them. In cases (b) and (d) there are two overlapping regions, one for each $\mathbb{A}$-subquiver, because alien arrows may use the leaf $x$ and the leaf $y$, although none joins $x$ with $y$. By Definition \ref{def:QF}(i) and Proposition \ref{0}, deleting the leaf $x$ or the leaf $y$ together with its incident arrows gives a quiver whose associated poset is of type $\mathbb{A}$. In case (c), no alien arrow ends at $w$, since $w$ is a source of the $\mathbb{A}$-subquiver $Q\setminus\{x\}$ which is not an extreme vertex, contradicting condition (b) of Definition \ref{aliensetDef}. No alien arrow ends at the leaf $x$ either, by Definition \ref{def:QF}(ii). In case (d) the branch vertex $w$ has a single outgoing arrow, so $w$, $x$ and $y$ reach the same sinks as $n-3$, and alien arrows may end at both leaves (Example \ref{ex:D7leaf}).
\end{remark}

\subsection{Restriction to the Dynkin quiver}\label{subsec:res}

Let $\P$ be a poset associated with $Q^F$ as in Definition \ref{def:QF}. Since every path of $Q$ is a path of $Q^F$, $\P$ is a poset over $Q$, and we consider the restriction functor $\operatorname{res}\colon\md(\Bbbk\P)\to\md(\Bbbk Q)$ of Section \ref{subsec:defD}. We use the notation $U(x)$ of Section \ref{subsectionposetsof typeA}, and for a $\Bbbk\P$-module $M$ and $x\in\P$ we consider the map
\[
\iota_x=({_z}h_x)_z\colon M_x\longrightarrow\bigoplus_{z\in\max\P,\ z\succeq x}M_z .
\]
By Proposition \ref{conditions,a,b}, $M$ is socle-projective if and only if $\iota_x$ is injective for every $x\in\P$ (for $x$ maximal, $\iota_x$ is the identity).

\begin{lemma}\label{lem:C}
Let $\P$ be a poset associated with $Q^F$ as in Definition \ref{def:QF}. Then:
\begin{enumerate}
\item[(1)] $\max\P$ is the set of sinks of $Q$; in particular, no alien arrow starts at a sink of $Q$.
\item[(2)] For every alien arrow $\alpha\colon s\to t$ in $F$, the set $U(t)$ consists of a single sink, and this sink belongs to $U(s)$.
\item[(3)] For every $z\in\max\P$ and every $x\in\P$ with $x\preceq z$, there is a path from $x$ to $z$ in $Q$. Equivalently, $\{z\in\max\P: z\succeq x\}=U(x)$ for all $x$.
\end{enumerate}
Moreover, let $Q$ be an admissible quiver and $F$ a set of arrows satisfying conditions (i) and (iii) of Definition \ref{def:QF}. Then the poset associated with $Q^F$ satisfies (1), and $F$ satisfies condition (ii) of Definition \ref{def:QF} if and only if (3) holds.
\end{lemma}

\begin{proof}
(1) If $z$ is a sink of $Q$ and $\alpha\colon z\to t$ were an alien arrow, condition (a) of Definition \ref{aliensetDef}, applied to an $\mathbb{A}$-subquiver $S$ containing $\alpha$ (it exists by Definition \ref{def:QF}(i)), would give a sink $z''$ of $S$ with $z,t\in\Supp I(z'')$; then $z''=z$, there is a path from $t$ to $z$, and together with $\alpha$ we get an oriented cycle, contradicting condition (d) for $S$. Thus the sinks of $Q$ are maximal in $\P$, and the remaining vertices are not, since an arrow of $Q$ starts at them.

(2) Let $S$ be an $\mathbb{A}$-subquiver containing $s$ and $t$ (it exists by Definition \ref{def:QF}(i)). By condition (a), there is a sink $z$ of $S$ with $s,t\in\Supp I(z)$ in $S$. We first show that $|U(t)|=1$. Suppose that $|U(t)|\ge 2$, and let $v$ be the first vertex with at least two outgoing arrows on the path that starts at $t$ and follows, at each vertex, its unique outgoing arrow. If $v=t$ is a source of $Q$, then $t$ has two neighbours in $S$ and is not an extreme vertex of $S$, contradicting condition (b). Otherwise $v$ receives an arrow, so it has valence three and $v=n-2$. Since $n-2$ has at least two outgoing arrows and valence three, it receives exactly one arrow $\ell\to n-2$; in particular $Q$ is in configuration $D_2$. (In configuration $D_4$ this case does not occur, since $n-2$ has a single outgoing arrow.) If $\ell$ is a leaf, then either $t=n-2$, which is a non-extreme source of $Q\setminus\{\ell\}$, contradicting (b) for this $\mathbb{A}$-subquiver (note that $s\neq\ell$ by condition (c)), or $t=\ell$, which is excluded by Definition \ref{def:QF}(ii). Hence $\ell=n-3$, and the two outgoing arrows of $n-2$ end at the leaves, which are sinks. Neither $s$ nor $t$ is a leaf, by (1) and because $|U(t)|\ge2$, so $\alpha$ is an arrow of $S'=Q\setminus\{n\}$. The only sink of $S'$ reachable from $t$ is $n-1$, so condition (a) for $S'$ gives a path from $s$ to $n-1$, which passes through $n-3$ and $n-2$. If $s=n-2$ or $t=n-2$, one of them reaches the other in $Q$. Otherwise $s$ and $t$ lie on the chain $1,\dots,n-3$ and both reach $n-3$ along it, so again one reaches the other. This contradicts condition (c) or (d). Hence $U(t)=\{z'\}$ for a sink $z'$. Since $t$ reaches $z$, every path of $Q$ starting at $z$ ends at a sink reachable from $t$, that is, at $z'$. Thus $z\in\Supp I(z')$, and hence $s\in\Supp I(z')$, that is, $z'\in U(s)$.

(3) Let $x\preceq z$ with $z\in\max\P$, and let $x=v_0\to v_1\to\cdots\to v_k=z$ be a path in $Q^F$. We show that $z\in U(v_i)$ by descending induction on $i$. This is clear for $i=k$ and for arrows of $Q$, and for alien arrows it follows from (2).

Finally, suppose that $F$ satisfies (i) but not (ii), so that $n-2$ receives exactly one arrow of $Q$, namely $\ell\to n-2$ for a leaf $\ell$, and there is $\alpha\colon s\to\ell$ in $F$. Since $n-2$ has valence three and receives no other arrow, it has two outgoing arrows, one of them to the other leaf $\ell'$, which is a sink. The leaf $\ell'$ can only be reached through $n-2$, and $n-2$ only through $\ell$, which is a source different from $s$. Thus $\ell'\notin U(s)$, although $s\prec\ell\prec n-2\prec\ell'$ in $\P$, so (3) fails. The proof of (1) does not use (ii), and those of (2) and (3) only used it to exclude this situation. Hence (ii) is equivalent to (3) in the presence of (i).
\end{proof}

\begin{proposition}\label{prop:ff}
Let $\P$ be a poset associated with $Q^F$ as in Definition \ref{def:QF}. Then:
\begin{enumerate}
\item[(1)] for every $x\in\P$, the components of $\iota_x$ are compositions of the structure maps along paths of $Q$;
\item[(2)] $\operatorname{res}$ sends $\md_{sp}(\Bbbk\P)$ to $\md_{sp}(\Bbbk Q)$;
\item[(3)] the functor $\operatorname{res}\colon\md_{sp}(\Bbbk\P)\to\md_{sp}(\Bbbk Q)$ is fully faithful; in particular, if $M$ is indecomposable then $\operatorname{res}M$ is indecomposable, and $M\cong N$ if and only if $\operatorname{res}M\cong\operatorname{res}N$;
\item[(4)] a module $N\in\md_{sp}(\Bbbk Q)$ is isomorphic to $\operatorname{res}M$ for some $M\in\md_{sp}(\Bbbk\P)$ if and only if, for every alien arrow $\alpha\colon s\to t$ in $F$,
\[
\bigl({_z}h_s\bigr)_{z\in U(t)}(N_s)\ \subseteq\ \operatorname{Im}\Bigl(\bigl({_z}h_t\bigr)_{z\in U(t)}\colon N_t\to\textstyle\bigoplus_{z\in U(t)}N_z\Bigr),
\]
where the maps ${_z}h_s$ and ${_z}h_t$ are compositions along paths of $Q$. In that case, the structure of $\Bbbk\P$-module on $N$ is unique.
\end{enumerate}
\end{proposition}

\begin{proof}
(1) Since $\Bbbk\P$ is an incidence algebra, ${_z}h_x$ is the composition along any path from $x$ to $z$ in $Q^F$, in particular along the path of $Q$ given by Lemma \ref{lem:C}(3).

(2) By (1) and Lemma \ref{lem:C}, the condition that $\iota_x$ is injective for $M$ is literally the same condition for $\operatorname{res}M$ over $\Bbbk Q=\Bbbk\P_Q$, since $\max\P_Q=\max\P$.

(3) The functor is faithful since it does not change the vector spaces. Let $f\colon\operatorname{res}M\to\operatorname{res}N$ be a morphism of $\Bbbk Q$-modules and let $x\preceq y$ in $\P$. For every $z\in U(y)$, using that $f$ commutes with compositions along paths of $Q$ and Lemma \ref{lem:C}(3), we get
\[
{_z}h^N_y\,f_y\,{_y}h^M_x=f_z\,{_z}h^M_y\,{_y}h^M_x=f_z\,{_z}h^M_x={_z}h^N_x\,f_x={_z}h^N_y\,{_y}h^N_x\,f_x .
\]
Hence $\iota^N_y\bigl(f_y\,{_y}h^M_x\bigr)=\iota^N_y\bigl({_y}h^N_x\,f_x\bigr)$, and since $\iota^N_y$ is injective, $f$ commutes with ${_y}h_x$. Thus $f$ is a morphism of $\Bbbk\P$-modules. The remaining statements are formal consequences of full faithfulness.

(4) The inclusion is invariant under isomorphisms of $\Bbbk Q$-modules, since an isomorphism $g\colon N\to N'$ commutes with composites along paths of $Q$, and $\bigoplus_{z\in U(t)}g_z$ maps each side of the inclusion for $N$ onto the corresponding side for $N'$. Hence it suffices to consider the case $N=\operatorname{res}M$. In that case, ${_z}h_s={_z}h_t\,{_t}h^M_s$ for all $z\in U(t)$, which gives the inclusion. Conversely, if the inclusion holds, there is a unique linear map ${_t}h_s$ with $\iota_t\,{_t}h_s=({_z}h_s)_{z\in U(t)}$, since $\iota_t$ is injective; we assign it to $\alpha$. By induction along a path $c$ of $Q^F$ from $x$ to $y$, the composite $h_c$ satisfies $\iota_y h_c=({_z}h_x)_{z\in U(y)}$, where for alien arrows we use $U(t)\subseteq U(s)$ (Lemma \ref{lem:C}(2)). This expression does not depend on $c$, and $\iota_y$ is injective, so neither does $h_c$. We obtain a $\Bbbk\P$-module, which is socle-projective by the proof of (2). Uniqueness follows from the injectivity of $\iota_t$.
\end{proof}

\begin{corollary}\label{finiterepresentation}
Let $\P$ be a poset associated with $Q^F$ as in Definition \ref{def:QF}. Then $\md_{sp}(\Bbbk\P)$ is of finite representation type, with at most $n(n-1)$ indecomposable objects up to isomorphism.
\end{corollary}
\begin{proof}
By Proposition \ref{prop:ff}(2) and (3), $\operatorname{res}$ induces an injective map from the isomorphism classes of indecomposable objects of $\md_{sp}(\Bbbk\P)$ to the isomorphism classes of indecomposable socle-projective $\Bbbk Q$-modules. By Gabriel's theorem \cite{gabriel1}, the latter are at most as many as the $n(n-1)$ positive roots of $\DD_n$.
\end{proof}

\begin{corollary}\label{finiteprin}
Let $\P$ be a poset associated with $Q^F$ as in Definition \ref{def:QF}, and let $\P^-=\P\setminus\max\P$. Then the category $\prin(\Bbbk\P)$ of prinjective $\Bbbk\P$-modules is of finite representation type; more precisely, the number of isomorphism classes of indecomposable prinjective $\Bbbk\P$-modules equals the number of isomorphism classes of indecomposable objects of $\md_{sp}(\Bbbk\P)$ plus $|\P^-|$.
\end{corollary}
\begin{proof}
Simson's adjustment functor $\vartheta_{\P}\colon\prin(\Bbbk\P)\to\md_{sp}(\Bbbk\P)$ is full, dense and additive, and induces an equivalence $\prin(\Bbbk\P)/[e_j\Bbbk\P^-:j\in\P^-]\simeq\md_{sp}(\Bbbk\P)$ \cite[Lemma 2.1]{simson3}. Hence the indecomposable prinjective $\Bbbk\P$-modules annihilated by $\vartheta_{\P}$ are precisely the modules $e_j\Bbbk\P^-$ with $j\in\P^-$, which are pairwise non-isomorphic because $\dm(e_j\Bbbk\P^-)$ is the $j$-th standard basis vector \cite[Section 2]{simson3}, and $\vartheta_{\P}$ induces a bijection between the isomorphism classes of the remaining indecomposable prinjective $\Bbbk\P$-modules and those of the indecomposable objects of $\md_{sp}(\Bbbk\P)$. Now apply Corollary \ref{finiterepresentation}.
\end{proof}

\begin{corollary}\label{suppoftypeD}
Let $\P$ be a poset associated with $Q^F$ as in Definition \ref{def:QF}, and let $M$ be an indecomposable object of $\md_{sp}(\Bbbk\P)$. Then $\underline{\dim}\,M$ is a positive root of $Q$, and its entries are as described in Remark \ref{applications}.
\end{corollary}
\begin{proof}
By Proposition \ref{prop:ff}(3), $\operatorname{res}M$ is an indecomposable representation of $Q$, and $\underline{\dim}\,M=\underline{\dim}\operatorname{res}M$. By Gabriel's theorem \cite{gabriel1}, this vector is a positive root of $Q$.
\end{proof}

We now show that the three descriptions of posets of type $\DD$ given in Section \ref{subsec:defD} agree. We first give a criterion for a set of new arrows to be an alien set.

\begin{lemma}\label{lem:coverA}
Let $\P$ be a poset and let $S$ be a Dynkin quiver of type $\AA$ whose vertices are points of $\P$, such that $x\prec y$ in $\P$ for every arrow $x\to y$ of $S$. Assume that the following conditions hold.
\begin{itemize}
\item[(A$'$)] Every sink of $S$ is a maximal point of $\P$.
\item[(B$'$)] For every vertex $x$ of $S$, there is a vertex $z$ of $S$ which is maximal in $\P$ and satisfies $x\preceq z$; and for every vertex $z$ of $S$ which is maximal in $\P$ and satisfies $x\preceq z$, there is a path from $x$ to $z$ in $S$.
\end{itemize}
Let $F_S$ be the set of new arrows $s\to t$ such that $s$ and $t$ are vertices of $S$, $t$ covers $s$ in $\P$, and $s\to t$ is not an arrow of $S$. Then $F_S$ is an alien set for $S$. If moreover every point of $\P$ is a vertex of $S$, then $\P=\P_{S^{F_S}}$, and hence $\P$ is a poset of type $\AA$.
\end{lemma}

\begin{proof}
Every arrow $x\to y$ of $S^{F_S}$ satisfies $x\prec y$ in $\P$. Hence, if $x_0\to x_1\to\cdots\to x_k$ is a path of $S^{F_S}$ with $k\ge1$, then $x_0\prec x_1\prec\cdots\prec x_k$ in $\P$. In particular, $S^{F_S}$ is acyclic, which is condition (d) of Definition \ref{aliensetDef}. Let $\alpha\colon s\to t$ be an arrow of $F_S$. We verify conditions (a), (b) and (c) for $\alpha$.

(a) By (B$'$), there is a vertex $z$ of $S$ which is maximal in $\P$ and satisfies $t\preceq z$. Since $s\prec t\preceq z$, condition (B$'$) shows that $z$ is reachable in $S$ from $s$ and from $t$. Moreover, $z$ is a sink of $S$, since an arrow $z\to y$ of $S$ would give $z\prec y$, contradicting the maximality of $z$. Hence $s,t\in\Supp I(z)$, where $I(z)$ is the indecomposable injective representation of $S$ at $z$.

(b) Suppose that $t$ is a source of $S$ which is not an extreme vertex of $S$. Then $t$ has two neighbours $u$ and $v$ in $S$, and $t\to u$ and $t\to v$ are arrows of $S$. Removing $t$ from the underlying graph of $S$, which is a path, we obtain two paths, $S_u$ containing $u$ and $S_v$ containing $v$. Starting at $u$ and following arrows of $S$, we reach a sink $z_u$ of $S$. This path does not pass through $t$, since no arrow of $S$ ends at $t$, so $z_u$ lies in $S_u$. In the same way we obtain a sink $z_v$ of $S$, reachable from $v$, which lies in $S_v$. By (A$'$), $z_u$ and $z_v$ are maximal in $\P$, and $s\prec t\prec z_u$ and $s\prec t\prec z_v$. By (B$'$), there are paths in $S$ from $s$ to $z_u$ and from $s$ to $z_v$. Since $s\neq t$, the vertex $s$ lies in $S_u$ or in $S_v$, say in $S_u$. Then the path from $s$ to $z_v$ passes through $t$, so it contains an arrow ending at $t$, which is impossible. Hence $t$ is not a source of $S$ unless it is an extreme vertex.

(c) Let $c$ be a path of $S^{F_S}$ from $s$ to $t$ different from $\alpha$. Since $s\to t$ is not an arrow of $S$ and $\alpha$ is the only arrow of $F_S$ from $s$ to $t$, the path $c$ has length at least two. Its second vertex $w$ then satisfies $s\prec w\prec t$, contradicting the fact that $t$ covers $s$. Hence $\alpha$ is the unique path from $s$ to $t$ in $S^{F_S}$.

Thus $F_S$ is an alien set for $S$. Assume now that every point of $\P$ is a vertex of $S$. Every relation of $\P_{S^{F_S}}$ is a relation of $\P$ by the first paragraph. Conversely, let $x\prec y$ in $\P$. Since $\P$ is finite, there is a chain $x=y_0\prec y_1\prec\cdots\prec y_k=y$ in which $y_i$ covers $y_{i-1}$ for every $i$. Each $y_{i-1}\to y_i$ is an arrow of $S$ or an arrow of $F_S$, so there is a path from $x$ to $y$ in $S^{F_S}$. Hence $\P=\P_{S^{F_S}}$, and $\P$ is of type $\AA$ by Proposition \ref{0}.
\end{proof}

\begin{theorem}\label{thm:charD}
Let $Q$ be a Dynkin quiver of type $\DD_n$, and let $\P$ be a poset over $Q$ which is not of type $\AA$. The following conditions are equivalent.
\begin{itemize}
\item[(1)] The functor $\operatorname{res}$ sends $\md_{sp}(\Bbbk\P)$ to $\md_{sp}(\Bbbk Q)$, and the induced functor $\operatorname{res}\colon\md_{sp}(\Bbbk\P)\to\md_{sp}(\Bbbk Q)$ is fully faithful.
\item[(2)] The set $\max\P$ is the set of sinks of $Q$, and for every $x\in\P$ and every $z\in\max\P$ with $x\preceq z$, there is a path from $x$ to $z$ in $Q$.
\item[(3)] $Q$ is admissible and $\P$ is associated with $Q^F$ for some set $F$ of new arrows as in Definition \ref{def:QF}.
\end{itemize}
Moreover, if these conditions hold, then (3) holds with $F$ the set of new arrows $s\to t$ such that $t$ covers $s$ in $\P$ and $s\to t$ is not an arrow of $Q$. In particular, $\P$ is a poset of type $\DD$ with Dynkin quiver $Q$ if and only if $\P$ is connected and satisfies conditions (A) and (B) of Section \ref{subsec:defD} with respect to $Q$.
\end{theorem}

\begin{proof}
The implication (3)$\Rightarrow$(1) is Proposition \ref{prop:ff}(2) and (3).

(1)$\Rightarrow$(2). Every maximal point $w$ of $\P$ is a sink of $Q$, since an arrow $w\to y$ of $Q$ would give $w\prec y$. Conversely, let $z$ be a sink of $Q$ and suppose that $z$ is not maximal in $\P$. Let $M$ be the $\Bbbk\P$-module with $M_y=\Bbbk$ if $z\preceq y$ and $M_y=0$ otherwise, and with ${_y}h_x$ the identity of $\Bbbk$ if $z\preceq x\preceq y$ and zero otherwise. Condition (a) of Proposition \ref{conditions,a,b} holds, since if $x\preceq y\preceq w$ and $z\preceq x$, then $z\preceq y$. The module $M$ is socle-projective. Indeed, if $M_x\neq0$ and $x$ is not maximal, there is $z'\in\max\P$ with $x\preceq z'$, and ${_{z'}}h_x$ is the identity, so the subspace $I_x$ of Proposition \ref{conditions,a,b}(b) is zero. Moreover, $\operatorname{End}_{\Bbbk\P}(M)\cong\Bbbk$. Indeed, if $f$ is an endomorphism of $M$ and $z\preceq y$, then $f_y=f_y\,{_y}h_z={_y}h_z\,f_z=f_z$, as scalars. On the other hand, every arrow of $Q$ incident with $z$ ends at $z$, since $z$ is a sink of $Q$, and starts at a point $x$ with $x\prec z$, so that $M_x=0$. Hence $\operatorname{res}M=S(z)\oplus M'$, where $M'_z=0$ and $M'_y=M_y$ for $y\neq z$. Since $z$ is not maximal, there is $y\succ z$, so $M'\neq0$ and $\dim_{\Bbbk}\operatorname{End}_{\Bbbk Q}(\operatorname{res}M)\ge2$. This contradicts (1), since a fully faithful functor induces an isomorphism $\operatorname{End}_{\Bbbk\P}(M)\cong\operatorname{End}_{\Bbbk Q}(\operatorname{res}M)$. Hence $\max\P$ is the set of sinks of $Q$.

Now let $x\in\P$ and $z\in\max\P$ with $x\preceq z$, and suppose that there is no path from $x$ to $z$ in $Q$. Let $N$ be the $\Bbbk\P$-module with $N_y=\Bbbk$ if $x\preceq y\preceq z$ and $N_y=0$ otherwise, and with ${_y}h_w$ the identity of $\Bbbk$ if $x\preceq w\preceq y\preceq z$ and zero otherwise. Condition (a) of Proposition \ref{conditions,a,b} holds. Indeed, let $w\preceq y\preceq v$. If $w$ and $v$ lie in the interval $[x,z]$, then so does $y$, and both ${_v}h_y\,{_y}h_w$ and ${_v}h_w$ are the identity. Otherwise ${_v}h_w=0$, and also ${_v}h_y\,{_y}h_w=0$, because $y$ does not lie in $[x,z]$ or one of $w$ and $v$ does not. The module $N$ is socle-projective, since ${_z}h_y$ is the identity for every non-maximal point $y$ with $N_y\neq0$. By (1), $\operatorname{res}N$ is a socle-projective representation of $Q$. Starting at $x$ and following arrows of $Q$ whose endpoints lie in $[x,z]$ as long as possible, we obtain, since $Q$ is acyclic, a vertex $y\in[x,z]$ reachable from $x$ in $Q$ such that no arrow of $Q$ goes from $y$ to a point of $[x,z]$. Then $y\neq z$, since $z$ is not reachable from $x$. Moreover, $y$ is not a sink of $Q$. Indeed, otherwise $y$ would be maximal in $\P$ by the previous paragraph, and $y\preceq z$ would give $y=z$. Every arrow of $Q$ starting at $y$ ends at a vertex $w$ with $N_w=0$, and $N_y=\Bbbk$, so $S(y)$ is a submodule of $\operatorname{res}N$. Since $y$ is not a sink, $S(y)$ is not projective, which contradicts the fact that $\operatorname{res}N$ is socle-projective. This proves (2).

(2)$\Rightarrow$(3). Assume (2), and let $F$ be the set of new arrows defined in the statement. We denote the two leaves by $\ell$ and $\ell'$, in either order. Recall that a leaf has a single neighbour, the branch vertex $n-2$, so it is never an inner vertex of a path of $Q$. The poset $\P$ is connected, since $Q$ is connected and its arrows are relations of $\P$.

We first show that $Q$ is admissible. Otherwise $n-2$ receives the arrows from $n-3$ and from a leaf $\ell$, and points to the other leaf $\ell'$. Let $S''$ be the quiver with the same vertices as $Q$, whose arrows are those of $Q$ except $\ell\to n-2$, together with a new arrow $\ell\to\ell'$. Its underlying graph is the path $1,\dots,n-3,n-2,\ell',\ell$, so $S''$ is a Dynkin quiver of type $\AA_n$. Its arrows are relations of $\P$, since $\ell\prec n-2\prec\ell'$. The vertex $\ell$ is not a sink of either quiver, and the arrows starting at any other vertex are the same in $Q$ and in $S''$. Hence $Q$ and $S''$ have the same sinks, and for every vertex $x\neq\ell$ the paths starting at $x$ are the same in both quivers. Moreover, the only sink reachable from $\ell$ is $\ell'$, both in $Q$, where the only path of length two starting at $\ell$ is $\ell\to n-2\to\ell'$, and in $S''$, where $\ell\to\ell'$ is the only arrow starting at $\ell$. Hence, by (2), conditions (A$'$) and (B$'$) of Lemma \ref{lem:coverA} hold for $S''$, whose vertices are all the points of $\P$, and Lemma \ref{lem:coverA} shows that $\P$ is of type $\AA$, a contradiction. Therefore $Q$ is admissible. In particular, if an arrow $n-2\to\ell$ belongs to $Q$, then $n-2$ has at least two outgoing arrows, since in configurations $D_1$ and $D_4$ no arrow starts at $n-2$ and ends at a leaf.

Next we show that $\P=\P_{Q^F}$. Every arrow $x\to y$ of $Q^F$ satisfies $x\prec y$ in $\P$, so every path of $Q^F$ of length at least one joins two distinct points of $\P$, and $Q^F$ is acyclic, which is condition (iii) of Definition \ref{def:QF}. Moreover every relation of $\P_{Q^F}$ is a relation of $\P$. Conversely, if $x\prec y$ in $\P$, there is a chain $x=y_0\prec y_1\prec\cdots\prec y_k=y$ in which $y_i$ covers $y_{i-1}$ for every $i$, and each $y_{i-1}\to y_i$ is an arrow of $Q$ or of $F$. Hence $x\preceq y$ in $\P_{Q^F}$.

We show that the restriction of $F$ to the $\AA$-subquiver $S=Q\setminus\{\ell\}$ is an alien set for $S$, by means of Lemma \ref{lem:coverA}. A vertex of $S$ which is a sink of $S$ but not of $Q$ must have lost its only outgoing arrow, so it is $n-2$ and $n-2\to\ell$ is its only outgoing arrow in $Q$. By the previous paragraph this does not happen, so every sink of $S$ is a sink of $Q$, and (A$'$) follows from (2). Let $x$ be a vertex of $S$. There is $z\in\max\P$ with $x\preceq z$, and $z$ is a sink of $Q$ reachable from $x$ by (2). If $z=\ell$, the path from $x$ to $\ell$ passes through $n-2$, which then has an outgoing arrow not ending at $\ell$; following arrows from it we reach a sink $z'\neq\ell$ of $Q$, which is maximal in $\P$ by (2) and satisfies $x\preceq z'$. Hence there is always a vertex $z$ of $S$ which is maximal in $\P$ and satisfies $x\preceq z$. Moreover, if $z$ is a vertex of $S$, maximal in $\P$, with $x\preceq z$, then by (2) there is a path from $x$ to $z$ in $Q$, which does not pass through $\ell$, since $x,z\neq\ell$ and $\ell$ is not an inner vertex of any path. This proves (B$'$). Finally, an arrow of $Q$ whose endpoints are vertices of $S$ is an arrow of $S$, so the set $F_S$ of Lemma \ref{lem:coverA} is the restriction of $F$ to $S$, which is therefore an alien set for $S$.

We now show that no arrow of $F$ joins the two leaves. Suppose first that both leaves are sources of $Q$, and that $\ell\prec\ell'$ in $\P$. Let $S'$ be the quiver with the same vertices as $Q$, whose arrows are those of $Q$ except $\ell\to n-2$, together with a new arrow $\ell\to\ell'$. Its underlying graph is the path $\ell,\ell',n-2,n-3,\dots,1$, so $S'$ is a Dynkin quiver of type $\AA_n$, and its arrows are relations of $\P$. The vertex $\ell$ is not a sink of either quiver, and the arrows starting at any other vertex are the same in $Q$ and in $S'$, so both quivers have the same sinks, and for $x\neq\ell$ the paths starting at $x$ are the same. For $x=\ell$, the sinks reachable from $\ell$ in $Q$ and in $S'$ are, in both cases, the sinks reachable from $n-2$, because $\ell\to n-2$ is the only arrow of $Q$ starting at $\ell$, and in $S'$ the only arrow starting at $\ell$ is $\ell\to\ell'$ and the only arrow starting at $\ell'$ is $\ell'\to n-2$. Hence, by (2), conditions (A$'$) and (B$'$) of Lemma \ref{lem:coverA} hold for $S'$, whose vertices are all the points of $\P$, and $\P$ is of type $\AA$, a contradiction. Hence the leaves are incomparable in $\P$. Suppose now that a leaf, say $\ell'$, is a sink of $Q$. Then $\ell'$ is maximal in $\P$ by (2), so no arrow of $F$ starts at $\ell'$. If $\ell$ is also a sink, it is maximal as well and the leaves are incomparable. Otherwise $\ell\to n-2$ and $n-2\to\ell'$ are arrows of $Q$, so $\ell\prec n-2\prec\ell'$, and $\ell'$ does not cover $\ell$. In all cases no arrow of $F$ joins the two leaves, and $F$ satisfies condition (i) of Definition \ref{def:QF}.

Finally, we prove condition (ii). Suppose that $n-2$ receives exactly one arrow of $Q$, namely $\ell\to n-2$ for a leaf $\ell$, and that there is an arrow $s\to\ell$ in $F$. Since $n-2$ has valence three, it has two outgoing arrows, one of them to the other leaf $\ell'$, which is then a sink of $Q$ and hence maximal in $\P$ by (2). Since $s\prec\ell\prec n-2\prec\ell'$, condition (2) gives a path from $s$ to $\ell'$ in $Q$. This path passes through $n-2$, the only neighbour of $\ell'$, and hence through $\ell$, since $\ell\to n-2$ is the only arrow of $Q$ ending at $n-2$. As $\ell$ is a source of $Q$, the path starts at $\ell$, that is, $s=\ell$, which is impossible. Hence $F$ satisfies condition (ii). Hence $\P$ is associated with $Q^F$.
\end{proof}

In particular, every Dynkin quiver of a poset of type $\DD$ is admissible, and a finite connected poset which is not of type $\AA$ is of type $\DD$ if and only if it is associated with $Q^F$ for some $Q$ and $F$ as in Definition \ref{def:QF}. Condition (2) is the easiest one to check in practice, since it only involves the Hasse diagram of $\P$ and the paths of $Q$, whereas condition (3) is the one used in Section \ref{equi} to describe the socle-projective $\Bbbk\P$-modules by arcs.

\section{Categorical Equivalence} \label{equi}

In this section we prove that, for a poset $\P$ of type $\DD$, the category $\md_{sp}(\Bbbk\P)$ is equivalent to a full subcategory $(\mathcal{C}/T)_F$ of the arc category $\mathcal{C}/T$, where $T$ is a triangulation of the punctured polygon associated with the Dynkin quiver of $\P$ (Theorem \ref{main}).

Let $T$ be a triangulation of a punctured $n$-gon such that $Q_T$ is a Dynkin quiver of type $\DD_n$.  Following  \cite{BMR,Schiffler}, we have the equivalence of categories \begin{equation} \label{equivalence1}
\varphi_T:\mathcal{C}/T \to \textup{rep }Q_T.  
\end{equation}
We label the arcs of $T$ by $\tau_1,\dots,\tau_n$ so that $\tau_i$ corresponds to the vertex $i$ of $Q_T$, numbered as in Figure \ref{numberingD}. The dimension vector of $\varphi_T(M_{a,b}^\ze)$ is given by
\begin{equation} \label{eins}
  (\underline{\dim}\,\varphi_T(M_{a,b}^\ze))_i =
\dim\Hom_{\mathcal{C}}(\tau^{-1}\,\tau_i, M_{a,b}^\ze) =
\dim\Ext^1_{\mathcal{C}}(M_{a,b}^\ze,\tau_i) = e(M_ {a,b}^\ze,\tau_i),
\end{equation}
where $e(-,-)$ is the crossing number defined in Section \ref{section2}.

\begin{remark}\label{applications} By Gabriel's theorem, the isomorphism classes of indecomposable representations of a quiver of type $\DD_n$ are determined by their dimension vectors $\mathbf{d}=(d_1,\dots,d_n)$, so \eqref{eins} determines $\varphi_T(M_{a,b}^\ze)$. Following \cite{schiffler10}, we describe these vectors and representations, with the vertices numbered as in Figure~\ref{numberingD}.
 The entries $d_i$ of the dimension vector are either 0,1 or 2, and if we have $d_i=2$, then (1) $i$ is one of the vertices $2,3,\dots,n-2$, (2) for all vertices $j$ with $i\leq j\leq n-2$ we have $d_j=2$, and (3) $d_{i-1}\geq 1$ and $d_{n-1}=d_n=1$.  Thus, the vertices $i$ with $d_i=2$ form a subgraph of type $\mathbb{A}$ that contains the vertex $n-2$. The vertices $i\neq n-1,n$ with $d_i=1$ also form a subgraph of type $\mathbb{A}$, and if $d_j\neq 2$ for all $j$, then all the vertices $i$ with $d_i=1$ form a subgraph of type $\mathbb{A}$ or a subgraph of type $\DD$.  The corresponding representation is $M=(M_i,\varphi_\alpha)$ with $M_i=\Bbbk^{d_i}$; and $\varphi_\alpha=1$ if $d_{s(\alpha)}=d_{t(\alpha)}$, $\varphi_\alpha=0$ if one of  $d_{s(\alpha)},d_{t(\alpha)}$ is zero. If one of the  $d_i$ is 2, then there are exactly three arrows that connect a vertex with dimension 1 to a vertex with dimension 2. Two of these arrows, let us call them $\beta_1,\beta_2,$ connect the vertex $n-2$ with the vertices $n-1$ and $n$, the vector space of dimension two being at $n-2$, while the third arrow $\alpha_i$ connects two vertices $i$ and $i+1$, the vector space of dimension two being at vertex $i+1$. Consider the one-dimensional subspace  of $M_{i+1}$ given by 

\[
     U_1=\begin{cases}
       \text{Im }\varphi_{\alpha_i}&\quad\text{if }\alpha_i\text{ points to }i+1,\\ \text{ker }\varphi_{\alpha_i} &\quad\text{otherwise.} 
     \end{cases}
\]
Since the maps between spaces of dimension two are identities, we also regard it as a subspace of $M_{n-2}$. Consider also the following two one-dimensional subspaces $U_2$ and $U_3$ of $M_{n-2}$:
\[
  U_2=\begin{cases}
       \text{Im }\varphi_{\beta_1}&\quad\text{if }\beta_1\text{ points to }n-2,\\ \text{ker }\varphi_{\beta_1} &\quad\text{otherwise;} 
     \end{cases}
\]
and 
\[
  U_3=\begin{cases}
       \text{Im }\varphi_{\beta_2}&\quad\text{if }\beta_2\text{ points to }n-2,\\ \text{ker }\varphi_{\beta_2} &\quad\text{otherwise.} 
     \end{cases}
\]
The representation $M$ is indecomposable if and only if these three one-dimensional subspaces are pairwise distinct, and in that case it is unique up to isomorphism. \end{remark}

Let $\P$ be a poset of type $\mathbb{D}$ associated with the quiver $Q^{F}$ as in Definition \ref{def:QF}, and let $T$ be a triangulation of the punctured $n$-gon with $Q_T=Q$.

\begin{definition}\label{definitionfan} A \textit{fan} is a subset $\Sigma$ of the triangulation $T$ of at least two arcs that satisfies one of the following conditions and is maximal, with respect to inclusion, among the subsets of $T$ satisfying one of them:
\begin{enumerate}
    \item [(a)] all of the arcs in $\Sigma$ share exactly one vertex $v$ of the polygon, and $\Sigma$ contains at most one of the two tagged arcs $M^{1}_{v,v}$, $M^{-1}_{v,v}$, or
    \item [(b)] the arcs form a  triangle as in Figure \ref{trianglefan}.

\end{enumerate}

\begin{figure}[h!]
\begin{center}
\begin{tikzpicture}[y=.3cm, x=.3cm,font=\normalsize, scale=1.5]
\draw (0,0) circle (1cm);
\filldraw (0,0) circle (1pt);
\filldraw (-1.4,-3) circle (1pt);
\filldraw (1.4,-3) circle (1pt);
\node (tau3) at (-0.7,-1.6) {$_{\tau_{3}}$};
\node (tau1) at (0.7,-1.6) {$_{\tau_{1}}$};
\node (tau2) at (0.0,0.6) {$_{\tau_{2}}$};
\draw (-1.4,-3) .. controls (-1.25,-1.8) and (-1,0.3) .. (-0.3,0.6);
\draw (1.4,-3) .. controls (1.25,-1.8) and (1,0.3) .. (0.3,0.6);
\draw[-, >=latex,black](0,0) -- (-0.65, -1.4);
\draw[-, >=latex,black](-0.82, -1.79) -- (-1.4,-3);
\draw[-, >=latex,black](0,0) -- (0.65, -1.4);
\draw[-, >=latex,black](0.82, -1.79) -- (1.4,-3);

\end{tikzpicture}

\end{center}
   
    \caption{Triangle with a vertex at the puncture.}
    \label{trianglefan}
\end{figure}

\end{definition} 
\begin{remark}
If $T$ contains both $M^{1}_{v,v}$ and $M^{-1}_{v,v}$, Definition \ref{definitionfan} gives two fans at $v$, each containing exactly one of them. For two arcs $\tau,\tau'$ of a fan we write $\tau\leq\tau'$ if $\tau'$ is obtained from $\tau$ by rotating clockwise about their common vertex. With this order every fan $\Sigma$ is a chain, and in case (b) of Definition \ref{definitionfan} the chain is $\tau_1\leq\tau_2\leq\tau_3$. The \textit{peak-arc} $\tau_\Sigma$ is the maximum element of $\Sigma$.
\end{remark}

\begin{example} \label{examplequiverD}
Let $Q$ be the quiver of type $\DD_{7}$ on the left, and let $T$ be the triangulation associated with $Q$, on the right.
\[
\begin{array}{c@{\qquad\qquad}c}
\begin{tikzcd}[ampersand replacement=\&, row sep=small, column sep=small, baseline=(current bounding box.center)]
\&\&\&\&6\arrow[d]\&\\
1\arrow[r]\&2\&3\arrow[l]\&4\arrow[l]\&5\arrow[l]\arrow[r]\&7
\end{tikzcd}
&
\xy/r4.5pc/: {\xypolygon7"A"{~<{}~>{-}{\scriptstyle\bullet}}},
*+{\scriptstyle\bullet},
\POS"A7"\drop{\begin{array}{c}\\  \end{array}  }
 \POS"A3"\drop{\begin{array}{c}  \\ \\ \end{array}  } 
\POS"A4"\drop{\begin{array}{c}  \end{array}  }
\POS"A1"\drop{\begin{array}{c} \end{array}  }  
\POS"A2"\drop{\begin{array}{c}\\ \\  \end{array}  }  
\POS"A6"\drop{\begin{array}{c}\\ \\  \end{array}  }  
\POS"A5"\drop{\begin{array}{c}\qquad \end{array}  }  
\POS"A2" \ar@{-} @/^0ex/ |-{\tau_{7}}  "A0"
\POS"A3" \ar@{-} @/^0ex/ |-{\tau_{6}} "A0"
\POS"A6" \ar@{-} @/^2ex/|-{\tau_{1}} "A4"
\POS"A3" \ar@{-} @/_5ex/|-{\tau_{2}} "A6"
\POS"A3" \ar@{-} @/_8ex/|-{\tau_{3}} "A7",
\POS"A3" \ar@{-} @/_11ex/ |-{\tau_{4}} "A1",
\POS"A3" \ar@{-} @/_13ex/|-{\tau_{5}} "A2",
\endxy
\end{array}
\]
This quiver is admissible, since its branch vertex $5$ receives exactly one arrow (configuration $D_2$). The fans of $T$ are $\Sigma_1=\{\tau_1,\tau_2\}$ and $\Sigma_2=\{\tau_2,\tau_3,\tau_4,\tau_5,\tau_6\}$, formed by all the arcs sharing a vertex of the polygon, and $\Sigma_3=\{\tau_5,\tau_6,\tau_7\}$, which forms a triangle as in Figure \ref{trianglefan} (see Figure \ref{examplesfans}). Their chains are $\tau_1\leq\tau_2$, $\tau_6\leq\tau_5\leq\tau_4\leq\tau_3\leq\tau_2$ and $\tau_6\leq\tau_5\leq\tau_7$, which correspond to the paths $1\to2$, $6\to5\to4\to3\to2$ and $6\to5\to7$ of $Q$. Thus the peak-arc of $\Sigma_1$ and of $\Sigma_2$ is $\tau_2$, and that of $\Sigma_3$ is $\tau_7$, corresponding to the sinks $2$ and $7$. Two fans may share several arcs, for example $\Sigma_2\cap\Sigma_3=\{\tau_5,\tau_6\}$.

\begin{figure}[h!]
\[\begin{array}{ccc}
\xy/r3.6pc/:{\xypolygon7"A"{~<{}~>{-}{\scriptstyle \bullet}}}
*+{\scriptstyle\bullet},
\POS"A6" \ar@{-} @/^1.6ex/|-{\tau_{1}} "A4"
\POS"A3" \ar@{-} @/_4ex/|-{\tau_{2}}@[blue] "A6"
\endxy
\quad 
&\quad
\xy/r3.6pc/:{\xypolygon7"A"{~<{}~>{-}{\scriptstyle \bullet}}}
*+{\scriptstyle\bullet},
\POS"A3" \ar@{-} @/_4ex/|-{\tau_{2}}@[blue] "A6"
\POS"A3" \ar@{-} @/_6.4ex/|-{\tau_{3}} "A7",
\POS"A3" \ar@{-} @/_8.8ex/ |-{\tau_{4}} "A1",
\POS"A3" \ar@{-} @/_10.4ex/|-{\tau_{5}} "A2",
\POS"A3" \ar@{-} @/^0ex/ |-{\tau_{6}} "A0"
\endxy 
\quad 
&\quad
\xy/r3.6pc/:{\xypolygon7"A"{~<{}~>{-}{\scriptstyle \bullet}}}
*+{\scriptstyle\bullet},
\POS"A2" \ar@{-} @/^0ex/ |-{\tau_{7}} @[blue] "A0"
\POS"A3" \ar@{-} @/^0ex/ |-{\tau_{6}} "A0"
\POS"A3" \ar@{-} @/_10.4ex/|-{\tau_{5}} "A2",
\endxy 
\end{array}\]
\caption{The fans $\Sigma_{1}$ (left polygon), $\Sigma_{2}$ (center polygon) and $\Sigma_{3}$ (right polygon) of $T$.}\label{examplesfans}
\end{figure}
\end{example}

For an arc $\tau\in T$, we denote by $\tau_{\vartriangle}$ the set of all arcs $\tau'$ in $T$ such that $\tau'\leq\tau$ with $\tau$ and $\tau'$ belonging to the same fan.

Recall from Section \ref{subsec:arcs} that each arrow of $Q_T$ comes from a triangle of $T$, as an arrow between two of its sides that lie in $T$ and share a corner. In each triangle, rotating clockwise about each corner takes a side to the next one, so the arrows between its sides follow the same cyclic direction; afterwards, the $2$-cycles are deleted. A $2$-cycle between two arcs $\sigma,\tau$ can only appear if $\sigma$ and $\tau$ are sides of two common triangles, that is, if they meet at two corners.

\begin{lemma}\label{lem:fans}
Let $T$ be a triangulation of a punctured $n$-gon, $n\ge4$, such that $Q_T$ is a Dynkin quiver of type $\DD_n$. Then:
\begin{enumerate}
\item[(1)] the fans of $T$, ordered by $\le$, are exactly the maximal paths of $Q_T$. In particular, the peak-arc of a fan is a sink, and the peak-arcs of the fans containing $\tau_x$ are the $\tau_z$ with $z$ a sink reachable from $x$;
\item[(2)] $\tau_x\in(\tau_{n-2})_\vartriangle$ if and only if there is a path from $x$ to $n-2$ in $Q_T$.
\end{enumerate}
\end{lemma}

\begin{proof}
We begin by describing the arcs of $T$ at the puncture. An arc $M_{c,d}$ with $c\neq d$ separates the punctured polygon into the region bounded by $M_{c,d}$ and the boundary path $\delta_{c,d}$, which does not contain the puncture, and its complement. Hence the arc $M_{a,a}^{\epsilon}$, which joins $a$ with the puncture, crosses $M_{c,d}$ if and only if $a$ is an interior vertex of $\delta_{c,d}$, and this does not depend on the tag $\epsilon$. Moreover, by Section \ref{subsec:arcs}, two arcs $M_{a,a}^{\epsilon}$ and $M_{c,c}^{\epsilon'}$ do not cross if and only if $a=c$ or $\epsilon=\epsilon'$.

The triangulation $T$ contains at least two arcs at the puncture. Indeed, if it contained none, the puncture would lie in a region of the complement of $T$ whose vertices are vertices of the polygon, and the arc joining one of them with the puncture would not cross any arc of $T$. If it contained only one, $M_{a,a}^{\epsilon}$, the arc $M_{a,a}^{-\epsilon}$ would not cross any arc of $T$. Both cases contradict the maximality of $T$. By the crossing rules above, the arcs of $T$ at the puncture are either the tagged pair $M_{v,v}^{1}$, $M_{v,v}^{-1}$ at a vertex $v$, or arcs $M_{a_1,a_1}^{\epsilon},\dots,M_{a_k,a_k}^{\epsilon}$ at distinct vertices with the same tag and $k\ge2$.

Suppose first that $T$ does not contain a tagged pair. Two consecutive arcs around the puncture are sides of the triangle with a corner at the puncture lying between them, and the corresponding arrows all follow the same direction of rotation. If $k\ge3$, two of these arcs meet at a single corner, so none of these arrows cancels and they form an oriented cycle of length $k$, contradicting that $Q_T$ is acyclic. Thus $k=2$, and we call $\lambda$ and $\lambda'$ the two arcs, with endpoints $a\neq b$ on the polygon. They cut the polygon into two sectors, and the triangle of each sector with a corner at the puncture has sides $\lambda$, $\lambda'$ and a third side joining $a$ and $b$. The two arrows between $\lambda$ and $\lambda'$ have opposite directions and cancel. If both third sides were arcs of $T$, the remaining arrows of the two triangles would form an oriented cycle of length four, contradicting that $Q_T$ is acyclic. Since $n\ge4$, they cannot both be boundary edges. Hence one of them is a boundary edge $\delta$, so that $a$ and $b$ are neighbours, and the other is an arc $\beta\in T$. In $Q_T$, the only neighbour of $\lambda$ and of $\lambda'$ is $\beta$, and the arrows are $\lambda'\to\beta\to\lambda$ or $\lambda\to\beta\to\lambda'$.

If, on the contrary, $T$ contains the tagged pair at $v$, we write $\lambda=M_{v,v}^{1}$ and $\lambda'=M_{v,v}^{-1}$. By the convention of Section \ref{subsec:arcs}, each of them is treated as a loop at $v$ around the puncture, and the triangle containing that loop has two other sides joining $v$ with a vertex $w$. If both were arcs of $T$, each of $\lambda,\lambda'$ would form with them an oriented cycle of length three which does not cancel, since those two sides meet at a single corner of the triangle, contradicting that $Q_T$ is acyclic. Since $n\ge4$, they cannot both be boundary edges. Hence one of them is a boundary edge $\delta=vw$ and the other is an arc $\beta\in T$. In $Q_T$, the only neighbour of $\lambda$ and of $\lambda'$ is $\beta$, and the two arrows have the same direction.

In both cases, $\beta$ and $\delta$ bound a region $D$ containing the puncture and $\lambda,\lambda'$, and the remaining arcs of $T$ triangulate the unpunctured $n$-gon $P'$ whose sides are $\beta$ and the boundary edges other than $\delta$. Since $\lambda$ and $\lambda'$ are leaves of $Q_T$ adjacent to $\beta$, and in a quiver of type $\DD_n$ the only vertex adjacent to two leaves is the branch vertex (for $n=4$, because each leaf has a single neighbour), the arc $\beta$ is the branch vertex $\tau_{n-2}$. Moreover, around its endpoint on the polygon, each of $\lambda,\lambda'$ lies between $\beta$ and the edge $\delta$. Finally, since two distinct sides of the triangulation of $P'$ do not join the same two vertices and $\lambda$ and $\lambda'$ have a single endpoint on the polygon, no pair of arcs of $T$ other than $\{\lambda,\lambda'\}$ meets at two corners. Hence the only $2$-cycle that cancels in $Q_T$ is the one formed by $\lambda$ and $\lambda'$ when $T$ does not contain a tagged pair, and these two arcs share no vertex of the polygon. When $T$ contains the tagged pair, there are no arrows between $\lambda$ and $\lambda'$, by the convention of Section \ref{subsec:arcs}.

It follows that every triangle of $T$ contained in $P'$ has at most two sides in $T$. Indeed, if it had three, its three arrows would form an oriented cycle, which does not cancel because none of its sides is $\lambda$ or $\lambda'$, contradicting that $Q_T$ is acyclic.

We now describe the arrows of $Q_T$. For a vertex $u$ of the polygon, let $\sigma_1,\dots,\sigma_k$ be the arcs of $T$ incident with $u$, in the order in which they appear when rotating clockwise about $u$, from one of the boundary edges incident with $u$ to the other. If $T$ contains the tagged pair at $u$, its two arcs occupy the same position and we take one of them. Two consecutive arcs $\sigma_i,\sigma_{i+1}$ are sides of the triangle with corner $u$ lying between them, so $\sigma_i\to\sigma_{i+1}$ is an arrow of $Q_T$. This arrow does not cancel, because $\sigma_i$ and $\sigma_{i+1}$ share the vertex $u$ of the polygon, whereas $\lambda$ and $\lambda'$ share none. Every arrow of $Q_T$ is of this form, since the only corners that are not vertices of the polygon lie at the puncture, where the arrows cancel if $T$ does not contain a tagged pair, and do not exist, by the convention of Section \ref{subsec:arcs}, if it does.

Next we show that, if $\sigma\to\tau\to\upsilon$ is a path of $Q_T$ with $\{\sigma,\upsilon\}\neq\{\lambda,\lambda'\}$, then
\begin{equation}\label{eq:consec}
\sigma,\tau,\upsilon \text{ are consecutive arcs around a vertex of the polygon.}
\end{equation}
Since $\lambda$ and $\lambda'$ are leaves, $\tau$ is not one of them. Let $u$ be the corner of the arrow $\sigma\to\tau$ and $u'$ that of the arrow $\tau\to\upsilon$. The arrow $\sigma\to\tau$ comes from the triangle preceding $\tau$ when rotating clockwise about $u$, and the arrow $\tau\to\upsilon$ comes from the triangle following $\tau$ about $u'$. If $u'$ is the other endpoint of $\tau$, the triangle following $\tau$ about $u'$ is the one preceding it about $u$, since the direction of rotation is reversed between the two endpoints of $\tau$. That triangle would then have the three sides $\sigma,\tau,\upsilon$ in $T$. As we have seen, this only happens in $D$, and the only triangle of $D$ with three sides in $T$ is the one with sides $\lambda,\lambda',\beta$, when $T$ does not contain a tagged pair. Since $\tau$ is not a leaf, we would have $\tau=\beta$ and $\{\sigma,\upsilon\}=\{\lambda,\lambda'\}$, contrary to our assumption. Hence $u'=u$, and $\sigma,\tau,\upsilon$ are consecutive around $u$.

We now prove the first statement of (1). Let $\Sigma$ be a fan. If it is of type (b), it is the triangle with sides $\lambda,\beta,\lambda'$, which is a path of $Q_T$ whose endpoints are the leaves $\lambda,\lambda'$, and hence a maximal path. If it is of type (a) at a vertex $u$, every set satisfying (a) with common vertex $u$ is contained in the set $A_u$ of all arcs of $T$ incident with $u$, taking only one arc of the tagged pair if $T$ contains it at $u$, and $A_u$ also satisfies (a). By the maximality of $\Sigma$, we have $\Sigma=A_u=\{\sigma_1,\dots,\sigma_k\}$, and by the above description of the arrows of $Q_T$ it is a path $\sigma_1\to\cdots\to\sigma_k$ ordered by $\le$. Suppose that this path can be extended, say by $\sigma_k\to\upsilon$. If $\{\sigma_{k-1},\upsilon\}\neq\{\lambda,\lambda'\}$, property \eqref{eq:consec}, applied to $\sigma_{k-1}\to\sigma_k\to\upsilon$, gives that $\upsilon$ follows $\sigma_k$ around $u$, contradicting that $\sigma_k$ is the last arc at $u$. If $\{\sigma_{k-1},\upsilon\}=\{\lambda,\lambda'\}$, then $\sigma_k=\beta$ and $\sigma_{k-1}\in\{\lambda,\lambda'\}$. Since $\lambda$ and $\lambda'$ lie between $\beta$ and the edge $\delta$ around their endpoint, they are the first or the last arc at $u$, so $k=2$ and $\Sigma=\{\sigma_1,\beta\}$ is properly contained in the fan of type (b). This contradicts that $\Sigma$ is maximal among all subsets satisfying (a) or (b), as required in Definition \ref{definitionfan}; note that $\Sigma$ is maximal among the subsets satisfying (a). The same holds at the beginning of the path. Thus $\Sigma$ is a maximal path.

Conversely, let $\pi$ be a maximal path of $Q_T$. If $\pi$ contains $\lambda'\to\beta\to\lambda$ or $\lambda\to\beta\to\lambda'$, then $\pi$ is that path, since $\lambda$ and $\lambda'$ are leaves, and it is the fan of type (b). Otherwise, by \eqref{eq:consec}, every three consecutive vertices of $\pi$ are consecutive arcs around a vertex of the polygon, and two triples overlapping in two arcs determine the same vertex, because two arcs of $T$ joined by an arrow share exactly one vertex of the polygon. The same trivially holds if $\pi$ has two vertices. Hence $\pi$ consists of consecutive arcs incident with a vertex $u$, which satisfy condition (a), so $\pi$ is contained in a fan. Here $\pi$ does not contain both arcs of a tagged pair, since both are adjacent only to $\beta$, with arrows in the same direction. Since that fan is a path and $\pi$ is maximal, they coincide.

For the second statement of (1), the peak-arc of a fan is the last vertex of a maximal path, that is, a sink. Moreover, every path from $x$ to a sink $z$ extends backwards to a maximal path ending at $z$. For (2), by definition $\tau_x\in(\tau_{n-2})_\vartriangle$ if and only if $\tau_x\le\tau_{n-2}$ in a common fan, that is, if $x$ precedes $n-2$ in a maximal path, and every path from $x$ to $n-2$ extends to a maximal path.
\end{proof}

Let $\Sigma$ be a fan containing an arc $\tau$ which is not its peak-arc. The \textit{successor} of $\tau$ in $\Sigma$ is the arc of $\Sigma$ that immediately follows $\tau$ with respect to $\le$. We say that the arc $\tau_{n-2}$ \textit{splits} if there are two fans containing $\tau_{n-2}$ in which $\tau_{n-2}$ has two distinct successors, and we write
\[
T_{\Box}=\begin{cases}(\tau_{n-2})_{\vartriangle}&\text{if }\tau_{n-2}\text{ splits},\\ \varnothing&\text{otherwise.}\end{cases}
\]

\begin{definition} \label{definitionarc}
An arc $\gamma$ not in $T$ is said to be \textit{$\star$-arc} if it satisfies the following conditions:
\begin{enumerate}
\item [(a)] For each arc $\tau$ in $T \setminus T_{\Box}$, if $e(\gamma,\tau) \neq 0$, there exists a fan $\Sigma$ containing $\tau$  such that $e(\gamma, \tau) \leq e(\gamma, \tau_\Sigma)$.
\item [(b)]  For each arc $\tau$ in $T_{\Box}$,  if $e(\gamma,\tau) \neq 0$,  there exists a fan  $\Sigma$ containing $\tau$ such that $e(\gamma, \tau_\Sigma)\neq 0$.
\end{enumerate}
\end{definition}

\begin{remark}\label{obs:star}
Let $T$ be as in Lemma \ref{lem:fans}, let $\gamma\notin T$ be an arc and write $e_x=e(\gamma,\tau_x)$. Let $B$ be the set of vertices $x$ of $Q_T$ such that there is a path from $x$ to $n-2$, if $n-2$ has at least two outgoing arrows, and let $B=\varnothing$ otherwise. Then $\gamma$ is a $\star$-arc if and only if, for every vertex $x$ with $e_x\neq0$,
\begin{equation}\label{eq:star}
\text{there is a sink } z \text{ reachable from } x \text{ with } e_z\neq0, \text{ and with } e_z\ge e_x \text{ if } x\notin B.
\end{equation}
Indeed, by Lemma \ref{lem:fans}(1), the fans containing $\tau_{n-2}$ are the maximal paths of $Q_T$ through $n-2$, and the successor of $\tau_{n-2}$ in such a fan is the target of the arrow of the path that starts at $n-2$. Every arrow starting at $n-2$ lies in a maximal path, so $\tau_{n-2}$ splits if and only if $n-2$ has at least two outgoing arrows. Together with Lemma \ref{lem:fans}(2), this shows that $\tau_x\in T_\Box$ if and only if $x\in B$. By Lemma \ref{lem:fans}(1), the peak-arcs of the fans containing $\tau_x$ are the $\tau_z$ with $z$ a sink reachable from $x$. Hence condition (b) of Definition \ref{definitionarc} applies exactly when $x\in B$, and asks for $e_z\neq0$; condition (a) applies when $x\notin B$, and asks for $e_z\ge e_x$, which implies $e_z\neq0$.

If $Q_T$ is admissible, the arc $\tau_{n-2}$ splits exactly in configurations $D_2$ and $D_3$. In configuration $D_1$ the set $T_\Box$ is empty, but replacing it by $(\tau_{n-2})_\vartriangle$ does not change the class of $\star$-arcs. Indeed, let $x$ be a vertex with a path to the sink $n-2$ and $e_x\neq0$. If $e_x=1$, the conditions $e_z\neq0$ and $e_z\ge e_x$ coincide. If $e_x=2$, then $e_{n-2}=2$ by Remark \ref{applications}, so the sink $z=n-2$ satisfies both conditions. In configuration $D_4$, on the other hand, $T_\Box$ is empty and condition (b) cannot be imposed on $(\tau_{n-2})_\vartriangle$ (Remark \ref{rem:admisib}).
\end{remark}

\begin{example} \label{examplequiverD1} Let $T$ be the triangulation of Example \ref{examplequiverD}, so that $(\tau_5)_\vartriangle=\{\tau_5,\tau_6\}$, and consider the arcs of Figure \ref{examplestararcs}. The arc $\gamma_1$ is not a $\star$-arc, since it crosses $\tau_3$ twice and $\tau_2$ only once, and $\Sigma_2$ is the only fan containing $\tau_3$, so condition (a) of Definition \ref{definitionarc} fails for $\tau=\tau_3$. Similarly, $\gamma_2$ crosses $\tau_4$ but not $\tau_2$, and condition (a) fails for $\tau=\tau_4$. On the other hand, $\gamma_3$ is a $\star$-arc. Indeed, $\gamma_3$ crosses $\tau_5$ twice and each of the other arcs of $T$ once. Condition (a) holds because the peak-arcs $\tau_2$ and $\tau_7$ cross $\gamma_3$ once, and condition (b) holds for $\tau_5$ and $\tau_6$ because $e(\gamma_3,\tau_2)\neq0$.
\end{example}

\begin{figure}[h!]
\[\begin{array}{ccc}
\xy/r3.6pc/:{\xypolygon7"A"{~<{}~>{-}{\scriptstyle \bullet}}}
*+{\scriptstyle\bullet},
\POS"A6" \ar@{-} @/_8.8ex/ |-{\gamma_{1}} "A4"
\endxy
\quad 
&\quad
\xy/r3.6pc/:{\xypolygon7"A"{~<{}~>{-}{\scriptstyle \bullet}}}
*+{\scriptstyle\bullet},
\POS"A7" \ar@{-}@/^10.4ex/ |-{\gamma_{2}} "A1",
\endxy 
\quad 
&\quad
\xy/r3.6pc/:{\xypolygon7"A"{~<{}~>{-}{\scriptstyle \bullet}}}
*+{\scriptstyle\bullet},
\POS"A5" \ar@{-}@/^6.4ex/ |-{\gamma_{3}} @[blue] "A1",
\endxy 
\end{array}\]
\caption{Arcs in a $7$-gon}\label{examplestararcs}
\end{figure}

\begin{definition} Given an alien arrow $\alpha\in F$, we say that an arc $\gamma$ is \textit{frozen} by $\alpha$ if $\gamma$ crosses $\tau_{s(\alpha)}$ and there is a sink $z$ of $Q$ with $s(\alpha), t(\alpha) \in \Supp I(z)$ such that $e(\gamma,\tau_{t(\alpha)})<e(\gamma,\tau_{z})$.
\end{definition}

\begin{example} \label{examplealienarrow}

Consider the arrow $\alpha\colon 6\to1$ on the quiver $Q$ of Example \ref{examplequiverD}. The set $F=\{\alpha\}$ satisfies Definition \ref{def:QF}, since its restriction to the $\mathbb{A}$-subquiver $Q\setminus\{7\}$ is an alien set and its restriction to $Q\setminus\{6\}$ is empty. The only sink $z$ with $6,1\in\Supp I(z)$ is $z=2$, so an arc $\gamma$ is frozen by $\alpha$ if and only if it crosses $\tau_6$ and $e(\gamma,\tau_1)<e(\gamma,\tau_2)$. There are exactly six arcs frozen by $\alpha$ (see Figure \ref{examplefrozenarc}), and the arc $\gamma_1$ of Example \ref{examplequiverD1} is one of them.

\begin{figure}[h!]

\[\begin{array}{ccc}
\xy/r4.5pc/:{\xypolygon7"A"{~<{}~>{-}{\scriptstyle \bullet}}}
*+{\scriptstyle\bullet},
\POS"A4" \ar|-(0.5){\SelectTips{cm}{}\object@{+}}@{-}  @[blue]"A0"
\POS"A4" \ar@{-} @/^2ex/ @[blue]  "A2"
\POS"A1" \ar@{-} @/_5ex/ @[blue] "A4"
\POS"A7" \ar@{-} @/_8ex/ @[blue] "A4",
\POS"A6" \ar@{-} @/_11ex/  @[blue] "A4"
\POS"A4" \ar@{-}@/^13ex/ @[blue] "A5",
\endxy
\end{array}\]
\caption{The six arcs frozen by the alien arrow of Example \ref{examplealienarrow}.} \label{examplefrozenarc}
\end{figure}
\end{example}

\begin{definition} \label{sparc} An arc $\gamma\notin T$ is an \textit{sp-arc} if it is a $\star$-arc and it is not frozen by an alien arrow in $F$.
\end{definition}

The following proposition relates $\star$-arcs to socle-projective modules.

\begin{proposition}\label{propstar}
Let $T$ be a triangulation of a punctured $n$-gon such that $Q_T$ is a Dynkin quiver of type $\DD_n$. An arc $\gamma\notin T$ is a $\star$-arc if and only if the indecomposable representation $\varphi_T(\gamma)$ of $Q_T$ is socle-projective.
\end{proposition}

Note that Proposition \ref{propstar} does not require $Q_T$ to be admissible. Admissibility is used later, through Lemma \ref{lem:C}, to describe the sp-arcs.

\begin{proof}
Let $M_\gamma=\varphi_T(\gamma)=(M_i,\varphi_\alpha)$, with $M_i=\Bbbk^{e_i}$, $e_i=e(\gamma,\tau_i)$, and $\varphi_\alpha$ as in Remark \ref{applications}. The simple representation $S(x)$ is a subrepresentation of $M_\gamma$ if and only if there is an injective map $f_x\colon S(x)_x\to M_x$ such that, for every arrow $\alpha\colon x\to x'$, the diagram
\[ \xymatrix{ S(x)_{x} \ar[r]^{0} \ar[d]_{f_{x}} & 0 \ar[d]^{0}\\
\Bbbk^{e_x} \ar[r]^{ \varphi_{\alpha}}&   \Bbbk^{e_{x'}}} \]
commutes, that is, $\varphi_\alpha f_x=0$. Since $S(x)$ is projective exactly when $x$ is a sink, $M_\gamma$ is socle-projective if and only if this does not happen for any vertex $x$ that is not a sink. Here the orientation matters, and it is the one of Section \ref{subsec:arcs}. Indeed, by \eqref{equivalence1} and Yoneda's lemma, $\varphi_T(\tau^{-1}\tau_i)=\operatorname{Hom}_{\mathcal C}(\tau^{-1}T,\tau^{-1}\tau_i)$ is the indecomposable projective representation $P(i)$, whose dimension vector is $(e(\tau^{-1}\tau_i,\tau_j))_j$ by \eqref{eins}, and with the clockwise rule of Section \ref{subsec:arcs} its support is the set of vertices reachable from $i$ \cite{BMR,Schiffler}.

We will use the following facts. First, since $M_\gamma$ is indecomposable, its support is connected, as for every indecomposable representation of a quiver, so, since $Q_T$ is a tree, the path between two vertices of the support is contained in the support. Second, by Remark \ref{applications}, if some $e_i$ equals $2$, the vertices with $e_i=2$ form an interval $\{m,\dots,n-2\}$ of the chain, and their neighbours outside it are $m-1$, $n-1$ and $n$, with $e=1$. Third, also by Remark \ref{applications}, if $e_{s(\alpha)}=e_{t(\alpha)}$ the map $\varphi_\alpha$ is the identity; if $e_{s(\alpha)}=1$ and $e_{t(\alpha)}=2$ it is injective; if $e_{s(\alpha)}=2$ and $e_{t(\alpha)}=1$ it is surjective with one-dimensional kernel, and the arrows of this kind starting at the same vertex have pairwise distinct kernels, because the subspaces $U_1,U_2,U_3$ are pairwise distinct.

By Remark \ref{obs:star}, $\gamma$ is a $\star$-arc if and only if condition \eqref{eq:star} holds for every vertex $x$ with $e_x\neq0$, where $B$ is the set defined there. Recall that $B$ is empty unless $n-2$ has at least two outgoing arrows, and that in that case $B$ consists of the vertices with a path to $n-2$.

First, suppose that $\gamma$ is a $\star$-arc. Let $x$ be a vertex that is not a sink and let $f_x$ be as above. We show that this is impossible. Since $f_x$ is injective, $M_x\neq0$, that is, $e_x=1$ or $2$. Let $z$ be a sink as in \eqref{eq:star}. Since $x$ is not a sink, $z\neq x$, and the path $\rho$ from $x$ to $z$ is contained in the support. Let $\alpha\colon x\to x'$ be the first arrow of $\rho$, so that $e_{x'}>0$. We have the following three possibilities for $e_x$ and $e_{x'}$.

(i) If $e_x=e_{x'}$, then $\varphi_\alpha$ is the identity. Thus the commutativity of the diagram implies that $f_x=0$, a contradiction.

(ii) If $e_x=1$ and $e_{x'}=2$, then $\varphi_\alpha=\left(\begin{smallmatrix}a\\ b\end{smallmatrix}\right)\neq0$ and $f_x=c$, with $a,b,c\in\Bbbk$. Therefore the commutativity of the diagram implies that $c=0$, a contradiction.

(iii) Suppose that $e_x=2$ and $e_{x'}=1$. If $x\notin B$, then $e_z\ge e_x=2$, and the path $\rho$ joins the vertices $x$ and $z$, both with $e=2$, passing through $x'$, with $e_{x'}=1$, contradicting that the vertices with $e=2$ form an interval. Suppose now that $x\in B$, so that $n-2$ has at least two outgoing arrows and there is a path from $x$ to $n-2$. Since $e_x=2$, Remark \ref{applications} gives $e_{n-2}=2$. If $x\neq n-2$, the path from $x$ to $n-2$ joins two vertices with $e=2$ and is contained in that interval, so its first arrow $\alpha'\colon x\to y'$ satisfies $e_{y'}=2$. Then $\varphi_{\alpha'}$ is the identity, and the commutativity of the diagram corresponding to $\alpha'$ implies that $f_x=0$. Suppose finally that $x=n-2$. Since $x\in B$, the vertex $x$ has at least two outgoing arrows. Thus there are a vertex $y'\neq x'$ and an arrow $\alpha'\colon x\to y'$, with $e_{y'}\ge1$. If $e_{y'}=2$, we conclude as before. If $e_{y'}=1$, then $\ker\varphi_{\alpha}\cap\ker\varphi_{\alpha'}=0$, and the commutativity of the diagram
\[ \xymatrix{ 0\ar[d]_{0} & S(x)_x \ar[l]_{0}\ar[r]^0\ar[d]^{f_x} & 0 \ar[d]^0\\
\Bbbk^{e_{y'}} &\Bbbk^{e_x} \ar[r]^{ \varphi_{\alpha}}  \ar[l]_{ \varphi_{\alpha'}}&   \Bbbk^{e_{x'}}} \]
implies that $\operatorname{Im} f_x\subset\ker\varphi_{\alpha}\cap\ker\varphi_{\alpha'}$, and hence $f_x=0$, a contradiction.

Consequently, all simple subrepresentations of $M_\gamma$ are projective, that is, $M_\gamma$ is socle-projective.

In the other direction, suppose that $M_\gamma$ is socle-projective but $\gamma$ is not a $\star$-arc. By \eqref{eq:star}, there is a vertex $x$ with $e_x\neq0$ such that every sink $z$ reachable from $x$ satisfies $e_z=0$, if $x\in B$, or $e_z<e_x$, if $x\notin B$. In the second case we use that $e_z\ge e_x$ already implies $e_z\neq0$. We first show that every vertex $y$ that is not a sink and has $e_y\neq0$ satisfies the following two conditions.
\begin{itemize}
\item[(a)] There is an arrow $y\to y'$ with $e_{y'}\neq0$.
\item[(b)] If $e_y=2$, there is an arrow $y\to y'$ with $e_{y'}=2$, or at least two arrows start at $y$.
\end{itemize}
Indeed, otherwise $S(y)$ would be a non-projective simple subrepresentation of $M_\gamma$. If all arrows starting at $y$ end at vertices with $e=0$, the corresponding maps are zero, and every injective map $f_y$ makes the diagrams commute. If $e_y=2$ and (b) fails, a single arrow $\alpha_y\colon y\to y'$ starts at $y$, and $e_{y'}=1$, since $e_{y'}\neq2$ and every neighbour of a vertex of dimension $2$ has dimension at least $1$. We choose an injective map $f_y$ such that $\operatorname{Im} f_y=\ker\varphi_{\alpha_y}$. This is always possible because the matrix associated to $\varphi_{\alpha_y}$ is a nonzero matrix of the form $\left(\begin{smallmatrix}a & b\end{smallmatrix}\right)$, and the matrix associated to $f_y$ is then $\left(\begin{smallmatrix}-b\\ a\end{smallmatrix}\right)$.

We now construct a path $x=x_0\to x_1\to\cdots$ in the support. If $x_j$ is not a sink, by (a) there is an arrow $x_j\to y$ with $e_y\neq0$, and we choose $x_{j+1}$ among these vertices, with $e_{x_{j+1}}=2$ if possible. Since $Q_T$ is acyclic, the path ends at a sink $z$ reachable from $x$ with $e_z\neq0$. This contradicts the choice of $x$ if $x\in B$ or if $e_x=1$, since in the latter case $e_z\ge1=e_x$. Suppose then that $e_x=2$ and that $x\notin B$. We show by induction that $e_{x_j}=2$ for all $j$, so that $e_z=2=e_x$, another contradiction. Suppose that $e_{x_j}=2$ and that $x_j$ is not a sink. Every neighbour of $x_j$ satisfies $e\ge1$, so, by (b) and by the choice of $x_{j+1}$, either $e_{x_{j+1}}=2$, or at least two arrows start at $x_j$ towards vertices with $e=1$. In the second case, since the only vertices with $e=1$ adjacent to the interval $\{m,\dots,n-2\}$ are $m-1$, $n-1$ and $n$, and the only vertex of the interval adjacent to two of them is $n-2$, we would have $x_j=n-2$. Then $n-2$ has at least two outgoing arrows, and $x_0\to\cdots\to x_j$ is a path from $x$ to $n-2$, so $x\in B$, a contradiction.

Therefore, $\gamma$ is a $\star$-arc.
\end{proof}

\begin{remark}\label{rem:admisib} The exception for the vertices of $B$ in \eqref{eq:star} is only needed in case (iii) with $x=n-2$, where the two outgoing arrows of $n-2$ have distinct kernels. It cannot be extended to all the vertices with a path to $n-2$ when $n-2$ has a single outgoing arrow. For example, let $Q_T$ be the quiver of type $\DD_5$ given by $1\to2\leftarrow3$, $4\to3$ and $5\to3$, which is in configuration $D_4$, and let $\gamma$ be the arc with crossing vector $(0,1,2,1,1)$. The vertex $3$ has a path to $n-2=3$, and the sink $2$ satisfies $e_2\neq0$, but $e_2=1<2=e_3$. Since $B=\varnothing$, condition \eqref{eq:star} fails for $x=3$, so $\gamma$ is not a $\star$-arc. Accordingly, $\varphi_T(\gamma)$ is not socle-projective. Indeed, the only arrow starting at $3$ is $3\to2$, whose map $\Bbbk^2\to\Bbbk$ has a one-dimensional kernel, and this kernel spans a simple subrepresentation isomorphic to $S(3)$, which is not projective.
\end{remark}

\begin{remark}\label{rem:D4ex}
For the quiver $Q$ of type $\DD_6$ on the right of Example \ref{ex:D4}, which is in configuration $D_4$, condition \eqref{eq:star} reads as follows. Here $n-2=4$ has a single outgoing arrow, so $B=\varnothing$, and the sinks are $1$ and $3$, with $U(2)=\{1,3\}$ and $U(x)=\{3\}$ for $x=3,4,5,6$. Hence an arc $\gamma\notin T$ is a $\star$-arc if and only if $e_3\ge e_x$ for every $x\in\{4,5,6\}$ with $e_x\neq0$, and $e_1\ge e_2$ or $e_3\ge e_2$ if $e_2\neq0$; the condition for $x=1,3$ is trivial. We make the first condition explicit. Since the leaves satisfy $e_5,e_6\le1$, for $x=5,6$ it only asks for $e_3\neq0$. For $x=4$, if $e_4=1$ it also asks for $e_3\neq0$, and if $e_4=2$ it asks for $e_3=2$. Since a neighbour of a vertex with $e=2$ has $e\ge1$ by Remark \ref{applications}, the first condition excludes exactly the arcs with $e_3=0$ and $e_x\neq0$ for some $x\in\{4,5,6\}$, and the arcs with $e_4=2$ and $e_3=1$.
\end{remark}

\subsection{The Category $(\mathcal{C}/T)_{F}$} \label{categorysp} Let $Q$ be an admissible quiver of type $\DD$, let $F$ be a set of alien arrows over $Q$ as in Definition \ref{def:QF}, and let $T$ be the triangulation associated to $Q$. We define $(\mathcal{C}/T)_{F}$ as the full subcategory of $\mathcal{C}/T$ whose objects are the direct sums of sp-arcs. By \eqref{equivalence1}, for sp-arcs $\gamma,\gamma'$ we have
\[
\Hom_{(\mathcal{C}/T)_{F}}(\gamma,\gamma')=\Hom_{\mathcal{C}/T}(\gamma,\gamma')\cong\Hom_{\rep Q_T}\bigl(\varphi_T(\gamma),\varphi_T(\gamma')\bigr).
\]
As in \cite{schifflerserna,Schiffler}, these morphism spaces are spanned by compositions of elementary moves modulo the mesh relations of $\mathcal{C}$ and the morphisms factoring through $T$.

\subsection{The Functor $\Theta$} Let $\P$ be the poset of type $\DD$ associated with the quiver $Q^{F}$, where $Q$ is an admissible Dynkin quiver of type $\DD$ and $F$ is a set of alien arrows over $Q$, and let $T$ be a triangulation associated with $Q$. We use the restriction functor $\operatorname{res}$ and the notation $U(x)$, $\iota_x$ of Section \ref{subsec:res}.

\begin{remark}\label{rem:facts}
Let $N=\varphi_T(\gamma)$ be an indecomposable representation of $Q$ and write $e_x=\dim N_x=e(\gamma,\tau_x)\in\{0,1,2\}$, see \eqref{eins}. By Remark \ref{applications}:
\begin{itemize}
\item[(F1)] the support of $N$ is connected, and the vertices with $e_x=2$ form a connected subgraph containing $n-2$;
\item[(F2)] along a path of $Q$ contained in the support of $N$, the composition of the structure maps of $N$ is nonzero. Indeed, the possible patterns are $1\cdots1$ and $2\cdots2$ (isomorphisms), $1\to2$ (injective), $2\to1$ (surjective) and $1\to2\cdots2\to1$, which is nonzero because the image line of the first map and the kernel line of the last one are distinct; the pattern $2\to1\to2$ does not occur by (F1);
\item[(F3)] the three one-dimensional subspaces $U_1,U_2,U_3$ of Remark \ref{applications}, one for each arrow joining a vertex with $e=1$ and a vertex with $e=2$, are pairwise distinct, and the image in the region $\{e=2\}$ of a path entering it through one of these arrows is the corresponding line.
\end{itemize}
\end{remark}

\begin{lemma}\label{lem:ext}
Let $\gamma$ be a $\star$-arc, let $N=\varphi_T(\gamma)$ and let $\alpha\colon s\to t$ be an alien arrow in $F$. The inclusion of Proposition \ref{prop:ff}(4) holds for $\alpha$ if and only if $\gamma$ is not frozen by $\alpha$.
\end{lemma}

\begin{proof}
By Lemma \ref{lem:C}(2), the sinks $z$ with $s,t\in\Supp I(z)$ are exactly those in $U(t)$, so $\gamma$ is frozen by $\alpha$ if and only if $e_s\neq0$ and $e_t<e_z$ for some $z\in U(t)$. If $e_s=0$, the inclusion is trivial. Assume $e_s\neq0$ and write $v=({_z}h_s)_{z\in U(t)}$.

Suppose that $\gamma$ is frozen, and let $z\in U(t)$ with $e_t<e_z$. By (F1) there is a path from $s$ to $z$ inside the support, so the component $v_z$ is nonzero by (F2). If $e_t=0$, the image of $\iota_t$ is zero and the inclusion fails. If $e_t=1$ and $e_z=2$, the $z$-component of $\operatorname{Im}\iota_t$ is the line $L_t$ through which $t$ enters the region $\{e=2\}$, by (F3). If $e_s=2$, then $v_z$ is an isomorphism and its image is not contained in a line. If $e_s=1$, the image of $v_z$ is the line $L_s$ through which $s$ enters that region. We claim that $s$ and $t$ enter through different arrows. Suppose that they enter through the same arrow $\alpha'$. By Remark \ref{applications}, the region $\{e=2\}$ is an interval $\{m,\dots,n-2\}$ of the chain whose only neighbours with $e=1$ are $m-1$, $n-1$ and $n$, so $\alpha'$ is one of the three arrows $\alpha_{m-1}$, $\beta_1$, $\beta_2$ of that remark. A directed path that enters the region through $\beta_1$ (respectively $\beta_2$) starts at the leaf $n-1$ (respectively $n$), since a leaf is not an inner vertex of any path of $Q$; as $s\neq t$, this excludes $\alpha'=\beta_1,\beta_2$. Hence $\alpha'=\alpha_{m-1}$ is oriented $m-1\to m$, and $s$ and $t$ lie in the connected component $\{j,\dots,m-1\}$ of $\{e=1\}$ containing $m-1$, with directed paths to $m-1$. Since $m-1$ is an end of this component, both paths run along the chain in the same direction, so one of them contains the other and one of $s,t$ reaches the other in $Q$, contradicting conditions (c) or (d) of Definition \ref{aliensetDef}. This proves the claim; hence $L_s\neq L_t$ by (F3), and the inclusion fails.

Suppose that $\gamma$ is not frozen, so $e_z\le e_t$ for all $z\in U(t)$. If $e_t=0$, all $N_z$ with $z\in U(t)$ vanish. By Lemma \ref{lem:C}(2), $U(t)=\{z'\}$ for a sink $z'$, and $e_{z'}\le e_t$. If $e_t=2$, then $\iota_t={_{z'}}h_t$ is injective, since $N$ is socle-projective by Proposition \ref{propstar}, and hence bijective, because $e_{z'}\le2$. If $e_t=1$ and $e_{z'}=0$, the target space is zero. If $e_t=e_{z'}=1$, the path from $t$ to $z'$ lies in the support by (F1), so ${_{z'}}h_t\neq0$ by (F2) and it is bijective. In all cases $\iota_t$ is surjective and the inclusion holds. This completes the proof.
\end{proof}

\begin{theorem} \label{main}
For every sp-arc $\gamma$ there is a unique $\Bbbk\P$-module $\Theta(\gamma)$ with $\operatorname{res}\Theta(\gamma)=\varphi_T(\gamma)$, and $\gamma\mapsto\Theta(\gamma)$ extends to a $\Bbbk$-linear equivalence of categories
\[
\Theta\colon(\mathcal{C}/T)_{F}\longrightarrow\md_{sp}(\Bbbk\P).
\]
\end{theorem}
\begin{proof}
Let $\gamma$ be an sp-arc. By Proposition \ref{propstar}, $\varphi_T(\gamma)$ is a socle-projective representation of $Q$, and since $\gamma$ is not frozen by any alien arrow, Lemma \ref{lem:ext} and Proposition \ref{prop:ff}(4) give a unique $\Bbbk\P$-module $\Theta(\gamma)\in\md_{sp}(\Bbbk\P)$ with $\operatorname{res}\Theta(\gamma)=\varphi_T(\gamma)$. On morphisms, for $\rho\colon\gamma\to\gamma'$ in $(\mathcal{C}/T)_F$ we define $\Theta(\rho)$ as the unique morphism of $\Bbbk\P$-modules with $\operatorname{res}\Theta(\rho)=\varphi_T(\rho)$, which exists by Proposition \ref{prop:ff}(3); we extend $\Theta$ additively. Then $\Theta$ is a functor, and it is full and faithful because
\[
\Hom_{\Bbbk\P}\bigl(\Theta(\gamma),\Theta(\gamma')\bigr)=\Hom_{\Bbbk Q}\bigl(\varphi_T(\gamma),\varphi_T(\gamma')\bigr)\cong\Hom_{\mathcal{C}/T}(\gamma,\gamma'),
\]
by Proposition \ref{prop:ff}(3) and \eqref{equivalence1}.

To prove that $\Theta$ is dense, since $\Theta$ is additive, it suffices to consider an indecomposable object $M$ of $\md_{sp}(\Bbbk\P)$. By Proposition \ref{prop:ff}, $\operatorname{res}M$ is an indecomposable socle-projective representation of $Q$, so $\operatorname{res}M\cong\varphi_T(\gamma)$ for an arc $\gamma\notin T$ by \eqref{equivalence1}, and $\gamma$ is a $\star$-arc by Proposition \ref{propstar}. Since $\varphi_T(\gamma)\cong\operatorname{res}M$, Proposition \ref{prop:ff}(4) shows that $\varphi_T(\gamma)$ satisfies the inclusion for every alien arrow, and by Lemma \ref{lem:ext} $\gamma$ is not frozen by any of them, that is, $\gamma$ is an sp-arc. Finally, since $\operatorname{res}M\cong\varphi_T(\gamma)=\operatorname{res}\Theta(\gamma)$, Proposition \ref{prop:ff}(3) gives $M\cong\Theta(\gamma)$.
\end{proof}

\begin{remark}\label{rem:count}
Theorem \ref{main} refines Corollary \ref{finiterepresentation}: the number of isomorphism classes of indecomposable objects of $\md_{sp}(\Bbbk\P)$ is exactly the number of sp-arcs, which may be strictly smaller than the bound $n(n-1)$ given by the number of arcs $\gamma\notin T$. For instance, in Example \ref{ex:AR} there are $16$ sp-arcs, whereas $n(n-1)=42$.
\end{remark}

\begin{corollary}\label{cor:hom}
For an sp-arc $\gamma$ write $d_\gamma=(e(\gamma,\tau_x))_{x\in Q_0}$, and let $\langle a,b\rangle_Q=\sum_{x\in Q_0}a_xb_x-\sum_{(x\to y)\in Q_1}a_xb_y$ be the Euler form of $Q$. For distinct sp-arcs $\gamma,\gamma'$ we have
\[
\dim_{\Bbbk}\Hom_{\Bbbk\P}\bigl(\Theta(\gamma),\Theta(\gamma')\bigr)=\max\bigl(\langle d_\gamma,d_{\gamma'}\rangle_Q,\,0\bigr).
\]
\end{corollary}
\begin{proof}
Let $M=\varphi_T(\gamma)$ and $N=\varphi_T(\gamma')$; these are non-isomorphic indecomposable $\Bbbk Q$-modules with dimension vectors $d_\gamma$ and $d_{\gamma'}$ by \eqref{eins}. By Proposition \ref{prop:ff}(3) and Theorem \ref{main}, $\Hom_{\Bbbk\P}(\Theta(\gamma),\Theta(\gamma'))=\Hom_{\Bbbk Q}(M,N)$. Since $\Bbbk Q$ is hereditary, $\dim\Hom_{\Bbbk Q}(M,N)-\dim\Ext^1_{\Bbbk Q}(M,N)=\langle d_\gamma,d_{\gamma'}\rangle_Q$ \cite{schiffler10}. Since $Q$ is a Dynkin quiver, the Auslander--Reiten quiver of $\md\Bbbk Q$ has no oriented cycles. A nonzero extension $0\to N\to E\to M\to0$ gives a path from $N$ to $M$ in it, and a nonzero morphism $M\to N$ gives a path from $M$ to $N$; hence $\Hom_{\Bbbk Q}(M,N)$ and $\Ext^1_{\Bbbk Q}(M,N)$ are not both nonzero, which proves the corollary.
\end{proof}

\begin{remark}\label{rem:ext}
For arcs $\gamma,\gamma'\notin T$ we also have
\[
e(\gamma,\gamma')=\dim\Ext^1_{\Bbbk Q}\bigl(\varphi_T(\gamma),\varphi_T(\gamma')\bigr)+\dim\Ext^1_{\Bbbk Q}\bigl(\varphi_T(\gamma'),\varphi_T(\gamma)\bigr),
\]
a formula that does not depend on $\P$. Indeed, let $T'=\tau^{-1}T$, which is a cluster-tilting object of $\mathcal C$ \cite{Schiffler}, and let $H=\operatorname{End}_{\mathcal C}(T')$. By \cite[Theorem 2.2]{BMR}, $\Hom_{\mathcal C}(T',-)$ induces an equivalence between $\mathcal C/T$ and $\md H^{\mathrm{op}}$, so $H$ is hereditary by \eqref{equivalence1}. Since $\mathcal C$ is equivalent to the cluster category of type $\DD_n$ \cite{Schiffler}, which is $2$-Calabi--Yau and algebraic \cite{Keller05}, by \cite[\S4.2]{KR} there is a triangle equivalence $\mathcal C\to\mathcal C_H$ sending $T'$ to $H$. The indecomposable objects of $\mathcal C_H$ are the indecomposable $H$-modules and the objects $P[1]\cong\tau P$ with $P$ indecomposable projective \cite[Proposition 1.6]{BMRRT}, and $T=\tau T'$ corresponds to $\tau H$. Hence each arc $\gamma\notin T$ corresponds to an indecomposable $H$-module $Y_\gamma$, and $\varphi_T(\gamma)=\Hom_{\mathcal C}(T',\gamma)$ is identified with $\Hom_{\mathcal C_H}(H,Y_\gamma)\cong\Hom_H(H,Y_\gamma)\cong Y_\gamma$ \cite[Proposition 1.7(d)]{BMRRT}. The formula then follows from $e(\gamma,\gamma')=\dim\Ext^1_{\mathcal C}(\gamma,\gamma')$ \cite{Schiffler} and \cite[Proposition 1.7(c)]{BMRRT}.
\end{remark}

\begin{example}\label{ex:AR}
Let $\P$ be the poset of type $\mathbb{D}$ whose Hasse diagram is shown on the left of Figure \ref{A-R}.

The Hasse quiver of $\P$ is $Q^{F}$, where $Q$ is the quiver of Example \ref{examplequiverD} and $F$ is the set of alien arrows of Example \ref{examplealienarrow}. Figure \ref{A-R} shows on the right the Auslander--Reiten quiver of $\md_{sp}(\Bbbk\P)$, where each module $M$ is represented by its dimension vector, writing each $x\in\P$ as many times as $\dim_{\Bbbk}M_x$.

\begin{figure}[h!]
\centering
\begin{minipage}{0.18\textwidth}\centering
\begin{tikzcd}[row sep= small, column sep = small]
& _6\arrow[d] \arrow[ddl]\\
 & _5\arrow[d]\arrow[dr] &\\
_1\arrow[ddr]& _4\arrow[d]& _7\\
& _3  \arrow[d]\\
& _2\\
\end{tikzcd}
\end{minipage}\hfill
\begin{minipage}{0.8\textwidth}\centering
\begin{adjustbox}{scale=0.3}

\begin{tikzcd}[ampersand replacement=\&, row sep= normal, column sep = normal]
\&\smrd{ & \\ &\\ \colorb{1}& & \\ & \colorb{2}& \text{ }& &\\}   \& \& \& \& \& \&\\
\smrud{  \text{ }&\colorb{2}& \text{ }}\&\&\smrd{ & &\\ & & \\ \colorb{1}&\colorb{3} & \text{ }\\ &\colorb{2}& \\ & & } \& \& \& \& \&\\
 \& \smrud{  \\ \colorb{3}\\ \colorb{2} \\ \\} \& \& \smrd{ &\colorb{4}\\  \colorb{1} &  \colorb{3}\\ &  \colorb{2} }\& \&  \smrd{ \colorb{5}&\\ & \colorb{7}} \& \& \sm{ &\colorb{6}&\\ & \colorb{5} \\ & \colorb{4}\\ \colorb{1}& \colorb{3}\\  & \colorb{2} }\\
 \& \& \smrud{\colorb{4}\\ \colorb{3}\\ \colorb{2}} \& \&\smr{&\colorb{5}&\\ &\colorb{4} & \colorb{7}\\ \colorb{1}&\colorb{3}&\\ &\colorb{2}&} \smrud{} \& \smr{&\colorb{6}\\&\colorb{5}&\\ &\colorb{4} & \colorb{7}\\ \colorb{1}&\colorb{3}&\\ &\colorb{2}&} \&  \smrud{&\colorb{6}\\&\colorb{5} \colorb{5} &\\ &\colorb{4} & \colorb{7}\\  \colorb{1}&\colorb{3}&\\ &\colorb{2}&}\& \\
 \& \&  \&  \smrud{\colorb{5}& \\ \colorb{4}& \colorb{7}\\ \colorb{3}&\\ \colorb{2}&} \& \& \smru{&\colorb{5} \\ &\colorb{4} \\ \colorb{1}&\colorb{3}\\ &\colorb{2}}\& \& \sm{\colorb{6}& \\ \colorb{5} & \\ & \colorb{7}} \\
 \& \& \smru{\colorb{7}}  \& \& \smru{\colorb{5} \\ \colorb{4}\\ \colorb{3}\\ \colorb{2}} \&  \& \& \\
\end{tikzcd}
\end{adjustbox}
\end{minipage}
\caption{The poset $\P$ of Example \ref{ex:AR} (left) and the Auslander--Reiten quiver of $\md_{sp}(\Bbbk\P)$ (right).}\label{A-R}
\end{figure}

Let $T$ be the triangulation of Example \ref{examplequiverD}. Figure \ref{A-R model} shows the Auslander--Reiten quiver of $\rep Q$, where each indecomposable representation $\varphi_T(\gamma)$ is represented by the arc $\gamma$. The blue arcs are the $16$ sp-arcs, that is, the indecomposable objects of $(\mathcal{C}/T)_F$, and the functor $\Theta$ of Theorem \ref{main} sends them to the $16$ modules of Figure \ref{A-R}. The quiver of Figure \ref{A-R} contains, via $\Theta$, the $13$ arrows of Figure \ref{A-R model} between blue arcs, and has $8$ further arrows, which correspond to morphisms that are irreducible in $\md_{sp}(\Bbbk\P)$ but not in $\rep Q$.
\begin{figure}[h!]
\[
\begin{adjustbox}{max totalsize={0.9\textwidth}{0.85\textheight},center}
\begin{turn}{90}
\xymatrix@R=10pt@C=10pt{ &\xy/r1pc/:{\xypolygon7"A"{~<{}~>{-}{}}}*+{\cdot}
\POS"A5" \ar@{-} @/_0.4ex/ @[blue] "A7"
\endxy \ar[rd] &&\xy/r1pc/:{\xypolygon7"A"{~<{}~>{-}{}}}*+{\cdot}
\POS"A6" \ar@{-} @/_0.4ex/ "A1"
\endxy \ar[rd]
&& \xy/r1pc/:{\xypolygon7"A"{~<{}~>{-}{}}}*+{\cdot}
\POS"A7" \ar@{-} @/_0.4ex/  "A2"
\endxy \ar[rd]
&&\xy/r1pc/:{\xypolygon7"A"{~<{}~>{-}{}}}*+{\cdot}
\POS"A1" \ar@{-} @/_0.4ex/ @[blue]  "A3"
\endxy \ar[rd]
&& \xy/r1pc/:{\xypolygon7"A"{~<{}~>{-}{}}}*+{\cdot}
\POS"A2" \ar@{-} @/_0.4ex/ "A4"
\endxy \ar[rd] 
&& \xy/r1pc/:{\xypolygon7"A"{~<{}~>{-}{}}}*+{\cdot}
\POS"A3" \ar@{-} @/_0.4ex/  "A5"
\endxy && &\\
\xy/r1pc/:{\xypolygon7"A"{~<{}~>{-}{}}}*+{\cdot}
\POS"A4" \ar@{-} @/_1ex/ @[blue] "A7"
\endxy \ar[rd] \ar[ru] 
&& \xy/r1pc/:{\xypolygon7"A"{~<{}~>{-}{}}}*+{\cdot}
\POS"A5" \ar@{-} @/_1ex/ @[blue]  "A1"
\endxy \ar[rd] \ar[ru] 
&& \xy/r1pc/:{\xypolygon7"A"{~<{}~>{-}{}}}*+{\cdot}
\POS"A6" \ar@{-} @/_1ex/   "A2"
\endxy
\ar[rd] \ar[ru]
&& \xy/r1pc/:{\xypolygon7"A"{~<{}~>{-}{}}}*+{\cdot}
\POS"A7" \ar@{-} @/_1ex/   "A3"
\endxy
\ar[rd] \ar[ru] 
&& \xy/r1pc/:{\xypolygon7"A"{~<{}~>{-}{}}}*+{\cdot}
\POS"A1" \ar@{-} @/_1ex/   "A4"
\endxy
\ar[rd] \ar[ru]
&&\xy/r1pc/:{\xypolygon7"A"{~<{}~>{-}{}}}*+{\cdot}
\POS"A2" \ar@{-} @/_1ex/ @[blue]  "A5"
\endxy
\ar[rd] \ar[ru] && &&
\\
& \xy/r1pc/:{\xypolygon7"A"{~<{}~>{-}{}}}*+{\cdot}
\POS"A4" \ar@{-} @/_1.8ex/ @[blue] "A1"
\endxy
\ar[rd] \ar[ru]
&& \xy/r1pc/:{\xypolygon7"A"{~<{}~>{-}{}}}*+{\cdot}
\POS"A5" \ar@{-} @/_1.8ex/ @[blue] "A2"
\endxy
\ar[rd] \ar[ru]
&& \xy/r1pc/:{\xypolygon7"A"{~<{}~>{-}{}}}*+{\cdot}
\POS"A6" \ar@{-} @/_1.8ex/ "A3"
\endxy
\ar[rd] \ar[ru]
&&
\xy/r1pc/:{\xypolygon7"A"{~<{}~>{-}{}}}*+{\cdot}
\POS"A7" \ar@{-} @/_1.8ex/ "A4"
\endxy
\ar[rd] \ar[ru]
&&
\xy/r1pc/:{\xypolygon7"A"{~<{}~>{-}{}}}*+{\cdot}
\POS"A1" \ar@{-} @/_1.8ex/ @[blue]"A5"
\endxy
\ar[rd] \ar[ru]
&& \xy/r1pc/:{\xypolygon7"A"{~<{}~>{-}{}}}*+{\cdot}
\POS"A2" \ar@{-} @/_1.8ex/ "A6"
\endxy
\ar[rd] && &
\\
&& \xy/r1pc/:{\xypolygon7"A"{~<{}~>{-}{}}}*+{\cdot}
\POS"A4" \ar@{-} @/_2.5ex/ @[blue] "A2"
\endxy
\ar[rd] \ar[ru] 
&&\xy/r1pc/:{\xypolygon7"A"{~<{}~>{-}{}}}*+{\cdot}
\POS"A5" \ar@{-} @/_2.5ex/ @[blue] "A3"
\endxy
\ar[rd] \ar[ru]
&&\xy/r1pc/:{\xypolygon7"A"{~<{}~>{-}{}}}*+{\cdot}
\POS"A6" \ar@{-} @/_2.5ex/ "A4"
\endxy
\ar[rd] \ar[ru]
&&\xy/r1pc/:{\xypolygon7"A"{~<{}~>{-}{}}}*+{\cdot}
\POS"A7" \ar@{-} @/_2.5ex/ "A5"
\endxy
\ar[rd] \ar[ru]
&&\xy/r1pc/:{\xypolygon7"A"{~<{}~>{-}{}}}*+{\cdot}
\POS"A1" \ar@{-} @/_2.5ex/ "A6"
\endxy
\ar[rd] \ar[ru]
&&\xy/r1pc/:{\xypolygon7"A"{~<{}~>{-}{}}}*+{\cdot}
\POS"A2" \ar@{-} @/_2.5ex/ "A7"
\endxy
\ar[rd] 
&&\\
&&&\xy/r1pc/:{\xypolygon7"A"{~<{}~>{-}{}}}*+{\cdot}
\POS"A4" \ar@{-} @/_3ex/ @[blue] "A3"
\endxy
\ar[rd] \ar[ru] \ar[r]
&\xy/r1pc/:{\xypolygon7"A"{~<{}~>{-}{}}}*+{\cdot}
\POS"A4" \ar|-(0.8){\SelectTips{cm}{}\object@{+}}@{-}    "A0"
\endxy \ar[r]
&\xy/r1pc/:{\xypolygon7"A"{~<{}~>{-}{}}}*+{\cdot}
\POS"A5" \ar@{-} @/_3ex/  "A4"
\endxy
\ar[rd] \ar[ru] \ar[r]
&\xy/r1pc/:{\xypolygon7"A"{~<{}~>{-}{}}}*+{\cdot}
\POS"A5" \ar@{-} @[blue]  "A0"
\endxy  \ar[r]
&\xy/r1pc/:{\xypolygon7"A"{~<{}~>{-}{}}}*+{\cdot}
\POS"A6" \ar@{-} @/_3ex/ "A5"
\endxy
\ar[rd] \ar[ru] \ar[r]
&\xy/r1pc/:{\xypolygon7"A"{~<{}~>{-}{}}}*+{\cdot}
\POS"A6" \ar|-(0.8){\SelectTips{cm}{}\object@{+}}@{-}   "A0"
\endxy \ar[r] 
&\xy/r1pc/:{\xypolygon7"A"{~<{}~>{-}{}}}*+{\cdot}
\POS"A7" \ar@{-} @/_3ex/ "A6"
\endxy
\ar[rd] \ar[ru] \ar[r]
&\xy/r1pc/:{\xypolygon7"A"{~<{}~>{-}{}}}*+{\cdot}
\POS"A7" \ar@{-} "A0"
\endxy \ar[r] 
&\xy/r1pc/:{\xypolygon7"A"{~<{}~>{-}{}}}*+{\cdot}
\POS"A1" \ar@{-} @/_3ex/ "A7"
\endxy
\ar[rd] \ar[ru] \ar[r]
&\xy/r1pc/:{\xypolygon7"A"{~<{}~>{-}{}}}*+{\cdot}
\POS"A1" \ar|-(0.8){\SelectTips{cm}{}\object@{+}}@{-}   @[blue] "A0"
\endxy \ar[r]
&\xy/r1pc/:{\xypolygon7"A"{~<{}~>{-}{}}}*+{\cdot}
\POS"A2" \ar@{-} @/_3ex/ "A1"
\endxy \ar[rd]&\\
&& \xy/r1pc/:{\xypolygon7"A"{~<{}~>{-}{}}}*+{\cdot}
\POS"A3" \ar|-(0.8){\SelectTips{cm}{}\object@{+}}@{-}   @[blue] "A0"
\endxy \ar[ru] 
&&\xy/r1pc/:{\xypolygon7"A"{~<{}~>{-}{}}}*+{\cdot}
\POS"A4" \ar@{-}  @[blue]  "A0"
\endxy \ar[ru] 
&& \xy/r1pc/:{\xypolygon7"A"{~<{}~>{-}{}}}*+{\cdot}
\POS"A5" \ar|-(0.8){\SelectTips{cm}{}\object@{+}}@{-}  @[blue] "A0"
\endxy \ar[ru] 
&& \xy/r1pc/:{\xypolygon7"A"{~<{}~>{-}{}}}*+{\cdot}
\POS"A6" \ar@{-}   "A0"
\endxy \ar[ru] 
&& \xy/r1pc/:{\xypolygon7"A"{~<{}~>{-}{}}}*+{\cdot}
\POS"A7" \ar|-(0.8){\SelectTips{cm}{}\object@{+}}@{-}   "A0"
\endxy \ar[ru] 
&& \xy/r1pc/:{\xypolygon7"A"{~<{}~>{-}{}}}*+{\cdot}
\POS"A1" \ar@{-}   "A0"
\endxy \ar[ru] 
&& \xy/r1pc/:{\xypolygon7"A"{~<{}~>{-}{}}}*+{\cdot}
\POS"A2" \ar|-(0.8){\SelectTips{cm}{}\object@{+}}@{-}   "A0"
\endxy
}
\end{turn}
\end{adjustbox}
\]
\caption{Auslander--Reiten quiver of $\rep Q$ in terms of arcs. The blue arcs are the sp-arcs of Example \ref{examplealienarrow}.} \label{A-R model}
\end{figure}
\end{example}

\section{Applications}\label{sec:applications}

In this section we give two applications of Theorem \ref{main}, to subalgebras of the cluster algebra of type $\DD$ generated by sp-arcs (Section \ref{sec:cluster}), and to total stability (Section \ref{sec:stability}).

\subsection{Cluster subalgebras}\label{sec:cluster}

We first recall the case of type $\mathbb{A}$ \cite{schifflerserna}. Let $\P$ be a poset of type $\mathbb{A}$ associated with the quiver $Q^F$ of Proposition~\ref{0}. Let $\mathcal{A}=\mathcal{A}(x,Q)$ be the cluster algebra associated with the initial seed $(x,Q)$, where $x=\{x_1,\dots,x_n\}$, and let $\mathcal{A}(\P)$ be the subalgebra of $\mathcal{A}$ generated by the initial cluster variables $x_1,\dots,x_n$ together with the cluster variables $x_\gamma$ such that $\gamma$ is an sp-diagonal in the category $\mathcal{C}_{(T,F)}$ (see Subsection~\ref{subsectionposetsof typeA}). Theorem~5.2 of \cite{schifflerserna} gives a partial answer to the question of when $\mathcal{A}(\P)=\mathcal{A}$.

Now let $\P$ be a poset of type $\mathbb{D}$ associated with $Q^F$ as in Definition~\ref{def:QF}, and let $\mathcal{A}=\mathcal{A}(x,Q)$. For an arc $\gamma\notin T$ we write $x_\gamma=X_{\varphi_T(\gamma)}$, where $X_M$ is the cluster variable associated with an indecomposable $\Bbbk Q$-module $M$ by the Caldero--Chapoton map \cite{CalderoChapoton}. Thus $\gamma\mapsto x_\gamma$ is a bijection between the arcs not in $T$ and the non-initial cluster variables of $\mathcal{A}$. Let $\mathcal{A}(\P)$ be the subalgebra of $\mathcal{A}$ generated by the initial cluster variables $x_1,\dots,x_n$ together with the cluster variables $x_\gamma$ such that $\gamma$ is an sp-arc in $(\mathcal{C}/T)_F$. We obtain the following property.

\begin{corollary}\label{clusterprop}
    Let $\P$ be a poset of type $\mathbb{D}$ associated with $Q^F$, where $F=\varnothing$. Then $\mathcal{A}(\P)=\mathcal{A}$.
\end{corollary}

\begin{proof}
    Since $F=\varnothing$, the sp-arcs are the $\star$-arcs, that is, by Proposition \ref{propstar}, the arcs $\gamma\notin T$ with $\varphi_T(\gamma)$ socle-projective. Among them are the arcs $\tau^{-1}\tau_i$, since $\varphi_T(\tau^{-1}\tau_i)=P(i)$ (see the proof of Proposition \ref{propstar}) and the socle of $P(i)$ is a direct sum of simple modules $S(z)$ with $z$ a sink. Hence $\mathcal{A}(\P)$ contains $x_1,\dots,x_n$ and the variables $x_{\tau^{-1}\tau_i}=X_{P(i)}$, which are the variables $x_{P_i}$ of \cite[Lemma 5.1]{schifflerserna}. Since $Q$ is of tree type, that lemma gives $\mathcal{A}(\P)=\mathcal{A}$.
\end{proof}

\begin{remark}\label{rem:Fnonempty}
The hypothesis $F=\varnothing$ in Corollary \ref{clusterprop} cannot be dropped. Let $\P$ be the poset of Example \ref{onepeak-D}, so that $Q$ is the quiver $1\to2\to3$, $4\to3$, $5\to3$ and $F=\{\alpha\colon1\to4,\ \beta\colon1\to5\}$. We show that the cluster variable $x_1'=(1+x_2)/x_1$ does not belong to $\mathcal{A}(\P)$; in fact, it is not even integral over $\mathcal{A}(\P)$.

We first show that the only sp-arc $\gamma$ with $e(\gamma,\tau_1)\neq0$ has crossing vector $d=(1,1,1,1,1)$. Indeed, let $d=(e(\gamma,\tau_i))_i$ with $d_1\neq0$. Since $\gamma$ is a $\star$-arc, by Remark \ref{obs:star} there is a sink reachable from $1$ whose coordinate in $d$ is nonzero; as $3$ is the only sink of $Q$, we have $d_3\geq1$, and as the support of $\varphi_T(\gamma)$ is connected, also $d_2\geq1$. The only sink $z$ with $1,4\in\Supp I(z)$, and also the only one with $1,5\in\Supp I(z)$, is $z=3$. Since $\gamma$ crosses $\tau_1$ and is frozen neither by $\alpha$ nor by $\beta$, we get $d_4\geq d_3$ and $d_5\geq d_3$; as the leaves satisfy $d_4,d_5\le1$, this gives $d_3=d_4=d_5=1$, and then $d=(1,1,1,1,1)$ by Remark \ref{applications}. By the Caldero--Chapoton formula \cite{CalderoChapoton}, every cluster variable $x_\gamma$ with $\gamma\notin T$ has the form $N_\gamma/x^{d}$ with $N_\gamma\in\mathbb{Z}[x_1,\dots,x_5]$, where $d$ is the dimension vector of $\varphi_T(\gamma)$, that is, the crossing vector of $\gamma$ by \eqref{eins}, and $x^{d}=\prod_i x_i^{d_i}$. For $d=(1,1,1,1,1)$, the module $\varphi_T(\gamma)$ has dimension $1$ at every vertex and nonzero maps, so its submodules correspond to the subsets of $Q_0$ closed under successors, and the formula gives
\[
N_\gamma=x_1\bigl((1+x_3)^2+x_2x_4x_5\bigr)+(1+x_2)\,x_3\,(1+x_3)^2 .
\]
Let $K=\mathbb{Q}(x_4,x_5)$, $u=1+x_2$ and $w=1+x_3$, so that $\mathbb{Q}(x_1,\dots,x_5)=K(x_1,u,w)$, and, since $x_1,u,w$ are algebraically independent over $K$, let $\nu$ be the monomial valuation with weights $(2,1,1)$ on $(x_1,u,w)$, that is,
\[
\nu\Bigl(\sum c_{abc}\,x_1^au^bw^c\Bigr)=\min\{2a+b+c: c_{abc}\neq0\}\qquad(c_{abc}\in K).
\]
Then $\nu(x_1)=2$ and $\nu(x_i)=0$ for $i\geq2$, so $\nu\geq0$ on $\mathbb{Z}[x_1,\dots,x_5]$ and on every $x_\gamma$ with $d_1=0$, since then $\nu(x^d)=0$. For the remaining sp-arc, $N_\gamma=x_1\bigl(w^2-x_4x_5+u\,x_4x_5\bigr)+u\,(w-1)\,w^2$ has $\nu(N_\gamma)=2$, so $\nu(x_\gamma)=2-\nu(x_1x_2x_3x_4x_5)=0$. Hence $\nu\geq0$ on all generators of $\mathcal{A}(\P)$, and therefore on every element that is integral over $\mathcal{A}(\P)$, since valuation rings are integrally closed. On the other hand, $\nu(x_1')=\nu(u)-\nu(x_1)=-1$. Thus $\mathcal{A}(\P)\subsetneq\mathcal{A}$.
\end{remark}

\subsection{Total stability}\label{sec:stability}

Following King \cite{King} and Rudakov \cite{Rudakov}, and in the convention of \cite[Definition 4.1]{IusenkoBravoRiosSerna2025}, a \emph{central charge} on $\md_{sp}(\Bbbk\P)$ is an additive function $Z=\kappa+i\theta\colon\mathbb{Z}^{\P}\to\mathbb{C}$ such that $\kappa(\underline{\dim}\,M)>0$ for every nonzero module $M$; we write $Z(M)=Z(\underline{\dim}\,M)$. A nonzero module $M$ is \emph{$Z$-stable} if $\arg Z(X)<\arg Z(M)$, that is, $\theta(X)/\kappa(X)<\theta(M)/\kappa(M)$, for every nonzero proper $\Bbbk\P$-submodule $X$ of $M$; such an $X$ is again socle-projective, since $\operatorname{soc}X$ is a direct summand of the projective module $\operatorname{soc}M$. We say that $Z$ is \emph{total} if every indecomposable module is $Z$-stable. The same definitions apply to $\md\Bbbk Q$. Reineke conjectured that $\md\Bbbk Q$ admits a total central charge for every Dynkin quiver $Q$ \cite[Conjecture 7.1]{Reineke}, and Chang, Qiu and Zhang proved it \cite[Theorem 5.6]{CQZ}. For $Q$ of type $\DD_n$, the values $Z(S_x)$ are, up to the rotation by $-i$ that passes from the upper half-plane convention of \cite{CQZ} to ours, the vectors of certain diagonals of a centrally symmetric convex $2(n-1)$-gon with two punctures \cite[Definition 4.7 and Proposition 5.2]{CQZ}.

Throughout this subsection, $\P$ is a poset of type $\DD$ associated with $Q^F$ (Definition \ref{def:QF}) and $T$ is the triangulation associated with $Q$. We use the restriction functor $\operatorname{res}\colon\md(\Bbbk\P)\to\md(\Bbbk Q)$ of Section \ref{subsec:res}; recall that it is exact, preserves dimension vectors and sends submodules to submodules.

\begin{theorem}\label{thm:total}
If $Z$ is a total central charge on $\md\Bbbk Q$, then $Z$ is also total on $\md_{sp}(\Bbbk\P)$. Moreover, for every sp-arc $\gamma$,
\[
Z(\Theta(\gamma))=\sum_{x\in Q_0}e(\gamma,\tau_x)\,Z(S_x).
\]
\end{theorem}

\begin{proof}
Since $Q_0=\P$ and $\underline{\dim}\,M=\underline{\dim}\operatorname{res}M$ for every $\Bbbk\P$-module $M$, the function $Z$ is also a central charge on $\md_{sp}(\Bbbk\P)$. Let $M$ be indecomposable in $\md_{sp}(\Bbbk\P)$ and let $X$ be a nonzero proper submodule of $M$. Then $\operatorname{res}X$ is a nonzero proper submodule of $\operatorname{res}M$ with the same dimension vector as $X$. By Proposition \ref{prop:ff}(3), $\operatorname{res}M$ is indecomposable, hence $Z$-stable, and therefore $\arg Z(X)=\arg Z(\operatorname{res}X)<\arg Z(\operatorname{res}M)=\arg Z(M)$. The formula follows from the additivity of $Z$, since $\dim_{\Bbbk}\Theta(\gamma)_x=e(\gamma,\tau_x)$ by Theorem \ref{main} and \eqref{eins}.
\end{proof}

Unlike in \cite{IusenkoBravoRiosSerna2025}, the central charge in Theorem \ref{thm:total} does not come from the bilinear form of $\P$, and it is the same for all posets of type $\DD$ with Dynkin quiver $Q$, whatever the set $F$ of alien arrows. Thus a single central charge, read off from the polygon of \cite{CQZ}, makes every indecomposable socle-projective $\Bbbk\P$-module stable. Moreover, under the equivalence between peak $\P$-spaces and $\md_{sp}(\Bbbk\P)$ \cite[Lemma 2.1(b)]{simson3}, every proper subrepresentation of a peak $\P$-space in the sense of \cite[Section 1.2]{IusenkoBravoRiosSerna2025} gives a $\Bbbk\P$-submodule with the same dimension vector. Hence Theorem \ref{thm:total} also gives the stability of every indecomposable peak $\P$-space in the sense of \cite{IusenkoBravoRiosSerna2025}.

\ifarxiv
\appendix
\section{Sincere Posets of Type $\DD$}\label{sincereposetsD}
\begin{remark}\label{rem:D4notsincere}
A poset of type $\DD$ associated with a quiver $Q^F$ in which $Q$ is in configuration $D_4$ is not sincere. Hence every sincere poset of type $\DD$ is associated with a quiver $Q^F$ in which $Q$ is in configuration $D_1$, $D_2$ or $D_3$, and the list below contains no poset coming only from configuration $D_4$. Indeed, let $\P$ be associated with $Q^F$, with $Q$ in configuration $D_4$, and suppose that $M$ is a sincere socle-projective $\Bbbk\P$-module. The points $n-2$, $n-1$ and $n$ are not maximal in $\P$, since $n-2\to n-3$, $n-1\to n-2$ and $n\to n-2$ are arrows of $Q$, so $(\operatorname{top}M)_x\neq0$, and hence $e_x=\dim_{\Bbbk}M_x\neq0$, for $x=n-2,n-1,n$. By Proposition \ref{prop:ff}(3), $\operatorname{res}M$ is an indecomposable representation of $Q$, so it is the representation described in Remark \ref{applications}, and in particular $e_{n-1}=e_n=1$. If $e_{n-2}=1$, the structure map ${_{n-2}}h_{n-1}$ is the identity of $\Bbbk$, so $M_{n-2}=\operatorname{Im}{_{n-2}}h_{n-1}$. If $e_{n-2}=2$, both arrows $\beta_1$ and $\beta_2$ of Remark \ref{applications} point to $n-2$, and since $\operatorname{res}M$ is indecomposable, $U_2=\operatorname{Im}{_{n-2}}h_{n-1}$ and $U_3=\operatorname{Im}{_{n-2}}h_{n}$ are distinct lines of $M_{n-2}=\Bbbk^2$, so $M_{n-2}=U_2+U_3$. In both cases $M_{n-2}=(\operatorname{rad}M)_{n-2}$, because $n-1,n\prec n-2$, so $(\operatorname{top}M)_{n-2}=0$, a contradiction. Finally, in configurations $D_1$, $D_2$ and $D_3$ condition (ii) of Definition \ref{def:QF} is equivalent to the condition used to obtain the list below (Remark \ref{rem:oldD}), so the list is not affected by configuration $D_4$.
\end{remark}

 By Corollary \ref{finiteprin}, every poset of type $\DD$ is of finite prinjective type. Hence every sincere poset of type $\DD$ occurs in the classification of sincere posets of finite representation type of \cite[Theorem 1]{kleiner75} and in the classifications of sincere $r$-peak posets of finite prinjective type of \cite[Table 3.2]{justina}, \cite[Table 2.2]{justina1}, and \cite[Table 2.2]{justina2}. The following list comprises the posets of type $\DD$ occurring in those lists.

\begin{longtable}{|c|}
\hline\begin{adjustbox}{max width=0.95\textwidth}
         
$\mathcal{F}_1^{(r)}=$\begin{tikzcd}[row sep= small, column sep = 0.4, scale=0.8]
 & _1 \arrow[d]\arrow[dl]\arrow[dr]&&_2\arrow[dl]\arrow[dr] & &_3\arrow[dl]\arrow[dr]&&\cdots &&\ _{r-2}\arrow[dl]\arrow[dr] \\
\star & \star & \star &&\star &&\star&\cdots &\star&&\star\\
\end{tikzcd} $_{r \geq 3}$
\end{adjustbox}
\\
\hline
\begin{adjustbox}{max width=0.95\textwidth}
$\mathcal{F}_2^{(r)}=$\begin{tikzcd}[row sep= small, column sep = 0.4]
& & _2 \arrow[dl]\arrow[dr]&&_3\arrow[dl]\arrow[dr] & &_4\arrow[dl]\arrow[dr]&&\cdots &&\ _{r-1}\arrow[dl]\arrow[dr] \\
&\ _1\arrow[dl]\arrow[dr]&  & \star &&\star &&\star&\cdots &\star&&\star\\
\star && \star\\ 
\end{tikzcd}  $_{r \geq 3}$
\end{adjustbox}
\\
\hline

\begin{adjustbox}{max width=0.95\textwidth}
$\mathcal{F}_{3,k}^{(r)}=$\begin{tikzcd}[row sep= small, column sep = 0.4]
&&&&&&&&&& _{k+1 }\arrow[dl]\arrow[dr]&& _{k+2} \arrow[dl]\arrow[dr]&&\cdots && _{r-1}\arrow[dl]\arrow[dr]\\
 & _1 \arrow[d]\arrow[dl]\arrow[dr]&&_2\arrow[dl]\arrow[dr]& &
 \cdots && _{k-1}\arrow[dr] \arrow[dl]&& _{k}\arrow[dl]&&\star && \star &\cdots &\star && \star \\
\star & \star & \star &&\star &\cdots &\star &&\star \\
\end{tikzcd} $_{r \geq 4\text{, }2 \leq k \leq r-2}$
\end{adjustbox}
\\
\hline

\begin{adjustbox}{max width=0.95\textwidth}
$\mathcal{F}_4^{(r)}=$\begin{tikzcd}[row sep= small, column sep = 0.4]
 & _1 \arrow[d]\arrow[dl]\arrow[dr]&&_2\arrow[dl]\arrow[dr] & &_3\arrow[dl]\arrow[dr]&&\cdots &&\ _{r-2}\arrow[dl]\arrow[dr]& _{r-1}\arrow[d] \\
\star & \star & \star &&\star &&\star&\cdots &\star&&\star\\
\end{tikzcd} $_{r \geq 3}$
\end{adjustbox}
\\
\hline
\begin{adjustbox}{max width=0.95\textwidth}
$\mathcal{F}_5^{(r)}=$\begin{tikzcd}[row sep= small, column sep = 0.4]
& & _2 \arrow[dl]\arrow[dr]&&_3\arrow[dl]\arrow[dr] & &_4\arrow[dl]\arrow[dr]&&\cdots &&\ _{r-1}\arrow[dl]\arrow[dr]& _r\arrow[d] \\
&\ _1\arrow[dl]\arrow[dr]&  & \star &&\star &&\star&\cdots &\star&&\star\\
\star && \star\\ 
\end{tikzcd} $_{r \geq 3}$
\end{adjustbox}
\\
\hline

\begin{adjustbox}{max width=0.95\textwidth}
$\mathcal{F}_{6,k}^{(r)}=$\begin{tikzcd}[row sep= small, column sep = 0.4]
&&&&&&&&&& _{k+1} \arrow[dl]\arrow[dr]&& _{k+2} \arrow[dl]\arrow[dr]&&\cdots && _{r-1}\arrow[dl]\arrow[dr] & _r\arrow[d]\\
 & _1 \arrow[d]\arrow[dl]\arrow[dr]&&_2\arrow[dl]\arrow[dr]& &
 \cdots && _{k-1}\arrow[dr] \arrow[dl]&& _{k}\arrow[dl]&&\star && \star &\cdots &\star && \star \\
\star & \star & \star &&\star &\cdots &\star &&\star \\
\end{tikzcd} $_{r \geq 4\text{, }2 \leq k \leq r-2}$
\end{adjustbox}
\\
\hline

\begin{adjustbox}{max width=0.95\textwidth}
$\mathcal{F}_7^{(r)}=$\begin{tikzcd}[row sep= small, column sep =0.4]
&&&&&&&&&& _{r-1}\arrow[d]\\
 & _1 \arrow[d]\arrow[dl]\arrow[dr]&&_2\arrow[dl]\arrow[dr] & &_3\arrow[dl]\arrow[dr]&&\cdots &&\ _{r-2}\arrow[dl]\arrow[dr]& _{r}\arrow[d] \\
\star & \star & \star &&\star &&\star&\cdots &\star&&\star\\
\end{tikzcd} $_{r \geq 3}$
\end{adjustbox}
\\
\hline
\begin{adjustbox}{max width=0.95\textwidth}
$\mathcal{F}_9^{(r)}=$\begin{tikzcd}[row sep= small, column sep = 0.4]
& _1\arrow[d]\\ 
 & _2 \arrow[dl]\arrow[dr]&&_3\arrow[dl]\arrow[dr] & &_4\arrow[dl]\arrow[dr]&&\cdots &&\ _{r}\arrow[dl]\arrow[dr] \\
\star & & \star &&\star &&\star&\cdots &\star&&\star\\
\end{tikzcd} $_{r \geq 2}$
\end{adjustbox}
\\
\hline
\begin{adjustbox}{max width=0.95\textwidth}
$\mathcal{F}_{11,k}^{(r)}=$\begin{tikzcd}[row sep= small, column sep = 0.4]
&_1\arrow[d]&&&&&&&&& _{k+1} \arrow[dl]\arrow[dr]&& _{k+2} \arrow[dl]\arrow[dr]&&\cdots && _{r+1}\arrow[dl]\arrow[dr]\\
 & _2 \arrow[dl]\arrow[dr]&&_3\arrow[dl]\arrow[dr]& &
 \cdots && _{k-1}\arrow[dr] \arrow[dl]&& _{k}\arrow[dl]&&\star && \star &\cdots &\star && \star \\
\star &  & \star &&\star &\cdots &\star &&\star \\
\end{tikzcd} $_{r \geq 3\text{, }3 \leq k \leq r,}$
\end{adjustbox}
\\
\hline
\begin{adjustbox}{max width=0.95\textwidth}
$\mathcal{F}_{12}^{(r)}=$\begin{tikzcd}[row sep= small, column sep = 0.4]
& _1\arrow[d]\\ 
 & _2 \arrow[dl]\arrow[dr]&&_3\arrow[dl]\arrow[dr] & &_4\arrow[dl]\arrow[dr]&&\cdots &&\ _{r}\arrow[dl]\arrow[dr]& _{r+1}\arrow[d] \\
\star & & \star &&\star &&\star&\cdots &\star&&\star\\
\end{tikzcd} $_{r \geq 2}$
\end{adjustbox}
\\
\hline\begin{adjustbox}{max width=0.95\textwidth}
 
$\mathcal{F}_{14}^{(r)}=$\begin{tikzcd}[row sep= small, column sep = 0.4]
 & _1 \arrow[dl]\arrow[dr]&&_2\arrow[dl]\arrow[dr] & &_3\arrow[dl]\arrow[dr]&&\cdots &&\ _{r-1}\arrow[dl]\arrow[dr] & _{r}\arrow[d]& _{r+1}\arrow[dl] \\
\star &  & \star &&\star &&\star&\cdots &\star&&\star\\
\end{tikzcd} $_{r \geq 1}$
\end{adjustbox}
\\
\hline
\begin{adjustbox}{max width=0.95\textwidth}
$\mathcal{F}_{15,k}^{(r)}=$\begin{tikzcd}[row sep= small, column sep = 0.4]
&_1\arrow[d]&&&&&&&&& _{k+1} \arrow[dl]\arrow[dr]&& _{k+2} \arrow[dl]\arrow[dr]&&\cdots && _{r+1}\arrow[dl]\arrow[dr]& _{r+2}\arrow[d]\\
 & _2 \arrow[dl]\arrow[dr]&&_3\arrow[dl]\arrow[dr]& &
 \cdots && _{k-1}\arrow[dr] \arrow[dl]&& _{k}\arrow[dl]&&\star && \star &\cdots &\star && \star \\
\star &  & \star &&\star &\cdots &\star &&\star \\
\end{tikzcd} $_{r \geq 3\text{, }3 \leq k \leq r}$
\end{adjustbox}
\\
\hline
\begin{adjustbox}{max width=0.95\textwidth}
$\mathcal{F}_{16}^{(r)}=$\begin{tikzcd}[row sep= small, column sep = 0.4]
_1 \arrow[d]& _2 \arrow[dl]\arrow[dr]&&_3\arrow[dl]\arrow[dr] & &_4\arrow[dl]\arrow[dr]&&\cdots &&\ _{r}\arrow[dl]\arrow[dr] & _{r+1}\arrow[d]& _{r+2}\arrow[dl] \\
\star &  & \star &&\star &&\star&\cdots &\star&&\star\\
\end{tikzcd} $_{r \geq 2}$
\end{adjustbox}
\\
\hline\begin{adjustbox}{max width=0.95\textwidth}
 
$\mathcal{F}_{17,k}^{(r)}=$\begin{tikzcd}[row sep= small, column sep = 0.4]
&&&&&&&&&&& _{k+1} \arrow[dl]\arrow[dr]&& _{k+2} \arrow[dl]\arrow[dr]&&\cdots && _{r+2}\arrow[dl]\arrow[dr]\\
 _1\arrow[dr]& _2\arrow[d] &_3\arrow[dl]\arrow[dr]& &_4\arrow[dl]\arrow[dr]
 &&\cdots && _{k-1}\arrow[dr] \arrow[dl]&& _{k}\arrow[dl]&&\star && \star &\cdots &\star && \star \\
&\star &  & \star &&\star &\cdots &\star &&\star \\
\end{tikzcd} $_{r \geq 2\text{, }3 \leq k \leq r+1}$
\end{adjustbox}
\\
\hline
\begin{adjustbox}{max width=0.95\textwidth}
$\mathcal{F}_{18}^{(r)}=$\begin{tikzcd}[row sep= small, column sep = 0.4]
_1\arrow[d]\\
_2 \arrow[d]& _3 \arrow[dl]\arrow[dr]&&_4\arrow[dl]\arrow[dr] & &_5\arrow[dl]\arrow[dr]&&\cdots &&\ _{r+1}\arrow[dl]\arrow[dr] & _{r+2}\arrow[d]& _{r+3}\arrow[dl] \\
\star &  & \star &&\star &&\star&\cdots &\star&&\star\\
\end{tikzcd} $_{r \geq 2}$
\end{adjustbox}
\\
\hline\begin{adjustbox}{max width=0.95\textwidth}

$\mathcal{F}_{19,k}^{(r)}=$\begin{tikzcd}[row sep=small, column sep =0.]
&&&&&&&&&&& _{k+1} \arrow[dl]\arrow[dr]&& _{k+2} \arrow[dl]\arrow[dr]&&\cdots && _{r+2}\arrow[dl]\arrow[dr]& _{r+3}\arrow[d]\\
 _1\arrow[dr]& _2\arrow[d] &_3\arrow[dl]\arrow[dr]& &_4\arrow[dl]\arrow[dr]
 &&\cdots && _{k-1}\arrow[dr] \arrow[dl]&& _{k}\arrow[dl]&&\star && \star &\cdots &\star && \star \\
&\star &  & \star &&\star &\cdots &\star &&\star \\
\end{tikzcd}

$_{r \geq 2\text{, }3 \leq k \leq r+1}$
\end{adjustbox}
\\
\hline
\begin{adjustbox}{max width=0.95\textwidth}
$\mathcal{F}_{20}^{(r)}=$\begin{tikzcd}[row sep= small, column sep = 0.4]
& _1\arrow[d]&&&&&&&&&_{r+1}\arrow[d]\\ 
 & _2 \arrow[dl]\arrow[dr]&&_3\arrow[dl]\arrow[dr] & &_4\arrow[dl]\arrow[dr]&&\cdots &&\ _{r}\arrow[dl]\arrow[dr]& _{r+2}\arrow[d] \\
\star & & \star &&\star &&\star&\cdots &\star&&\star\\
\end{tikzcd}$_{r \geq 2}$
\end{adjustbox}
\\
\hline

\end{longtable}

\section{Sincere Representations of the Sincere Posets of Type $\DD$}\label{sincererepsD}

For each sincere poset $\mathcal{F}$ of type $\DD$ listed in \ref{sincereposetsD} we display below the complete list of the sincere indecomposable socle-projective $\Bbbk\mathcal{F}$-modules $L=(L_x,{_y}h_x)_{x,y\in\mathcal{F}}$. The list is the one obtained by Kosakowska in \cite[Table 8.2]{justina2} (see also \cite[Table 8.1]{justina} and \cite[Table 8.1]{justina1} for the two-peak and the three-peak cases, and \cite[Theorem 1]{kleiner75}). Kosakowska's tables are stated for prinjective $\Bbbk\mathcal{F}$-modules. The category $\prin(\Bbbk\mathcal{F})$ of prinjective modules and $\md_{sp}(\Bbbk\mathcal{F})$ are different full subcategories of $\md(\Bbbk\mathcal{F})$; they are related by Simson's adjustment functor $\vartheta_{\mathcal{F}}\colon\prin(\Bbbk\mathcal{F})\to\md_{sp}(\Bbbk\mathcal{F})$, which is full, dense and additive, and induces an equivalence $\prin(\Bbbk\mathcal{F})/[e_j\Bbbk\mathcal{F}^- : j\in\mathcal{F}^-]\simeq\md_{sp}(\Bbbk\mathcal{F})$, where $\mathcal{F}^-=\mathcal{F}\setminus\max\mathcal{F}$ \cite[Lemma 2.1]{simson3}. Moreover, $\dm X=\dm\vartheta_{\mathcal{F}}(X)$ for every indecomposable prinjective $\Bbbk\mathcal{F}$-module $X$ with $\vartheta_{\mathcal{F}}(X)\neq 0$, while $\dm(e_j\Bbbk\mathcal{F}^-)$ is the $j$-th standard basis vector \cite[Section 2]{simson3}. Since $\mathcal{F}$ has more than one point, the modules $e_j\Bbbk\mathcal{F}^-$ are not sincere, so $\vartheta_{\mathcal{F}}$ induces a bijection, preserving coordinate vectors, between the isomorphism classes of sincere indecomposable prinjective $\Bbbk\mathcal{F}$-modules and those of sincere indecomposable socle-projective $\Bbbk\mathcal{F}$-modules. Hence the modules tabulated by Kosakowska correspond, via $\vartheta_{\mathcal{F}}$, precisely to the sincere socle-projective $\Bbbk\mathcal{F}$-modules listed below.

In the diagrams of this appendix, the vector space $L_x$ is written at the place occupied by the point $x$ in the Hasse diagram of $\mathcal{F}$ given in \ref{sincereposetsD}, and each arrow $x\to y$ carries the structure map ${_y}h_x$. Following \cite[Section 8]{justina2}, an unlabelled arrow is the identity whenever $\dim L_x=\dim L_y$, and it is the canonical map
$\left[\begin{smallmatrix}1&0\end{smallmatrix}\right]$ (respectively, $\left[\begin{smallmatrix}1&0\end{smallmatrix}\right]^{t}$) whenever $\dim L_x=2$ and $\dim L_y=1$ (respectively, $\dim L_x=1$ and $\dim L_y=2$); an arrow labelled $\varepsilon$ is the complementary canonical map $\left[\begin{smallmatrix}0&1\end{smallmatrix}\right]$ (respectively, $\left[\begin{smallmatrix}0&1\end{smallmatrix}\right]^{t}$); and an arrow labelled $d$ is the map given by the matrix $\left[\begin{smallmatrix}1&1\end{smallmatrix}\right]$ (respectively, $\left[\begin{smallmatrix}1&1\end{smallmatrix}\right]^{t}$).

Each poset below is named $\mathcal{F}_i^{(r)}$ (or $\mathcal{F}_{i,k}^{(r)}$), the same name it receives in \ref{sincereposetsD}, following the numbering and order of \cite[Table 2.2]{justina2}. The families $\mathcal{F}^{(r)}_{8}$, $\mathcal{F}^{(r)}_{10}$ and $\mathcal{F}^{(r)}_{13}$ of that table, and their sincere representations $M^{(1,r)}_{8}$, $M^{(1,r)}_{10}$ and $M^{(1,r)}_{13}$, do not appear in the list below, since they correspond to sincere posets of type $\AA$, which are covered by Lemma \ref{1}.

In the tables below, the label in the left-hand column is the name of the module in \cite[Table 8.2]{justina2}; the parameter $k$ occurring in a family $\mathcal{F}^{(r)}_{i,k}$ is the same parameter as in \ref{sincereposetsD}.

\begin{longtable}{|c|c|}
\hline
$M^{(1,r)}_{1}$ &
\begin{adjustbox}{max width=0.60\textwidth}
\begin{tikzcd}[row sep=small, column sep=0.45]
 & \kO \arrow[d]\arrow[dl]\arrow[dr]&&\kO\arrow[dl]\arrow[dr] & &\kO\arrow[dl]\arrow[dr]&&\cdots &&\kO\arrow[dl]\arrow[dr] \\
\kO & \kO & \kO &&\kO &&\kO&\cdots &\kO&&\kO\\
\end{tikzcd}
\end{adjustbox}
\\
\hline
$M^{(2,r)}_{1}$ &
\begin{adjustbox}{max width=0.60\textwidth}
\begin{tikzcd}[row sep=small, column sep=0.45]
 & \kT \arrow[d,"\varepsilon"',pos=0.28]\arrow[dl]\arrow[dr,"d"]&&\kO\arrow[dl]\arrow[dr] & &\kO\arrow[dl]\arrow[dr]&&\cdots &&\kO\arrow[dl]\arrow[dr] \\
\kO & \kO & \kO &&\kO &&\kO&\cdots &\kO&&\kO\\
\end{tikzcd}
\end{adjustbox}
\\
\hline
$M^{(3,r)}_{1}$ &
\begin{adjustbox}{max width=0.60\textwidth}
\begin{tikzcd}[row sep=small, column sep=0.45]
 & \kT \arrow[d,"\varepsilon"',pos=0.28]\arrow[dl]\arrow[dr]&&\kO\arrow[dl,"d"']\arrow[dr] & &\kO\arrow[dl]\arrow[dr]&&\cdots &&\kO\arrow[dl]\arrow[dr] \\
\kO & \kO & \kT &&\kO &&\kO&\cdots &\kO&&\kO\\
\end{tikzcd}
\end{adjustbox}
\\
\hline
$M^{(2k,r)}_{1}$, $2\leq k\leq r-3$ &
\begin{adjustbox}{max width=0.60\textwidth}
\begin{tikzcd}[row sep=small, column sep=0.45]
 & \kT \arrow[d,"\varepsilon"',pos=0.28]\arrow[dl]\arrow[dr]&&\kT\arrow[dl]\arrow[dr] & &\cdots & &\kT\arrow[dl]\arrow[dr,"d"]&&\kO\arrow[dl]\arrow[dr]&&\cdots &&\kO\arrow[dl]\arrow[dr] \\
\kO & \kO & \kT &&\kT &\cdots &\kT &&\kO &&\kO &\cdots &\kO &&\kO\\
\end{tikzcd}
\end{adjustbox}
\\
\hline
$M^{(2k+1,r)}_{1}$, $2\leq k\leq r-4$ &
\begin{adjustbox}{max width=0.60\textwidth}
\begin{tikzcd}[row sep=small, column sep=0.45]
 & \kT \arrow[d,"\varepsilon"',pos=0.28]\arrow[dl]\arrow[dr]&&\kT\arrow[dl]\arrow[dr] & &\cdots & &\kT\arrow[dl]\arrow[dr]&&\kO\arrow[dl,"d"']\arrow[dr]&&\cdots &&\kO\arrow[dl]\arrow[dr] \\
\kO & \kO & \kT &&\kT &\cdots &\kT &&\kT &&\kO &\cdots &\kO &&\kO\\
\end{tikzcd}
\end{adjustbox}
\\
\hline
$M^{(2r-5,r)}_{1}$ &
\begin{adjustbox}{max width=0.60\textwidth}
\begin{tikzcd}[row sep=small, column sep=0.45]
 & \kT \arrow[d,"\varepsilon"',pos=0.28]\arrow[dl]\arrow[dr]&&\kT\arrow[dl]\arrow[dr] & &\cdots & &\kT\arrow[dl]\arrow[dr]&&\kO\arrow[dl,"d"']\arrow[dr] \\
\kO & \kO & \kT &&\kT &\cdots &\kT &&\kT &&\kO\\
\end{tikzcd}
\end{adjustbox}
\\
\hline
$M^{(2r-4,r)}_{1}$ &
\begin{adjustbox}{max width=0.60\textwidth}
\begin{tikzcd}[row sep=small, column sep=0.45]
 & \kT \arrow[d,"\varepsilon"',pos=0.28]\arrow[dl]\arrow[dr]&&\kT\arrow[dl]\arrow[dr] & &\cdots & &\kT\arrow[dl]\arrow[dr]&&\kT\arrow[dl]\arrow[dr,"d"] \\
\kO & \kO & \kT &&\kT &\cdots &\kT &&\kT &&\kO\\
\end{tikzcd}
\end{adjustbox}
\\
\hline
$M^{(1,r)}_{2}$ &
\begin{adjustbox}{max width=0.60\textwidth}
\begin{tikzcd}[row sep=small, column sep=0.45]
& & \kO \arrow[dl]\arrow[dr]&&\kO\arrow[dl]\arrow[dr] & &\kO\arrow[dl]\arrow[dr]&&\cdots &&\kO\arrow[dl]\arrow[dr] \\
&\kT\arrow[dl,"d"']\arrow[dr]&  & \kO &&\kO &&\kO&\cdots &\kO&&\kO\\
\kO && \kO\\
\end{tikzcd}
\end{adjustbox}
\\
\hline
$M^{(1,r)}_{3,k}$ &
\begin{adjustbox}{max width=0.60\textwidth}
\begin{tikzcd}[row sep=small, column sep=0.45]
&&&&&&&&&& \kO\arrow[dl,"d"']\arrow[dr]&& \kO \arrow[dl]\arrow[dr]&&\cdots && \kO\arrow[dl]\arrow[dr]\\
 & \kT \arrow[d,"\varepsilon"',pos=0.28]\arrow[dl]\arrow[dr]&&\kT\arrow[dl]\arrow[dr]& &
 \cdots && \kT\arrow[dr] \arrow[dl]&& \kT\arrow[dl]&&\kO && \kO &\cdots &\kO && \kO \\
\kO & \kO & \kT &&\kT &\cdots &\kT &&\kT \\
\end{tikzcd}
\end{adjustbox}
\\
\hline
$M^{(1,r)}_{4}$ &
\begin{adjustbox}{max width=0.60\textwidth}
\begin{tikzcd}[row sep=small, column sep=0.45]
 & \kO \arrow[d]\arrow[dl]\arrow[dr]&&\kO\arrow[dl]\arrow[dr] & &\kO\arrow[dl]\arrow[dr]&&\cdots &&\kO\arrow[dl]\arrow[dr]& \kO\arrow[d] \\
\kO & \kO & \kO &&\kO &&\kO&\cdots &\kO&&\kO\\
\end{tikzcd}
\end{adjustbox}
\\
\hline
$M^{(2,r)}_{4}$ &
\begin{adjustbox}{max width=0.60\textwidth}
\begin{tikzcd}[row sep=small, column sep=0.45]
 & \kT \arrow[d,"\varepsilon"',pos=0.28]\arrow[dl]\arrow[dr,"d"]&&\kO\arrow[dl]\arrow[dr] & &\kO\arrow[dl]\arrow[dr]&&\cdots &&\kO\arrow[dl]\arrow[dr]& \kO\arrow[d] \\
\kO & \kO & \kO &&\kO &&\kO&\cdots &\kO&&\kO\\
\end{tikzcd}
\end{adjustbox}
\\
\hline
$M^{(3,r)}_{4}$ &
\begin{adjustbox}{max width=0.60\textwidth}
\begin{tikzcd}[row sep=small, column sep=0.45]
 & \kT \arrow[d,"\varepsilon"',pos=0.28]\arrow[dl]\arrow[dr]&&\kO\arrow[dl,"d"']\arrow[dr] & &\kO\arrow[dl]\arrow[dr]&&\cdots &&\kO\arrow[dl]\arrow[dr]& \kO\arrow[d] \\
\kO & \kO & \kT &&\kO &&\kO&\cdots &\kO&&\kO\\
\end{tikzcd}
\end{adjustbox}
\\
\hline
$M^{(2k,r)}_{4}$, $2\leq k\leq r-3$ &
\begin{adjustbox}{max width=0.60\textwidth}
\begin{tikzcd}[row sep=small, column sep=0.45]
 & \kT \arrow[d,"\varepsilon"',pos=0.28]\arrow[dl]\arrow[dr]&&\kT\arrow[dl]\arrow[dr] & &\cdots & &\kT\arrow[dl]\arrow[dr,"d"]&&\kO\arrow[dl]\arrow[dr]&&\cdots &&\kO\arrow[dl]\arrow[dr]& \kO\arrow[d] \\
\kO & \kO & \kT &&\kT &\cdots &\kT &&\kO &&\kO &\cdots &\kO &&\kO\\
\end{tikzcd}
\end{adjustbox}
\\
\hline
$M^{(2k+1,r)}_{4}$, $2\leq k\leq r-3$ &
\begin{adjustbox}{max width=0.60\textwidth}
\begin{tikzcd}[row sep=small, column sep=0.45]
 & \kT \arrow[d,"\varepsilon"',pos=0.28]\arrow[dl]\arrow[dr]&&\kT\arrow[dl]\arrow[dr] & &\cdots & &\kT\arrow[dl]\arrow[dr]&&\kO\arrow[dl,"d"']\arrow[dr]&&\cdots &&\kO\arrow[dl]\arrow[dr]& \kO\arrow[d] \\
\kO & \kO & \kT &&\kT &\cdots &\kT &&\kT &&\kO &\cdots &\kO &&\kO\\
\end{tikzcd}
\end{adjustbox}
\\
\hline
$M^{(2r-4,r)}_{4}$ &
\begin{adjustbox}{max width=0.60\textwidth}
\begin{tikzcd}[row sep=small, column sep=0.45]
 & \kT \arrow[d,"\varepsilon"',pos=0.28]\arrow[dl]\arrow[dr]&&\kT\arrow[dl]\arrow[dr] & &\cdots & &\kT\arrow[dl]\arrow[dr]&&\kT\arrow[dl]\arrow[dr,"d"]& \kO\arrow[d] \\
\kO & \kO & \kT &&\kT &\cdots &\kT &&\kT &&\kO\\
\end{tikzcd}
\end{adjustbox}
\\
\hline
$M^{(2r-3,r)}_{4}$ &
\begin{adjustbox}{max width=0.60\textwidth}
\begin{tikzcd}[row sep=small, column sep=0.45]
 & \kT \arrow[d,"\varepsilon"',pos=0.28]\arrow[dl]\arrow[dr]&&\kT\arrow[dl]\arrow[dr] & &\cdots & &\kT\arrow[dl]\arrow[dr]&&\kT\arrow[dl]\arrow[dr]& \kO\arrow[d,"d"] \\
\kO & \kO & \kT &&\kT &\cdots &\kT &&\kT &&\kT\\
\end{tikzcd}
\end{adjustbox}
\\
\hline
$M^{(1,r)}_{5}$ &
\begin{adjustbox}{max width=0.60\textwidth}
\begin{tikzcd}[row sep=small, column sep=0.45]
& & \kO \arrow[dl]\arrow[dr]&&\kO\arrow[dl]\arrow[dr] & &\kO\arrow[dl]\arrow[dr]&&\cdots &&\kO\arrow[dl]\arrow[dr]& \kO\arrow[d] \\
&\kT\arrow[dl,"d"']\arrow[dr]&  & \kO &&\kO &&\kO&\cdots &\kO&&\kO\\
\kO && \kO\\
\end{tikzcd}
\end{adjustbox}
\\
\hline
$M^{(1,r)}_{6,k}$ &
\begin{adjustbox}{max width=0.60\textwidth}
\begin{tikzcd}[row sep=small, column sep=0.45]
&&&&&&&&&& \kO\arrow[dl,"d"']\arrow[dr]&& \kO \arrow[dl]\arrow[dr]&&\cdots && \kO\arrow[dl]\arrow[dr] & \kO\arrow[d]\\
 & \kT \arrow[d,"\varepsilon"',pos=0.28]\arrow[dl]\arrow[dr]&&\kT\arrow[dl]\arrow[dr]& &
 \cdots && \kT\arrow[dr] \arrow[dl]&& \kT\arrow[dl]&&\kO && \kO &\cdots &\kO && \kO \\
\kO & \kO & \kT &&\kT &\cdots &\kT &&\kT \\
\end{tikzcd}
\end{adjustbox}
\\
\hline
$M^{(1,r)}_{7}$ &
\begin{adjustbox}{max width=0.60\textwidth}
\begin{tikzcd}[row sep=small, column sep=0.45]
&&&&&&&&&& \kO\arrow[d,"d"]\\
 & \kT \arrow[d,"\varepsilon"',pos=0.28]\arrow[dl]\arrow[dr]&&\kT\arrow[dl]\arrow[dr] & &\cdots & &\kT\arrow[dl]\arrow[dr]&&\kT\arrow[dl]\arrow[dr]& \kT\arrow[d] \\
\kO & \kO & \kT &&\kT &\cdots &\kT &&\kT &&\kT\\
\end{tikzcd}
\end{adjustbox}
\\
\hline
$M^{(1,r)}_{9}$ &
\begin{adjustbox}{max width=0.60\textwidth}
\begin{tikzcd}[row sep=small, column sep=0.45]
& \kO\arrow[d]\\
 & \kT \arrow[dl]\arrow[dr,"d"]&&\kO\arrow[dl]\arrow[dr] & &\kO\arrow[dl]\arrow[dr]&&\cdots &&\kO\arrow[dl]\arrow[dr] \\
\kO & & \kO &&\kO &&\kO&\cdots &\kO&&\kO\\
\end{tikzcd}
\end{adjustbox}
\\
\hline
$M^{(2,r)}_{9}$ &
\begin{adjustbox}{max width=0.60\textwidth}
\begin{tikzcd}[row sep=small, column sep=0.45]
& \kO\arrow[d]\\
 & \kT \arrow[dl]\arrow[dr]&&\kO\arrow[dl,"d"']\arrow[dr] & &\kO\arrow[dl]\arrow[dr]&&\cdots &&\kO\arrow[dl]\arrow[dr] \\
\kO & & \kT &&\kO &&\kO&\cdots &\kO&&\kO\\
\end{tikzcd}
\end{adjustbox}
\\
\hline
$M^{(2k+1,r)}_{9}$, $1\leq k\leq r-3$ &
\begin{adjustbox}{max width=0.60\textwidth}
\begin{tikzcd}[row sep=small, column sep=0.45]
& \kO\arrow[d]\\
 & \kT \arrow[dl]\arrow[dr]&&\kT\arrow[dl]\arrow[dr] & &\cdots & &\kT\arrow[dl]\arrow[dr,"d"]&&\kO\arrow[dl]\arrow[dr]&&\cdots &&\kO\arrow[dl]\arrow[dr] \\
\kO & & \kT &&\kT &\cdots &\kT &&\kO &&\kO &\cdots &\kO &&\kO\\
\end{tikzcd}
\end{adjustbox}
\\
\hline
$M^{(2k,r)}_{9}$, $2\leq k\leq r-3$ &
\begin{adjustbox}{max width=0.60\textwidth}
\begin{tikzcd}[row sep=small, column sep=0.45]
& \kO\arrow[d]\\
 & \kT \arrow[dl]\arrow[dr]&&\kT\arrow[dl]\arrow[dr] & &\cdots & &\kT\arrow[dl]\arrow[dr]&&\kO\arrow[dl,"d"']\arrow[dr]&&\cdots &&\kO\arrow[dl]\arrow[dr] \\
\kO & & \kT &&\kT &\cdots &\kT &&\kT &&\kO &\cdots &\kO &&\kO\\
\end{tikzcd}
\end{adjustbox}
\\
\hline
$M^{(2r-4,r)}_{9}$ &
\begin{adjustbox}{max width=0.60\textwidth}
\begin{tikzcd}[row sep=small, column sep=0.45]
& \kO\arrow[d]\\
 & \kT \arrow[dl]\arrow[dr]&&\kT\arrow[dl]\arrow[dr] & &\cdots & &\kT\arrow[dl]\arrow[dr]&&\kO\arrow[dl,"d"']\arrow[dr] \\
\kO & & \kT &&\kT &\cdots &\kT &&\kT &&\kO\\
\end{tikzcd}
\end{adjustbox}
\\
\hline
$M^{(2r-3,r)}_{9}$ &
\begin{adjustbox}{max width=0.60\textwidth}
\begin{tikzcd}[row sep=small, column sep=0.45]
& \kO\arrow[d]\\
 & \kT \arrow[dl]\arrow[dr]&&\kT\arrow[dl]\arrow[dr] & &\cdots & &\kT\arrow[dl]\arrow[dr]&&\kT\arrow[dl]\arrow[dr,"d"] \\
\kO & & \kT &&\kT &\cdots &\kT &&\kT &&\kO\\
\end{tikzcd}
\end{adjustbox}
\\
\hline
$M^{(1,r)}_{11,k}$ &
\begin{adjustbox}{max width=0.60\textwidth}
\begin{tikzcd}[row sep=small, column sep=0.45]
&\kO\arrow[d]&&&&&&&&& \kO\arrow[dl,"d"']\arrow[dr]&& \kO \arrow[dl]\arrow[dr]&&\cdots && \kO\arrow[dl]\arrow[dr]\\
 & \kT \arrow[dl]\arrow[dr]&&\kT\arrow[dl]\arrow[dr]& &
 \cdots && \kT\arrow[dr] \arrow[dl]&& \kT\arrow[dl]&&\kO && \kO &\cdots &\kO && \kO \\
\kO &  & \kT &&\kT &\cdots &\kT &&\kT \\
\end{tikzcd}
\end{adjustbox}
\\
\hline
$M^{(1,r)}_{12}$ &
\begin{adjustbox}{max width=0.60\textwidth}
\begin{tikzcd}[row sep=small, column sep=0.45]
& \kO\arrow[d]\\
 & \kT \arrow[dl]\arrow[dr,"d"]&&\kO\arrow[dl]\arrow[dr] & &\kO\arrow[dl]\arrow[dr]&&\cdots &&\kO\arrow[dl]\arrow[dr]& \kO\arrow[d] \\
\kO & & \kO &&\kO &&\kO&\cdots &\kO&&\kO\\
\end{tikzcd}
\end{adjustbox}
\\
\hline
$M^{(2,r)}_{12}$ &
\begin{adjustbox}{max width=0.60\textwidth}
\begin{tikzcd}[row sep=small, column sep=0.45]
& \kO\arrow[d]\\
 & \kT \arrow[dl]\arrow[dr]&&\kO\arrow[dl,"d"']\arrow[dr] & &\kO\arrow[dl]\arrow[dr]&&\cdots &&\kO\arrow[dl]\arrow[dr]& \kO\arrow[d] \\
\kO & & \kT &&\kO &&\kO&\cdots &\kO&&\kO\\
\end{tikzcd}
\end{adjustbox}
\\
\hline
$M^{(2k+1,r)}_{12}$, $1\leq k\leq r-3$ &
\begin{adjustbox}{max width=0.60\textwidth}
\begin{tikzcd}[row sep=small, column sep=0.45]
& \kO\arrow[d]\\
 & \kT \arrow[dl]\arrow[dr]&&\kT\arrow[dl]\arrow[dr] & &\cdots & &\kT\arrow[dl]\arrow[dr,"d"]&&\kO\arrow[dl]\arrow[dr]&&\cdots &&\kO\arrow[dl]\arrow[dr]& \kO\arrow[d] \\
\kO & & \kT &&\kT &\cdots &\kT &&\kO &&\kO &\cdots &\kO &&\kO\\
\end{tikzcd}
\end{adjustbox}
\\
\hline
$M^{(2k,r)}_{12}$, $2\leq k\leq r-2$ &
\begin{adjustbox}{max width=0.60\textwidth}
\begin{tikzcd}[row sep=small, column sep=0.45]
& \kO\arrow[d]\\
 & \kT \arrow[dl]\arrow[dr]&&\kT\arrow[dl]\arrow[dr] & &\cdots & &\kT\arrow[dl]\arrow[dr]&&\kO\arrow[dl,"d"']\arrow[dr]&&\cdots &&\kO\arrow[dl]\arrow[dr]& \kO\arrow[d] \\
\kO & & \kT &&\kT &\cdots &\kT &&\kT &&\kO &\cdots &\kO &&\kO\\
\end{tikzcd}
\end{adjustbox}
\\
\hline
$M^{(2r-3,r)}_{12}$ &
\begin{adjustbox}{max width=0.60\textwidth}
\begin{tikzcd}[row sep=small, column sep=0.45]
& \kO\arrow[d]\\
 & \kT \arrow[dl]\arrow[dr]&&\kT\arrow[dl]\arrow[dr] & &\cdots & &\kT\arrow[dl]\arrow[dr]&&\kT\arrow[dl]\arrow[dr,"d"]& \kO\arrow[d] \\
\kO & & \kT &&\kT &\cdots &\kT &&\kT &&\kO\\
\end{tikzcd}
\end{adjustbox}
\\
\hline
$M^{(2r-2,r)}_{12}$ &
\begin{adjustbox}{max width=0.60\textwidth}
\begin{tikzcd}[row sep=small, column sep=0.45]
& \kO\arrow[d]\\
 & \kT \arrow[dl]\arrow[dr]&&\kT\arrow[dl]\arrow[dr] & &\cdots & &\kT\arrow[dl]\arrow[dr]&&\kT\arrow[dl]\arrow[dr]& \kO\arrow[d,"d"] \\
\kO & & \kT &&\kT &\cdots &\kT &&\kT &&\kT\\
\end{tikzcd}
\end{adjustbox}
\\
\hline
$M^{(1,r)}_{14}$ &
\begin{adjustbox}{max width=0.60\textwidth}
\begin{tikzcd}[row sep=small, column sep=0.45]
 & \kO \arrow[dl]\arrow[dr]&&\kO\arrow[dl]\arrow[dr] & &\kO\arrow[dl]\arrow[dr]&&\cdots &&\kO\arrow[dl]\arrow[dr] & \kO\arrow[d]& \kO\arrow[dl] \\
\kO &  & \kO &&\kO &&\kO&\cdots &\kO&&\kO\\
\end{tikzcd}
\end{adjustbox}
\\
\hline
$M^{(2,r)}_{14}$ &
\begin{adjustbox}{max width=0.60\textwidth}
\begin{tikzcd}[row sep=small, column sep=0.45]
 & \kO \arrow[dl]\arrow[dr]&&\kO\arrow[dl]\arrow[dr] & &\kO\arrow[dl]\arrow[dr]&&\cdots &&\kO\arrow[dl]\arrow[dr,"d"',pos=0.6] & \kO\arrow[d,"\varepsilon"',pos=0.28]& \kO\arrow[dl] \\
\kO &  & \kO &&\kO &&\kO&\cdots &\kO&&\kT\\
\end{tikzcd}
\end{adjustbox}
\\
\hline
$M^{(3,r)}_{14}$ &
\begin{adjustbox}{max width=0.60\textwidth}
\begin{tikzcd}[row sep=small, column sep=0.45]
 & \kO \arrow[dl]\arrow[dr]&&\kO\arrow[dl]\arrow[dr] & &\kO\arrow[dl]\arrow[dr]&&\cdots &&\kT\arrow[dl,"d"']\arrow[dr] & \kO\arrow[d,"\varepsilon"',pos=0.28]& \kO\arrow[dl] \\
\kO &  & \kO &&\kO &&\kO&\cdots &\kO&&\kT\\
\end{tikzcd}
\end{adjustbox}
\\
\hline
$M^{(2k,r)}_{14}$, $2\leq k\leq r-3$ &
\begin{adjustbox}{max width=0.60\textwidth}
\begin{tikzcd}[row sep=small, column sep=0.45]
 & \kO\arrow[dl]\arrow[dr]&&\kO\arrow[dl]\arrow[dr] & &\cdots & &\kO\arrow[dl]\arrow[dr,"d"]&&\kT\arrow[dl]\arrow[dr]&&\cdots &&\kT\arrow[dl]\arrow[dr]& \kO\arrow[d,"\varepsilon"',pos=0.28]& \kO\arrow[dl] \\
\kO &  & \kO &&\kO &\cdots &\kO &&\kT &&\kT &\cdots &\kT &&\kT\\
\end{tikzcd}
\end{adjustbox}
\\
\hline
$M^{(2k+1,r)}_{14}$, $2\leq k\leq r-3$ &
\begin{adjustbox}{max width=0.60\textwidth}
\begin{tikzcd}[row sep=small, column sep=0.45]
 & \kO\arrow[dl]\arrow[dr]&&\kO\arrow[dl]\arrow[dr] & &\cdots & &\kO\arrow[dl]\arrow[dr]&&\kT\arrow[dl,"d"']\arrow[dr]&&\cdots &&\kT\arrow[dl]\arrow[dr]& \kO\arrow[d,"\varepsilon"',pos=0.28]& \kO\arrow[dl] \\
\kO &  & \kO &&\kO &\cdots &\kO &&\kO &&\kT &\cdots &\kT &&\kT\\
\end{tikzcd}
\end{adjustbox}
\\
\hline
$M^{(2r-4,r)}_{14}$ &
\begin{adjustbox}{max width=0.60\textwidth}
\begin{tikzcd}[row sep=small, column sep=0.45]
 & \kO\arrow[dl]\arrow[dr,"d"]&&\kT\arrow[dl]\arrow[dr] & &\cdots & &\kT\arrow[dl]\arrow[dr]& \kO\arrow[d,"\varepsilon"',pos=0.28]& \kO\arrow[dl] \\
\kO &  & \kT &&\kT &\cdots &\kT &&\kT\\
\end{tikzcd}
\end{adjustbox}
\\
\hline
$M^{(2r-3,r)}_{14}$ &
\begin{adjustbox}{max width=0.60\textwidth}
\begin{tikzcd}[row sep=small, column sep=0.45]
 & \kT\arrow[dl,"d"']\arrow[dr]&&\kT\arrow[dl]\arrow[dr] & &\cdots & &\kT\arrow[dl]\arrow[dr]& \kO\arrow[d,"\varepsilon"',pos=0.28]& \kO\arrow[dl] \\
\kO &  & \kT &&\kT &\cdots &\kT &&\kT\\
\end{tikzcd}
\end{adjustbox}
\\
\hline
$M^{(1,r)}_{15,k}$ &
\begin{adjustbox}{max width=0.60\textwidth}
\begin{tikzcd}[row sep=small, column sep=0.45]
&\kO\arrow[d]&&&&&&&&& \kO\arrow[dl,"d"']\arrow[dr]&& \kO \arrow[dl]\arrow[dr]&&\cdots && \kO\arrow[dl]\arrow[dr] & \kO\arrow[d]\\
 & \kT \arrow[dl]\arrow[dr]&&\kT\arrow[dl]\arrow[dr]& &
 \cdots && \kT\arrow[dr] \arrow[dl]&& \kT\arrow[dl]&&\kO && \kO &\cdots &\kO && \kO \\
\kO &  & \kT &&\kT &\cdots &\kT &&\kT \\
\end{tikzcd}
\end{adjustbox}
\\
\hline
$M^{(1,r)}_{16}$ &
\begin{adjustbox}{max width=0.60\textwidth}
\begin{tikzcd}[row sep=small, column sep=0.45]
\kO\arrow[d]& \kO \arrow[dl]\arrow[dr]&&\kO\arrow[dl]\arrow[dr] & &\kO\arrow[dl]\arrow[dr]&&\cdots &&\kO\arrow[dl]\arrow[dr] & \kO\arrow[d]& \kO\arrow[dl] \\
\kO &  & \kO &&\kO &&\kO&\cdots &\kO&&\kO\\
\end{tikzcd}
\end{adjustbox}
\\
\hline
$M^{(2,r)}_{16}$ &
\begin{adjustbox}{max width=0.60\textwidth}
\begin{tikzcd}[row sep=small, column sep=0.45]
\kO\arrow[d]& \kO \arrow[dl]\arrow[dr]&&\kO\arrow[dl]\arrow[dr] & &\kO\arrow[dl]\arrow[dr]&&\cdots &&\kO\arrow[dl]\arrow[dr,"d"',pos=0.6] & \kO\arrow[d,"\varepsilon"',pos=0.28]& \kO\arrow[dl] \\
\kO &  & \kO &&\kO &&\kO&\cdots &\kO&&\kT\\
\end{tikzcd}
\end{adjustbox}
\\
\hline
$M^{(3,r)}_{16}$ &
\begin{adjustbox}{max width=0.60\textwidth}
\begin{tikzcd}[row sep=small, column sep=0.45]
\kO\arrow[d]& \kO \arrow[dl]\arrow[dr]&&\kO\arrow[dl]\arrow[dr] & &\kO\arrow[dl]\arrow[dr]&&\cdots &&\kT\arrow[dl,"d"']\arrow[dr] & \kO\arrow[d,"\varepsilon"',pos=0.28]& \kO\arrow[dl] \\
\kO &  & \kO &&\kO &&\kO&\cdots &\kO&&\kT\\
\end{tikzcd}
\end{adjustbox}
\\
\hline
$M^{(2k+2,r)}_{16}$, $1\leq k\leq r-2$ &
\begin{adjustbox}{max width=0.60\textwidth}
\begin{tikzcd}[row sep=small, column sep=0.45]
\kO\arrow[d]& \kO\arrow[dl]\arrow[dr]&&\kO\arrow[dl]\arrow[dr] & &\cdots & &\kO\arrow[dl]\arrow[dr,"d"]&&\kT\arrow[dl]\arrow[dr]&&\cdots &&\kT\arrow[dl]\arrow[dr]& \kO\arrow[d,"\varepsilon"',pos=0.28]& \kO\arrow[dl] \\
\kO &  & \kO &&\kO &\cdots &\kO &&\kT &&\kT &\cdots &\kT &&\kT\\
\end{tikzcd}
\end{adjustbox}
\\
\hline
$M^{(2k+1,r)}_{16}$, $2\leq k\leq r-2$ &
\begin{adjustbox}{max width=0.60\textwidth}
\begin{tikzcd}[row sep=small, column sep=0.45]
\kO\arrow[d]& \kO\arrow[dl]\arrow[dr]&&\kO\arrow[dl]\arrow[dr] & &\cdots & &\kO\arrow[dl]\arrow[dr]&&\kT\arrow[dl,"d"']\arrow[dr]&&\cdots &&\kT\arrow[dl]\arrow[dr]& \kO\arrow[d,"\varepsilon"',pos=0.28]& \kO\arrow[dl] \\
\kO &  & \kO &&\kO &\cdots &\kO &&\kO &&\kT &\cdots &\kT &&\kT\\
\end{tikzcd}
\end{adjustbox}
\\
\hline
$M^{(2r-1,r)}_{16}$ &
\begin{adjustbox}{max width=0.60\textwidth}
\begin{tikzcd}[row sep=small, column sep=0.45]
\kO\arrow[d]& \kT\arrow[dl,"d"']\arrow[dr]&&\kT\arrow[dl]\arrow[dr] & &\cdots & &\kT\arrow[dl]\arrow[dr]& \kO\arrow[d,"\varepsilon"',pos=0.28]& \kO\arrow[dl] \\
\kO &  & \kT &&\kT &\cdots &\kT &&\kT\\
\end{tikzcd}
\end{adjustbox}
\\
\hline
$M^{(2r,r)}_{16}$ &
\begin{adjustbox}{max width=0.60\textwidth}
\begin{tikzcd}[row sep=small, column sep=0.45]
\kO\arrow[d,"d"']& \kT\arrow[dl]\arrow[dr]&&\kT\arrow[dl]\arrow[dr] & &\cdots & &\kT\arrow[dl]\arrow[dr]& \kO\arrow[d,"\varepsilon"',pos=0.28]& \kO\arrow[dl] \\
\kT &  & \kT &&\kT &\cdots &\kT &&\kT\\
\end{tikzcd}
\end{adjustbox}
\\
\hline
$M^{(1,r)}_{17,k}$ &
\begin{adjustbox}{max width=0.60\textwidth}
\begin{tikzcd}[row sep=small, column sep=0.45]
&&&&&&&&&&& \kO\arrow[dl,"d"']\arrow[dr]&& \kO \arrow[dl]\arrow[dr]&&\cdots && \kO\arrow[dl]\arrow[dr]\\
 \kO\arrow[dr]& \kO\arrow[d,"\varepsilon"',pos=0.28] &\kT\arrow[dl]\arrow[dr]& &\kT\arrow[dl]\arrow[dr]
 &&\cdots && \kT\arrow[dr] \arrow[dl]&& \kT\arrow[dl]&&\kO && \kO &\cdots &\kO && \kO \\
&\kT &  & \kT &&\kT &\cdots &\kT &&\kT \\
\end{tikzcd}
\end{adjustbox}
\\
\hline
$M^{(1,r)}_{18}$ &
\begin{adjustbox}{max width=0.60\textwidth}
\begin{tikzcd}[row sep=small, column sep=0.45]
\kO\arrow[d,"d"']\\
\kT\arrow[d]& \kT\arrow[dl]\arrow[dr]&&\kT\arrow[dl]\arrow[dr] & &\cdots & &\kT\arrow[dl]\arrow[dr]& \kO\arrow[d,"\varepsilon"',pos=0.28]& \kO\arrow[dl] \\
\kT &  & \kT &&\kT &\cdots &\kT &&\kT\\
\end{tikzcd}
\end{adjustbox}
\\
\hline
$M^{(1,r)}_{19,k}$ &
\begin{adjustbox}{max width=0.60\textwidth}
\begin{tikzcd}[row sep=small, column sep=0.45]
&&&&&&&&&&& \kO\arrow[dl,"d"']\arrow[dr]&& \kO \arrow[dl]\arrow[dr]&&\cdots && \kO\arrow[dl]\arrow[dr] & \kO\arrow[d]\\
 \kO\arrow[dr]& \kO\arrow[d,"\varepsilon"',pos=0.28] &\kT\arrow[dl]\arrow[dr]& &\kT\arrow[dl]\arrow[dr]
 &&\cdots && \kT\arrow[dr] \arrow[dl]&& \kT\arrow[dl]&&\kO && \kO &\cdots &\kO && \kO \\
&\kT &  & \kT &&\kT &\cdots &\kT &&\kT \\
\end{tikzcd}
\end{adjustbox}
\\
\hline
$M^{(1,r)}_{20}$ &
\begin{adjustbox}{max width=0.60\textwidth}
\begin{tikzcd}[row sep=small, column sep=0.45]
& \kO\arrow[d]&&&&&&&&& \kO\arrow[d,"d"]\\
 & \kT \arrow[dl]\arrow[dr]&&\kT\arrow[dl]\arrow[dr] & &\cdots & &\kT\arrow[dl]\arrow[dr]&&\kT\arrow[dl]\arrow[dr]& \kT\arrow[d] \\
\kO & & \kT &&\kT &\cdots &\kT &&\kT &&\kT\\
\end{tikzcd}
\end{adjustbox}
\\
\hline
\end{longtable}

\begin{remark}\label{sincereDmodules}
Let $S$ be a sincere poset of type $\DD$, with associated quiver $Q_S$, branch vertex $z$, and leaves $u$ and $v$ at $z$. For every sincere socle-projective $\Bbbk S$-module $L$ (tabulated in \ref{sincererepsD}):
\begin{itemize}
\item[(a)] $\dim_{\Bbbk}L_u=\dim_{\Bbbk}L_v=1$, and if $\dim_{\Bbbk}L_x=2$ for some $x\in S$, then $\dim_{\Bbbk}L_z=2$.
\item[(b)] The vertices $x$ with $\dim_{\Bbbk}L_x=2$ form a connected subgraph of $Q_S$ of type $\AA$ containing $z$, and the vertices $x\neq u,v$ with $\dim_{\Bbbk}L_x=1$ also form a subgraph of $Q_S$ of type $\AA$.
\end{itemize}
This is observed directly in the tables of \ref{sincererepsD}, and it coincides with the general behavior of dimension vectors of indecomposable representations of a Dynkin quiver of type $\DD$ described in Remark \ref{applications}.
\end{remark}

\fi

\section*{Acknowledgements}
The first author was supported by Universidad Distrital Francisco Jos\'e de Caldas, Project PPNFF-2025-2060. The second author was supported by NSF Grants DMS-2054561 and DMS-2348909. The third author was supported by Universidad Pedag\'ogica y Tecnol\'ogica de Colombia, SGI 4196.

\bibliographystyle{elsarticle-num}
\bibliography{references}

\end{document}